\documentclass[a4, 11pt]{article}
\usepackage{empheq}
\usepackage[framemethod=tikz]{mdframed}
\usepackage{amssymb, amsfonts, amsmath, amsthm, bbm}
\usepackage{lipsum}
\usepackage{tcolorbox}
\usepackage{fourier-orns}
\usepackage{enumitem}
\usepackage[T1]{fontenc}
\usepackage{lmodern}
\usepackage{stmaryrd}
\usetikzlibrary{fit,shapes}
\usepackage{xlop}
\usepackage{enumitem}
\usepackage[toc,page]{appendix}
\usepackage{soul}
\usepackage{hyperref}
\usepackage{subfig}
\usepackage{graphicx}

\usepackage[applemac]{inputenc}

\numberwithin{equation}{section}

\theoremstyle{plain}
\newtheorem{theorem}{Theorem}[section]
\newtheorem{corollary}[theorem]{Corollary}

\newtheorem{lemma}[theorem]{Lemma}
\newtheorem{proposition}[theorem]{Proposition}

\newtheorem{definition}[theorem]{Definition}

\theoremstyle{definition}

\theoremstyle{remark}
\newtheorem*{remark}{\textbf{Remark}}

\usepackage[english]{babel}
\newmdenv[innerlinewidth=0.5pt, roundcorner=4pt,innerleftmargin=6pt,
innerrightmargin=6pt,innertopmargin=10pt,innerbottommargin=10pt,backgroundcolor=gray!21]{mybox}

\date{}
\title{Fluid limit and gelation in the frozen Erd\H{o}s-R\'enyi random graph}

\author{B\'en\'edicte Haas \thanks{Universit\'e Sorbonne Paris Nord, LAGA, CNRS (UMR  7539) 93430 Villetaneuse, France \newline \hspace*{0.5cm} E-mail: haas@math.univ-paris13.fr}  \quad \& \hspace{0.2cm}  Vincent Viau \thanks{Universit\'e Sorbonne Paris Nord, LAGA, CNRS (UMR  7539) 93430 Villetaneuse, France \newline \hspace*{0.5cm} E-mail: viau@math.univ-paris13.fr}}

\begin{document}

\let\oldproofname=\proofname
\renewcommand{\proofname}{\rm\bf{\oldproofname}}

\maketitle

\begin{abstract}
The frozen Erd\H{o}s-R\'enyi random graph is a variant of the  standard dynamical Erd\H{o}s-R\'enyi random graph that prevents the creation of the giant component by freezing the evolution of connected components with a unique cycle. The formation of multicyclic components is forbidden, and the growth of components with a unique cycle is slowed down, depending on a parameter $p \in [0,1]$ that quantifies the slowdown. At the time when all connected components of the graph have a (necessary unique) cycle, the graph is entirely frozen and the process stops.  In this paper we study the fluid limit of the main statistics of this process, that is their functional convergence as the number of vertices of the graph becomes large and after a proper rescaling, to the solution of a system of differential equations. Our proofs are based on an adaption of Wormald's differential equation method. We also obtain, as a main application, a precise description of the asymptotic behavior of the first time when the graph is entirely frozen. 
\end{abstract}

\tableofcontents

\section{Introduction and main results}

We study a variant of the standard dynamical Erd\H{o}s-R\'enyi random graph which generates a dynamical random graph with only \emph{simple} connected components. By simple, we mean either a \emph{tree} (the number of edges in the connected component is equal to the number of vertices minus one), or a connected component with a unique cycle, called a \emph{unicycle} (the number of edges in the connected component is equal to the number of vertices). 
In the standard Erd\H{o}s-R\'enyi graph, there is essentially one non-simple connected component: the giant component, which emerges in the so-called supercritical phase, the other non-bounded components being simple with high probability. Our variant model in a sense prevents the creation of the giant component by freezing the evolution of the unicycles. This model was introduced recently by Contat and Curien \cite{ContatCurien23}, motivated by connections with a parking model on a Cayley tree, and then studied by Viau \cite{viau24,viau25} and Krapivsky \cite{krapisvky24} in the physics literature. It is a discrete-time evolving model of graph on $n$ labelled vertices $\{1,\ldots,n\}$ which may be frozen or not frozen.  Its dynamics depends on a parameter $p\in [0,1]$ which slows down the growth of unicycles in the standard Erd\H{o}s-R\'enyi graph and prevents the formation of multicyclic components, that is with more edges than vertices. We denote this model by $$\big(\mathrm{F}_{p,n}(m), m\in \mathbb Z_+\big)$$ and refer to it as the \emph{$p$-frozen model}. Its construction proceeds recursively on $m \in \mathbb Z_{+}$ at follows. Initially, $\mathrm{F}_{p,n}(0)$ is the graph composed of $n$ isolated and \emph{non-frozen} vertices. Then at step $m$, given $\mathrm{F}_{p,n}(m-1)$, one of the $n(n-1)/2$ possible edges is selected uniformly at random and:

\begin{center}
\begin{minipage}{15.5cm}
\begin{enumerate}[topsep=-0cm,leftmargin=0.3cm]
\item[$\bullet$] If the selected edge connects two vertices of trees of $\mathrm{F}_{p,n}(m-1)$ (this may be two vertices of a same tree, or of two different trees), then it is \emph{added} to the graph to form a new connected component (the other connected components of $\mathrm{F}_{p,n}(m-1)$ remain unchanged). If two different trees were involved, this operation produces a new tree and none of its vertices are frozen. Otherwise it produces a unicycle component and we decide that this unicycle and its vertices are \emph{frozen}. This operation gives us a new graph: $\mathrm{F}_{p,n}(m)$.
\item[$\bullet$] If the selected edge connects two vertices of unicycle components of $\mathrm{F}_{p,n}(m-1)$ (possibly the same unicycle), then it is \emph{discarded} and $\mathrm{F}_{p,n}(m)=\mathrm{F}_{p,n}(m-1)$.
\item[$\bullet$] If  the selected edge connects a tree and a unicycle of $\mathrm{F}_{p,n}(m-1)$, then it is added with probability $p$ and discarded with probability $1-p$. If added, the tree is glued on the unicycle to form a new, bigger unicycle, whose vertices are all frozen. This gives $\mathrm{F}_{p,n}(m)$.
\end{enumerate}
\end{minipage}
\end{center}
Recall that the standard Erd\H{o}s-R\'enyi graph evolves similarly by selecting at each step an edge uniformly among the $n(n-1)/2$ possible edges, but then the selected edge is systematically added to the current graph. Throughout the paper we will denote by $(\mathrm{ER}_n(m), m\in \mathbb Z_+)$ a version of this standard model. We emphasize that various other variants of this standard model with different constraints preventing the formation of components, or destroying components, have been studied, see for example \cite{BR00, LMP23, RT09, Ross21} and the references therein, or  \cite{Ald00, BBKK23} for other models of frozen graphs.

Returning to the frozen model, the \emph{forest part} of $\mathrm{F}_{p,n}(m)$ is its subgraph corresponding to the set of trees. The \emph{gel} of $\mathrm{F}_{p,n}(m)$ is the set of all frozen vertices, that is the set of all vertices involved in a unicycle. A vertex of $\mathrm{F}_{p,n}(m)$ is thus either in the forest or in the gel. We will call \emph{total gelation time} the first time when all vertices are frozen: from this time the graph is completely frozen, i.e. it no longer evolves.

This paper addresses two main questions, namely 1) the existence of a \emph{fluid limit}, or \emph{law of large numbers}, for several statistics of the $p$-frozen model $\mathrm{F}_{p,n}$, such as its number of frozen vertices, discarded edges, trees of a given size, etc., and 2) the asymptotic behavior of its total gelation time as well as the distribution and extinction of the number of trees of a given size in the neighborhood of the gelation time. In general, the \emph{size} of a connected component refers to its number of vertices.

These questions have natural counterparts in the standard Erd\H{o}s-Rényi graph, respectively the fluid limit for the size of the largest component and the time needed for the graph to be connected. Both questions have been deeply studied and it is well-known since the initial works of Erd\H{o}s and R\'enyi \cite{ErdosRenyi59,ErdosRenyi60} that the size of the largest component in a graph with $n$ vertices exhibits a phase transition when the number of edges approaches $n/2$. This can be resumed as follows: the largest components of $\mathrm{ER}_n(m)$  are of order $\ln(n)$ when $m/n\sim t<1/2$ (subcritical regime), of order $n^{2/3}$ when $m/n\sim 1/2$ (critical regime) and there is a unique, giant, component of order $n$ when $m/n\sim t>1/2$ (supercritical regime), the others being of order at most $\ln(n)$. Among their numerous results, Erd\H{o}s and Rényi displayed the expression of the fluid limit of the size of the giant component in the supercritical regime, and proved that it is determined by the unique non-null function $g_{\mathrm{ER}}$ verifying the following equation
\begin{equation*}
	g_{\mathrm{ER}}(t)=1-\mathrm{e}^{-2tg_{\mathrm{ER}}(t)},\quad t>1/2.
\end{equation*}  
Since then, the phase transition has been studied extensively. We refer e.g. to \cite{Pittel88} for a study of the subcritical regime, to \cite{aldous97,JansonKnuthLuczakPittel93,luczak90} for the critical regime and the emergence of the giant component, and to \cite{barraezBoucheronVega00,pittel90,rath18,stepanov70} for the supercritical regime and especially for results on the fluctuations around the fluid limit. For an overview on Erd\H{o}s-Rényi random graph, one could also refer to the books \cite{AlonSpencer16,bollobas01,vanderhofstad24} which also address the connectedness of the graph. This last issue has been initially raised by Erd\H{o}s and Rényi \cite{ErdosRenyi59} and mostly investigated in the early papers on the subject \cite{ErdosRenyi59,ErdosRenyi60,stepanov70}. Erd\H{o}s and Rényi proved that the time needed for the graph to be connected, say $A_{\mathrm{ER},n}$, coincides in the limit with the vanishing time of isolated vertices,  which enabled them to show that
\begin{equation}
\label{connec:ER}
	\frac{A_{\mathrm{ER},n}}{n}-\frac{\ln(n)}{2}  \; \underset{n \rightarrow \infty}{\overset{\mathrm{(d)}}\longrightarrow} \; \frac{\mathrm{Gu}}{2},
\end{equation}
where $\mathrm{Gu}$ denotes a standard Gumbel distribution, $\mathbb P(\mathrm{Gu}\leq x)=e^{-e^{-x}}, x \in \mathbb R$.

\bigskip

For the frozen model $\mathrm{F}_{p,n}$ a similar phase transition at $m/n\sim 1/2$ has been demonstrated in \cite{ContatCurien23} for $p=1/2$ and generalized in \cite{viau25} for all $p \in [0,1]$.  This will be recalled later in more detail. The aim of this paper is to understand the effect of the gelation, and the induced slowdowns, in relation to the standard Erd\H{o}s-R\'enyi graph, essentially in the supercritical regime when $m \gg n/2$.

\bigskip

\emph{\textbf{Remarks.}}  1) In the frozen model, when $p=1$, the tree components are systematically added to the gel when the selected edge connects to the gel. There is then an obvious coupling with the standard Erd\H{o}s-R\'enyi  model so that the forest part of the frozen model coincides with that of the standard Erd\H{o}s-R\'enyi (and the gel coincides with the set of vertices in components with at least one cycle, namely the \emph{cyclic components}, in the  standard Erd\H{o}s-R\'enyi). So, several results that we will state below retrieve similar results on the standard Erd\H{o}s-R\'enyi. 

\vspace{-0.1cm}

2) In contrast, the tree components of the frozen model when $p=0$ are never added to the gel when the selected edge connects to  the gel. 

\vspace{-0.1cm}

3) Although there is no obvious monotonicity of the gel size in the parameter $p\in [0,1]$, there is also an obvious coupling, for each $p$,  so that the gel of the $p$-frozen model is included in the set of vertices in cyclic components of the standard Erd\H{o}s-R\'enyi model: the gel process in $\mathrm{F}_{1,n}$ is thus stochastically larger than the gel process in $\mathrm{F}_{p,n}$, whatever $p\in [0,1]$.

\bigskip

We present our main results in the three forthcoming subsections, focussing in this paper on the case 
$$
p \in (0,1].
$$
This is implicit in all statements.
The case $p=0$, where unicycles become unattractive as soon as they are created, shows partially different behavior and requires an adapted approach, although several of our intermediate steps for the implementation of the main results are still valid for $p=0$. We discuss the expected results and open questions on this case  in Section \ref{sec:open}.

\bigskip

\textbf{Notation.} We will use the following notation throughout the paper, for $m \in \mathbb Z_+$:
\begin{enumerate}[topsep=0cm]
\item[-] $G_{p,n}(m)$ is the size of the gel at time $m$, that is the number of frozen vertices 
\item[-] $D_{p,n}(m)$ is the number of discarded edges at time $m$
\item[-] $V_{p,n}(m)$ is the number of vertices in the forest part of $\mathrm{F}_{p,n}(m)$
\item[-] $E_{p,n}(m)$ is the number of edges  in the forest part of $\mathrm{F}_{p,n}(m)$
\item[-] $N_{p,n}^{(k)}(m)$ is the number of trees of size $k$ in $\mathrm{F}_{p,n}(m)$, $k \in \mathbb N$. 
\end{enumerate}

Note the following obvious relations, which we will regularly use to pass from one quantity to the other,
\begin{equation}
\label{def:V_{p,n}E_{p,n}}
V_{p,n}(m)+G_{p,n}(m)=n \qquad \text{and} \qquad  E_{p,n}(m)+G_{p,n}(m)+D_{p,n}(m)=m.
\end{equation}
Of course, for $n$ fixed, the processes $G_{p,n}$ and $D_{p,n}$ are non-decreasing. Note also the trivial bound $G_{p,n}(m)\leq \min(m,n)$.

\bigskip

\emph{\textbf{Remark.}} As usual in dynamical random graph models, it may be easier in some situations to work with a continuous version of the model. In Section \ref{sec:continuous} we introduce a \emph{Poissonized} counterpart of $\mathrm{F}_{p,n}$. This will be useful for studying the total gelation time. In this introduction the main results are stated on the discrete model, their continuous counterparts will be given in the core of the paper.

\bigskip

\emph{\textbf{Remark.}} Krapivsky's paper \cite{krapisvky24} has related interests and was published in the physics literature while we were working on this project. He obtained, via a more intuitive approach, the expression and some properties of the fluid limits of the gel, of the number of trees of size $k\geq 1$ and of the average number of unicycles of size $k\geq 1$, by identifying the differential equations there are solutions to. He also develops heuristics for the total gelation time. Our paper confirms and completes his predictions.

\subsection{Fluid limit of the gel}

We start by defining the function that will describe the fluid limit of the gel (size) $G_{p,n}$. An equivalent definition, as a solution to a differential equation, is given in Section \ref{sec:fluid}.

\begin{definition}
\label{def:gel_function}
We call \emph{gel mass function} the function $g_p:[0,\infty)\rightarrow [0,1)$ which is null on $[0,1/2]$ and defined on $[1/2,\infty)$ as the \emph{inverse} of the function $f_p:[0,1) \rightarrow [1/2,\infty)$ given for $t \in [0,1)$ by
\begin{equation*}
f_p(t)~=~\frac{1}{2}+\frac{t}{2p}\int_{0}^1 \frac{u^{\frac{1}{p}}}{1-tu} \mathrm du~=~\frac{1}{2}\sum_{n=0}^{\infty}\frac{t^n}{1+pn}.
\end{equation*}
\end{definition}

Note that $f_p$ is decreasing in $p$ and so $g_p$ is increasing in $p$. This monotonicity was a priori not obvious since there is no stochastic monotonicity in $p \mapsto G_{p,n}(m)$. However, as already observed, the gel size $G_{p,n}(m)$ is stochastically smaller (whatever $p$) than the total number of vertices involved in cyclic components at time $m$ of the standard Erd\H{o}s-R\'enyi graph, which is distributed as $G_{1,n}(m)$. So the bound $g_p \leq g_1$ was predictable.

We will work in  detail on the function $g_p$ in Section \ref{sec:fluid}, but already emphasize here that $g_p$ is infinitely differentiable on $(1/2,\infty)$, $g_p(1/2)=0$, $g_p'(1/2^+)=2(1+p)$ and $g_p(t)=1-e^{-2pt}+o\left(e^{-2pt}\right)$ as $t \rightarrow \infty$. When $p=1$, the inverse function $f_1$ is particularly simple
$$f_1(s)=\frac{-\ln(1-s)}{2s},\quad s\in [0,1),$$
and we see that the gel mass function $g_1$ is indeed equal to the fluid limit $g_{\mathrm{ER}}$ of the giant component in the classical Erd\H{o}s-Rényi random graph. 

We introduce simultaneously the function $d_p:[0,\infty) \rightarrow [0,1)$ defined by 
\begin{equation}
\label{def:d_p}
d_p(t)=t-g_p(t)-t(1-g_p(t))^2.
\end{equation}
Note that $d_p(t)=0$ for $t \in [0,1/2]$, $d_p(1/2^+)=0$ and $d_p(t)\sim t-1+O(e^{-2pt})$ when $t \rightarrow \infty$.

Our main result expresses the scaling limit of the processes $G_{p,n}$ and $D_{p,n}$ in terms of these deterministic functions. All other results of the paper rely on this one.

\begin{theorem}
\label{thm:fluid limit}
As $n \rightarrow \infty$, for the topology of uniform convergence on compacts,
$$\left(\bigg(\frac{G_{p,n}(\lfloor nt\rfloor)}{n}, \frac{D_{p,n}(\lfloor nt\rfloor)}{n} \bigg), t \geq 0 \right) ~\overset{\mathbb P}\longrightarrow~\left(\big(g_p(t),d_p(t)\big), ~t\geq 0 \right).$$
\end{theorem}

\vspace{-0.2cm}

\begin{figure}[!h]
	\centering
	\subfloat[The functions $g_1$ (blue) and $g_{1/2}$ (red)]{\includegraphics[width=8cm, height=5cm]{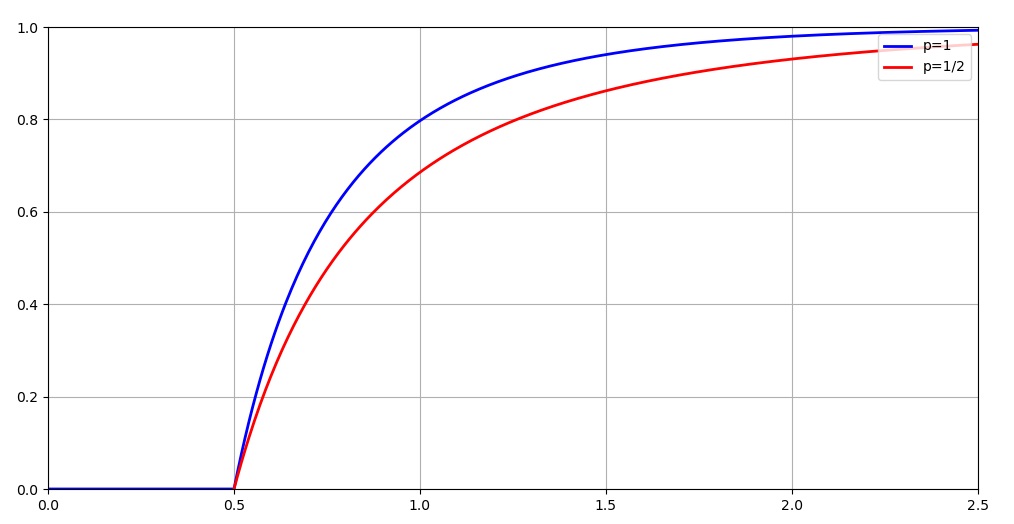}}
	\qquad \subfloat[The functions $d_1$ (blue) and $d_{1/2}$ (red)]{\includegraphics[width=8cm, height=4.9cm]{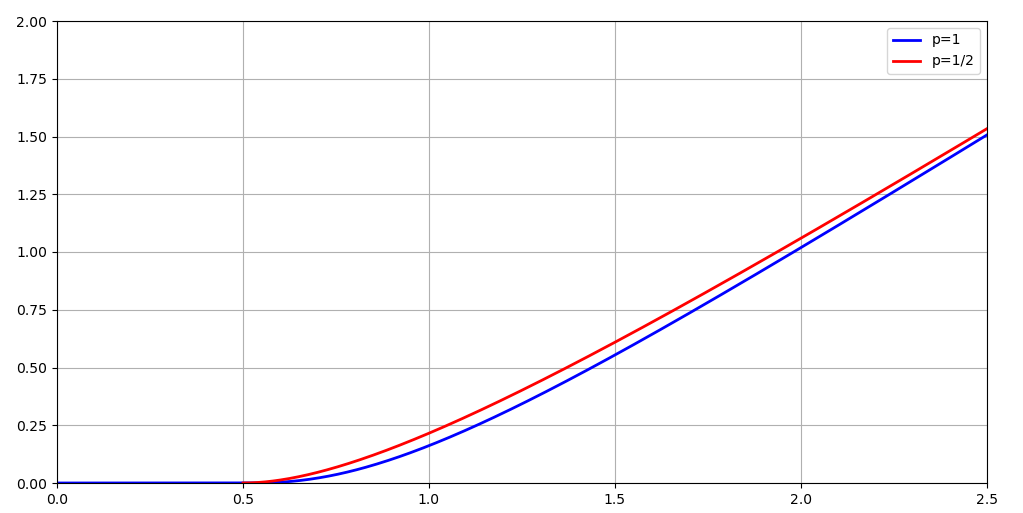}}	
	\label{fig:f_p_d_p}
\end{figure}

\bigskip

\emph{\textbf{Remark.}} When $t<1/2$, the limits $g_p(t)$ and $d_p(t)$ are null.  The phase transition at $t=1/2$ for the frozen model $\mathrm F_{p,n}$ was revealed by Contat and Curien \cite{ContatCurien23} (generalized in \cite{viau25} for $p\neq 1/2$). They proved that, as for the Erd\H{o}s-R\'enyi graph, the sizes of the connected components of the frozen model behave in $n^{2/3}$ when $m$ is of order $n/2$, as well as the size of the gel, and more precisely that appropriately rescaled in the critical window $ \lfloor n/2  + \lambda n^{2/3}\rfloor, \lambda \in \mathbb R$, the connected components converge, as a process in the variable $\lambda$, to a frozen multiplicative coalescent, generalizing thus the well-known result of Aldous \cite{aldous97} for the standard Erd\H{o}s-R\'enyi model. 
This is completed by the paper \cite{viau24}, which studies the behavior of the gel at the exit of the critical window. Theorem \ref{thm:fluid limit} therefore completes these results by describing the asymptotics in the supercritical regime $t>1/2$. Note that for $p=1$, it gives the functional convergence of the rescaled number of vertices in cyclic components of the standard Erd\H{o}s-R\'enyi graph towards the function $g_{\mathrm{ER}}$. 

\bigskip

\textbf{A word on the proof.} The approximation of trajectories of random processes by solutions to differential equations has been deeply studied. See e.g. Darling and Norris's survey \cite{DarlingNorris08} for background and references.  The proof of Theorem \ref{thm:fluid limit} is based on the so-called \emph{differential equation method}  as developed by  Wormald \cite{wormald95,wormald97} for discrete-time processes (notably related to combinatorial structures) whose jumps are not too big and well approximated by sufficiently smooth functions. The implementation of this method will not be trivial here because the differential equations involving the limit functions $g_p$ and $d_p$ (see (\ref{eq:f}), (\ref{eq:syst_EDO_0}) in Section \ref{sec:fluid}) are \emph{not smooth enough} around the critical time $t=1/2$ to apply the method as is. Bypassing this flaw will require a detailed technical work based on approximations of the differential equations, which will be undertaken in Section \ref{sec:mise_en_place}. We also emphasize that to obtain the fluid limit of $G_{p,n}$ via this approach, we really need to consider the two-dimensional process $(G_{p,n},D_{p,n})$ and apply the method to this bivariate process. The reason is that the expectation of the jump of $G_{p,n}$ at time $m$ given the past of the process $(G_{p,n},D_{p,n})$ until then depends both on (and only on) $G_{p,n}(m)$ and $D_{p,n}(m)$. 

At the heart of our approach there is a useful connection between  the frozen Erd\H{o}s-R\'enyi model and uniform random forests: conditionally on its number of vertices and edges at a given time, the forest part of the frozen model is a uniform random forest. This was highlighted by Contat and Curien \cite{ContatCurien23} in the case $p=1/2$, generalized without difficulty in \cite{viau25} to any $p \in [0,1]$, and called the \textit{free forest property}. Let us state it formally and denote, for $N \in \mathbb N$ and $M\in \mathbb Z_+$, by $\mathcal{W}(N,M)$ the set of unrooted unordered forests with $N$ labeled vertices $\{1,\ldots,N\}$ and $M$ edges (hence $N-M$ trees). 

\begin{proposition}[Free forest property, \cite{ContatCurien23},\cite{viau25}]
\label{prop:freeforest}
	For any $n\in \mathbb N,m\in \mathbb Z_+$, conditionally on $G_{p,n}(m)$ and $D_{p,n}(m)$, the forest part of $\mathrm{F}_{p,n}(m)$ is uniformly distributed over 
	$\mathcal{W}\left(V_{p,n}(m),E_{p,n}(m) \right)$ when $V_{p,n}(m) \geq 1$.
\end{proposition}

This property will be crucial at different steps of our study. 

\bigskip

\textbf{Further results.} To complete the results of Theorem \ref{thm:fluid limit} and those of  \cite{ContatCurien23} and \cite{viau25} in the critical window, we note that when $t<1/2$, the gel $G_{p,n}(\lfloor nt\rfloor)$ is bounded in probability: 

\begin{proposition}
For all $t<1/2$, $G_{p,n}(\lfloor nt\rfloor)=O_{\mathbb P}(1)$.
\end{proposition}

Indeed, as previously mentioned, $G_{p,n}(\lfloor nt\rfloor)$ is stochastically smaller than the total number of vertices involved in cyclic components at time $\lfloor nt\rfloor$ of the standard Erd\H{o}s-R\'enyi random graph, and it is known that for $t<1/2$ this number  converges in distribution as $n \rightarrow \infty$ (see e.g. Theorem 5 of \cite{Pittel88}; in fact this theorem states the convergence in distribution of the total number of vertices involved in unicycles, but jointly with the well-known fact that at time $t<1/2$ the number of vertices which are not involved in trees or unicycles converges in probability to 0, this gives the result). 

Let us also emphasize the following corollary of Theorem \ref{thm:fluid limit}, which
identifies the asymptotic distribution of the first time a which a given vertex is frozen. It follows from the fact that the probability that a given vertex is frozen at time $\lfloor nt \rfloor$ is equal, by exchangeability, to $\mathbb E[G_{p,n}(\lfloor nt \rfloor)]/n$. 

\begin{corollary}
\label{cor:entrance1}
Let $\tau^*_{p,n}$ be the time at which the vertex $1$ is frozen in the $p$-frozen model, $p \in (0,1]$. Then,
$$\frac{\tau^*_{p,n}}{n} ~ \underset{n \rightarrow \infty}{\overset{\mathrm{(d)}} \longrightarrow} ~X_{p},$$
where $X_{p}$ is a random variable with cumulative distribution function $g_p$.
\end{corollary}

Other consequences of Theorem \ref{thm:fluid limit} are developed in the next two sections.

\subsection{Fluid limit of the forest}

With the relations (\ref{def:V_{p,n}E_{p,n}}), the asymptotics of the number of vertices $V_{p,n}$ and edges $E_{p,n}$ of the forest part of the graph $\mathrm{F}_{p,n}$ follow directly from Theorem \ref{thm:fluid limit}, as well as that of the ratio
$$
R_{p,n}(m):=\frac{E_{p,n}(m)}{V_{p,n}(m)}, \quad m \in \mathbb Z_+,
$$
where we use the convention $0/0=0$.
Recalling the free forest property of Proposition \ref{prop:freeforest}, this ratio is a major source of information since there is also a phase transition for uniform random forests depending on the position of the ratio relative to $1/2$, see Britikov \cite{britikov88} and Luczak-Pittel \cite{LuczakPittel92} (their results are summarized in Section \ref{section:forests}). This point will e.g. lead to the forthcoming Corollary \ref{prop:cc}. 

We complete these asymptotics with the behavior of the number of trees of a given size. In that aim, consider the functions $t_{p,k}:[0,\infty) \rightarrow [0,1)$, $k \in \mathbb N,$ defined for $t \geq 0$ by 
\begin{equation}\label{def:t_pk}
	t_{p,k}(t)=\frac{k^{k-2}}{k!}\left(2t\right)^{k-1}\left(1-g_p(t)\right)^k \mathrm{e}^{-2kt\left(1-g_p(t)\right)}.
\end{equation}
We emphasize that for each fixed $t \geq 0$, the weights
$$
\frac{k \cdot t_{p,k}(t)}{1-g_p(t)}, \quad k \geq 1
$$
are those of a Borel distribution of parameter $2t(1-g_p(t)) \in [0,1]$, see the Appendix \ref{app:BT} for background. Such a  distribution is, among other things, the distribution of the total progeny of a subcritical Galton-Watson tree with Poisson offspring distribution with mean $2t(1-g_p(t))$. Based on Theorem \ref{thm:fluid limit}, we obtain:

\begin{theorem}
\label{thm:forest}
As $n \rightarrow \infty$, for the topology of uniform convergence on compacts,
$$\left(\bigg(\frac{V_{p,n}(\lfloor nt\rfloor)}{n}, \frac{E_{p,n}(\lfloor nt\rfloor)}{n}, R_{p,n}(\lfloor nt\rfloor)\bigg), ~ t\geq 0\right) ~\overset{\mathbb P}\longrightarrow ~\Big((v_p(t),e_p(t),r_p(t)),~ t\geq 0 \Big)$$
where
$$
v_p(t)=1-g_p(t); \qquad e_p(t)=t(1-g_p(t))^2; \qquad r_p(t)=t(1-g_p(t)). 
$$
Moreover,
$$\left(\bigg(\frac{k \cdot N_{p,n}^{(k)}(\lfloor nt\rfloor)}{n}\bigg)_{k\geq 1}\bigg),~ t\geq 0\right)~\overset{\mathbb P}\longrightarrow~\left(\big(k\cdot t_{p,k}(t) \big)_{k\geq 1}\big), ~t\geq 0 \right)$$
for the usual norm $\|x\|_1:=\sum_{k\geq 1}|x_k|$ on $\ell^1$, the space of summable sequences. 
\end{theorem}

\begin{figure}
	\centering
		\subfloat[The functions $e_1$ (blue) and $e_{1/2}$ (red)]{\includegraphics[width=7.5cm, height=5cm]{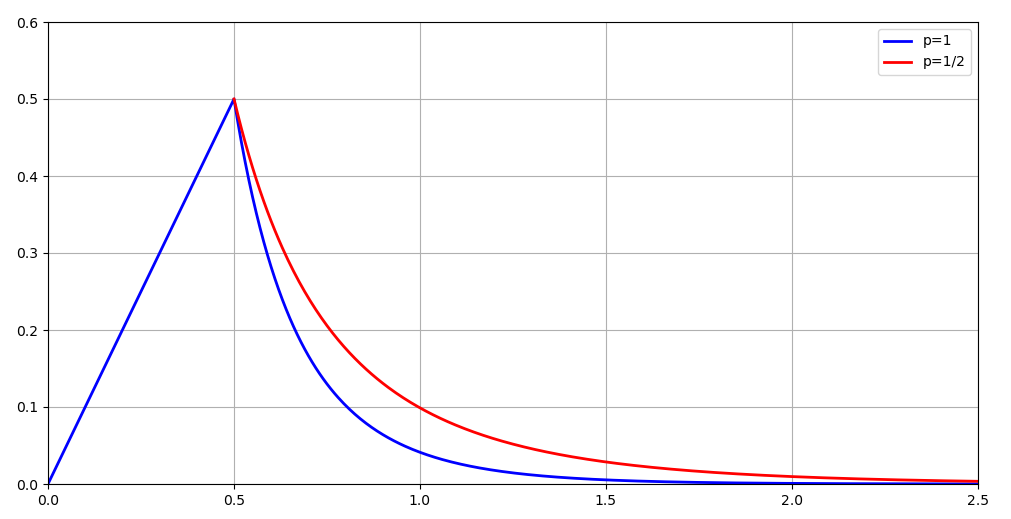}}
	\qquad \subfloat[The functions $r_1$ (blue) and $r_{1/2}$ (red)]{\includegraphics[width=7.5cm, height=5cm]{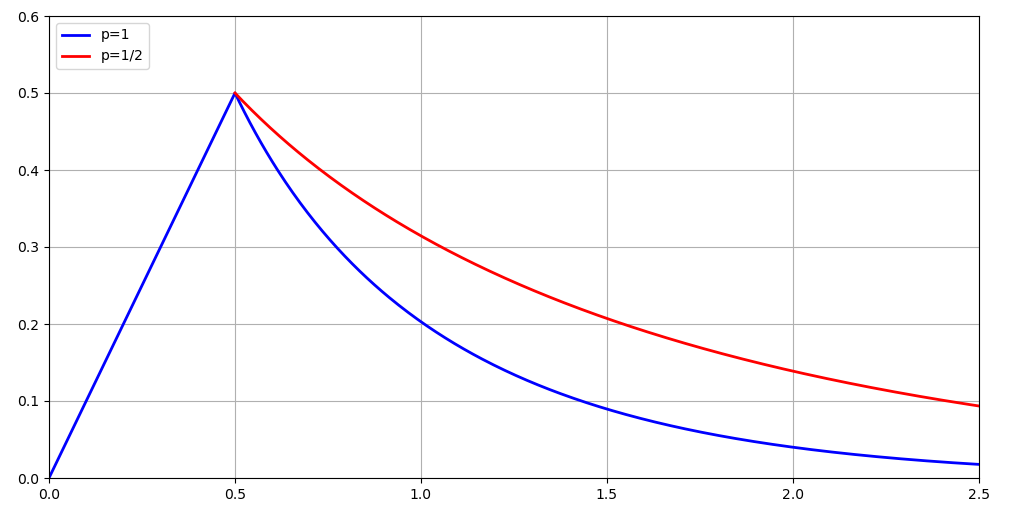}}
	\label{fig:f_p_d_p}
\end{figure}

\bigskip

Some properties of the functions $v_p,e_p, r_p$ and $t_{p,k}$ will be highlighted in Section \ref{sec:fluid}, in connection with the functions $g_p$ and $d_p$. For example, their right-derivative at $t=1/2$ are respectively $e'_p(1/2^+)=-1-2p$ and $r'_p(1/2^+)=-p$. We will also see that the ratio $r_p$ is increasing on $(0,1/2]$, decreasing on $[1/2,1)$, with a maximum equals to $1/2$ reached at $t=1/2$ (whatever $p \in (0,1]$). In particular,
$$
r_p(t)<1/2 \quad \text{for }t\neq 1/2.
$$
This means that for $t\neq 1/2$ the forest part of the graph $\mathrm{F}_{p,n}$ is in a subcritical regime. For $p=1$, this is related to the facts that the complement of the giant in the standard Erd\H{o}s-R\'enyi graph behaves as a standard Erd\H{o}s-R\'enyi graph conditioned to have connected components smaller than the initial giant, and that $t(1-g_{\mathrm{ER}}(t))<1/2$ for $t\neq 1/2$. For general $p \in (0,1]$, this subcriticality leads to: 
 
\bigskip

\begin{corollary}[Largest trees]
\label{prop:cc}
Let $\#T^{(i)}_{p,n}(m)$ be the size of the $i$-th largest tree in $\mathrm{F}_{p,n}(m)$. Then for all $i \in \mathbb N$:
\begin{enumerate}[topsep=0cm]
\item[\emph{1)}] When $t<1/2$, 
$$\frac{\#T^{(i)}_{p,n}(\lfloor nt \rfloor)}{\ln(n)} ~ \underset{n \rightarrow \infty}{\overset{\mathbb P} \longrightarrow}~  \frac{1}{2t-1-\ln(2t)}.$$
\item[\emph{2)}] When $t>1/2$, $$\frac{\#T^{(i)}_{p,n}(\lfloor nt \rfloor)}{\ln(n)} ~ \underset{n \rightarrow \infty}{\overset{\mathbb P} \longrightarrow}~  \frac{1}{2t(1-g_p(t))-1-\ln(2t(1-g_p(t)))}.$$\end{enumerate}
\end{corollary}

\bigskip

Informally, we have therefore, for large $n$,
$$
\#T^{(i)}_{p,n}(\lfloor nt \rfloor)~\underset{t \rightarrow 0}\approx~\frac{\ln(n)}{|\ln(t)|}, \quad \text{while} \quad  \#T^{(i)}_{p,n}(\lfloor nt \rfloor)~\underset{t \rightarrow \infty}\approx~\frac{\ln(n)}{2pt}.
$$

\bigskip

\emph{\textbf{Remark.}} To complete the above corollary, let us state the results of \cite{ContatCurien23} and \cite{viau25} in the critical window a little more precisely: if we let $\#U^{(i)}_{p,n}(m)$ be the size of the $i$-th largest unicycle in $\mathrm{F}_{p,n}(m)$, then the process of couple of sequences $$\left(\left(\left(\frac{\#T^{(i)}_{p,n}(\lfloor \frac{n}{2} + \lambda n^{2/3}\rfloor)}{n^{2/3}}\right)_{i\geq 1},~\left(\frac{\#U^{(i)}_{p,n}(\lfloor \frac{n}{2} + \lambda n^{2/3}\rfloor)}{n^{2/3}}\right)_{i\geq 1}\right), \lambda \in \mathbb R \right)$$
has a limit in distribution in $\ell^2 \times \ell^1$ towards a frozen multiplicative coalescent.

\bigskip

We finish this section with a corollary on the geometry of a typical tree.

\begin{corollary}[Typical tree]
\label{cor:CC1}
Let $\mathrm{CC}^*_{p,n}(m)$ denote the connected component of $\mathrm F_{p,n}(m)$ containing the vertex $1$, with vertices relabeled in increasing order from $1$ to $\# \mathrm{CC}^*_{p,n}(m)$. Then,
$$
\mathrm{CC}^*_{p,n}(\lfloor nt \rfloor)~ | ~  \mathrm{CC}^*_{p,n}(\lfloor nt \rfloor)\text{ is a tree }~ \underset{n \rightarrow \infty}{\overset{\mathrm{(d)}} \longrightarrow}~\mathrm{GW}_{\mathrm {Poi}(2t(1-g_p(t)))},
$$
where $\mathrm{GW}_{\mathrm {Poi}(2t(1-g_p(t)))}$ designs a Galton-Watson tree whose offspring distribution is Poisson with mean $2t(1-g_p(t))$, equipped with uniform random labels from $1$ to $\#\mathrm{GW}_{\mathrm {Poi}{2(1-g_p(t))}}$ on its vertices, and where the original order is forgotten, as well as the root. 
In a related way,
$$\mathbb P \left( \mathrm{CC}^*_{p,n}(\lfloor nt \rfloor)\text{ is a tree of size $k$} \right)~ \underset{n \rightarrow \infty}{\overset{\mathrm{(d)}} \longrightarrow} ~ k\cdot t_{p,k}(t).$$
\end{corollary}

\bigskip

As already mentioned, the function $t \mapsto 2t(1-g_p(t))=2r_p(t)$ is increasing on $(0,1/2]$, decreasing on $[1/2,1)$, with a maximum equal to 1 reached at $t=1/2$. The Galton-Watson tree appearing in the limit above is therefore subcritical for $t \neq 1/2$ and critical for $t=1/2$. When $t \in (0,1/2]$, since $g_p(t)=0$, it is simply a Galton-Watson tree with a Poisson offspring distribution with mean $2t$.

\subsection{Total gelation time and vicinity}

Theorem \ref{thm:fluid limit} also leads us to a precise asymptotic for the first time at which the $n$ vertices of $\mathrm F_{p,n}$ are all frozen, as well as related quantities. We call this time the \emph{absorption time} or \emph{total gelation time} and denote it by 
$$A_{p,n}:=\inf\big\{m\geq 0 : G_{p,n}(m)=n\big\}.$$
Similarly to what happens in the standard Erd\H{o}s-R\'enyi model for the first time at which the graph is connected, we will see that $A_{p,n}$ is identical, with high probability, to the first time at which there are no more isolated vertices  in the process. This is explained by the fact that larger trees aggregate more quickly to the gel and so disappear earlier. If we let $A^{(k)}_{p,n}$ denote the last time at which there are some trees of size $k \in \mathbb N$ in the process and $A_{p,n}^{(k+)}$ the last time at which there are some trees of size larger or equal to $k \in \mathbb N$ (so that $A_{p,n}^{(1+)}=A_{p,n}$), we will see that for $n$ large
$$
A_{p,n}   \  \approx    \; A^{(1)}_{p,n}  \  \approx   \  \frac{n\ln n}{2p}
$$
and more generally that
$$
A_{p,n}^{(k+)} \; \approx \; A^{(k)}_{p,n}   \; \approx   \; \frac{n}{2} \cdot \left( \frac{\ln(n)}{kp}+\frac{(k-1)}{kp}\cdot \ln\left(\frac{\ln(n)}{kp}\right) \right).
$$
In the following we set
\begin{equation}
\label{def:tnk}
\mathsf t^{(k)}_{p,n}=\frac{\ln(n)}{kp}+\frac{(k-1)}{kp}\cdot \ln\left(\frac{\ln(n)}{kp}\right),
\end{equation}
and, to describe precisely the above asymptotics, introduce the following notation:
\begin{enumerate}[topsep=0cm]
\item[-] $\gamma_E$ is Euler's constant
\item[-] $\psi$ is the digamma function, that is $\psi(x)=\frac{\Gamma'(x)}{\Gamma(x)}, x>0$, with $\Gamma$ the gamma function
\item[-] $\mathrm{Gu}$ is a standard Gumbel random variable, that is with cumulative distribution function 
$e^{-e^{-x}}$, $x \in \mathbb R$.
\end{enumerate}
We emphasize that the function $p \in (0,1] \rightarrow \psi(1/p) +\gamma_E$ is decreasing, equal to 0 when $p=1$ and to $1$ when $p=1/2$, and that $\psi(1/p) +\gamma_E \sim -\ln p$ when $p \rightarrow 0$.

\begin{theorem}
\label{thm:absorption} 
For all $k \in \mathbb N$, 
as $n \rightarrow \infty$, 
$$
\mathbb P\Big(A^{(k+)}_{p,n} =A^{(k)}_{p,n}  \Big)  \; \underset{n \rightarrow \infty}{\longrightarrow} \; 1
$$
and
$$
\frac{A^{(k+)}_{p,n} }{n} - \frac{\mathsf t^{(k)}_{p,n}}{2}  \; \underset{n \rightarrow \infty}{\overset{(\mathrm d)}\longrightarrow} \; \frac{\mathrm{Gu}}{2kp}-\frac{\psi(1/p)+\gamma_{\mathrm E}}{2p}+\frac{\ln (k^{k-2}/k!)}{2kp}.
$$
\end{theorem}

\bigskip

In particular, the absorption time $A_{p,n}$ behaves as
$$
\frac{A_{p,n}}{n} - \frac{\ln(n)}{2p}  \; \underset{n \rightarrow \infty}{\overset{(\mathrm d)}\longrightarrow} \; \frac{\mathrm{Gu}}{2p}-\frac{\psi(1/p)+\gamma_{\mathrm E}}{2p}.
$$

The proof of Theorem \ref{thm:absorption} uses a continuous version of the $p$-frozen model. In this continuous framework, we obtain exact expressions of the factorial moments of the total number of vertices involved in a tree of size $k \in \mathbb N$ at any time $t \geq 0$, which depend on the (continuous version of the) gel process $G_{p,n}$. Theorem \ref{thm:fluid limit} will then give their asymptotic behaviors. As an intermediate and complementary result to Theorem \ref{thm:absorption}, we obtain via this approach the behavior in distribution of the number of trees of size $k$ around the threshold times $n \cdot \mathsf t^{(k)}_{p,n}/2$. More precisely, if we let $N_{p,n}^{(k)}\big(\big \lfloor n \cdot \big(\mathsf t^{(k)}_{p,n}+c\big) /2\big \rfloor \big)$ be the number of trees of size $k$ at time $\big \lfloor n \big(\mathsf t^{(k)}_{p,n}+c\big)/2\big \rfloor $, for $c \in \mathbb R$, in the discrete model $\mathrm{F}_{p,n}$, one then has:

\medskip

\begin{proposition}
\label{prop:arbreskdiscrets}
For all $k \in \mathbb N$ and all $c \in \mathbb R$
$$
N_{p,n}^{(k)}\left(\left \lfloor \frac{n}{2} \cdot \big(\mathsf t^{(k)}_{p,n}+c\big) \right \rfloor \right) \; \underset{n \rightarrow \infty}{\overset{(\mathrm d)}\longrightarrow} \; \mathcal P\left(\frac{k^{k-2}e^{-kpc} e^{-k\left(\psi(1/p)+\gamma_{\mathrm E}\right)}}{k!} \right),
$$
where  the notation $\mathcal P(\lambda)$ refers to a Poisson distribution with expectation $\lambda>0$.

Additionally, each positive moment of $N_{p,n}^{(k)}\big(\big \lfloor n \cdot \big(\mathsf t^{(k)}_{p,n}+c\big) /2\big \rfloor \big)$ converges to the corresponding moment of the Poisson distribution.
\end{proposition}

\textbf{Remark.} When $p=1$, using that the forest part of the frozen model is distributed as the forest part of the standard Erd\H{o}s-R\'enyi graph, Proposition \ref{prop:arbreskdiscrets} recovers a result by Erd\H{o}s and R\'enyi \cite{ErdosRenyi60} for the number of trees of size $k$ in their model. When $k=1$ this result is well known, but the cases $k\geq 2$ are perhaps less known.  Regarding  Theorem \ref{thm:absorption} again when $p=1$, the case $k=1$  corresponds to the asymptotic behavior of the connectedness time $A_{\mathrm{ER},n}$ in the standard Erd\H{o}s-R\'enyi graph and retrieves (\ref{connec:ER}). We are not aware if similar counterparts when $k\geq 2$ were investigated for the standard Erd\H{o}s-R\'enyi model. In any case,  Theorem \ref{thm:absorption} gives the asymptotics  for this model of the last time at which there are some trees of size $k$ (say $A^{(k)}_{\mathrm{ER},n}$), and the last time at which there are some trees of size larger or equal to $k$ (say $A^{(k+)}_{\mathrm{ER},n}$), with $~\mathbb P\big(A^{(k+)}_{\mathrm{ER},n} =A^{(k)}_{\mathrm{ER},n}  \big)  \underset{n \rightarrow \infty}{\rightarrow}  1~$
and
$$
\frac{A^{(k+)}_{\mathrm{ER},n} }{n} - \frac{\mathsf t^{(k)}_{1,n}}{2}  \; \underset{n \rightarrow \infty}{\overset{(\mathrm d)}\longrightarrow} \; \frac{\mathrm{Gu}}{2k}+\frac{\ln (k^{k-2}/k!)}{2k}.
$$

\subsection{Organization of the paper}

In Section \ref{section:forests}, the connections of the model $\mathrm F_{p,n}$ with uniform random forests are used to obtain asymptotics on the jumps of the processes $G_{p,n}$ and $D_{p,n}$, which are preliminary results needed to establish the fluid limit results. Some properties of the functions $g_p,d_p,v_p,e_p$ and $r_p$ are then highlighted in Section \ref{sec:fluid}. Section \ref{sec:mise_en_place} is devoted to the proofs of the fluid limit results, Theorem \ref{thm:fluid limit} and Theorem \ref{thm:forest}, as well as Corollary \ref{prop:cc} and Corollary \ref{cor:CC1}.  
A continuous, {Poissonized}, version of the model is  introduced and studied in Section \ref{sec:continuous}. It is used in Section \ref{sec:gelation} to prove continuous counterparts of Theorem \ref{thm:absorption} and Proposition \ref{prop:arbreskdiscrets}, and then these results themselves. Finally, in Section \ref{sec:open}, we discuss the case $p=0$ and some open questions concerning the sizes of unicycle components in the supercritical regime, for any $p$. The paper ends with an Appendix \ref{sec:app} recalling background on the Borel-Tanner distribution and Wormald's result on the differential equation method which we will use for the proofs of the fluid limit.

\section{Frozen Erd\H{o}s-R\'enyi and uniform random forests}
\label{section:forests}

As stated in Proposition \ref{prop:freeforest}, it turns out that conditionally on its number of vertices $V_{p,n}(m)$ and edges $E_{p,n}(m)$, the forest part of the frozen model $\mathrm{F}_{p,n}(m)$ is a uniform random forest. This property is crucial for our study. We develop here several consequences.

A main point concerns the expressions of the distribution of the jumps of the processes $G_{p,n}$ and $D_{p,n}$ at time $m$, given their history until then. This will be fondamental to implement the results on the fluid limit.
We use in the following proposition and throughout the paper the notation $$ \Delta G_{p,n}(m)=G_{p,n}(m+1)-G_{p,n}(m),$$ and similarly for $D_{p,n}$, to denote the increments of these processes, and let $(\mathbf F_{p,n}(m),m\geq 1)$ design the filtration generated by $(G_{p,n},D_{p,n})$, or equivalently $(V_{p,n},E_{p,n})$ recalling the relations (\ref{def:V_{p,n}E_{p,n}}). Recall also that for $N \in \mathbb N$ and $M\in \mathbb Z_+$, with $M \leq N-1$, $\mathcal{W}(N,M)$ denotes the set of unrooted unordered forests with $N$ labeled vertices $\{1,\ldots,N\}$ and $M$ edges. 

\medskip

\begin{proposition}\label{lm_transitions}
	For every $n\in \mathbb N$, $m\in \mathbb Z_+$ and $k\in \llbracket 1; E_{p,n}(m)+1 \rrbracket$,
	\setlength{\jot}{10pt}
	\begin{eqnarray*}
      && \hspace{-1cm}\mathbb{P}\left( \Delta G_{p,n}(m)=k ~| ~ \mathbf{F}_{p,n}(m)\right) \\
	  & =&\binom{V_{p,n}(m)}{k} \cdot k^{k-2} \cdot \frac{\#\mathcal{W}(V_{p,n}(m)-k,E_{p,n}(m)-k+1)}{\#\mathcal{W}(V_{p,n}(m),E_{p,n}(m))} \cdot \left(\frac{k(k-1)+2pkG_{p,n}(m)}{n(n-1)}\right),
	\end{eqnarray*}
	with the conventions $\#\mathcal{W}(0,0)=1$, $\#\mathcal{W}(N,N)=0$ for any $N \in \mathbb N$, $\#\mathcal{W}(-1,0)=0$,
	and 
	$$\mathbb{P}\left(\Delta D_{p,n}(m)=1~| ~ \mathbf{F}_{p,n}(m)\right)=2(1-p) \cdot \frac{G_{p,n}(m)(n-G_{p,n}(m))}{n(n-1)}+\frac{G_{p,n}(m)(G_{p,n}(m)-1)}{n(n-1)}.$$
\end{proposition}

\begin{proof} Given that a tree of size $k$ is a connected component of $\mathrm{F}_{p,n}(m)$, it will freeze at time $m+1$ 
\begin{enumerate}[topsep=0pt]
	\item[-] either if the edge selected at time $m+1$ involves two vertices of that tree, which happens with probability $\frac{k(k-1)}{n(n-1)}$
	\item[-] or if the edge selected at time $m+1$ involves a vertex of the tree and a vertex of the freezer and is retained, which happens with probability $p \cdot \frac{2k G_{p,n}(m)}{n(n-1)}$.
\end{enumerate}
Next, with a set of $V_{p,n}(m) \geq 1$ vertices, one can build for $k\leq V_{p,n}(m)$ 
$$
\binom{V_{p,n}(m)}{k} \cdot k^{k-2} \quad \text{different trees of size $k$}
$$
(recall Cayley's formula: there are $k^{k-2}$ different trees on a fixed set of $k$ vertices). And since the forest part of $\mathrm{F}_{p,n}(m)$, conditionally on $\mathbf{F}_{p,n}(m)$, is a uniform random forest with $V_{p,n}(m)$ vertices and $E_{p,n}(m)$ edges, the probability that a given tree of size $k \leq E_{p,n}(m)+1$, $k\geq 1$, belongs to this forest is 
$$
\frac{\#\mathcal{W}(V_{p,n}(m)-k,E_{p,n}(m)-k+1)}{\#\mathcal{W}(V_{p,n}(m),E_{p,n}(m))},
$$
with the conventions of the statement when $E_{p,n}(m)=V_{p,n}(m)-1$ or $E_{p,n}(m)=V_{p,n}(m)=0$. Gathering these remarks gives the stated expression of $\mathbb{P}\left( \Delta G_{p,n}(m)=k ~| ~ \mathbf{F}_{p,n}(m)\right)$.

Regarding $D_{p,n}$, simply note that the edge selected at time $m+1$ is discarded
\begin{enumerate}[topsep=0pt]
	\item[-] either if it involves two vertices of the freezer of $\mathrm{F}_{p,n}(m)$, which, conditionally on $\mathbf{F}_{p,n}(m)$, happens with probability $\frac{G_{p,n}(m)\left(G_{p,n}(m)-1\right)}{n(n-1)}$
	\item[-] or if it involves a vertex of the forest and a vertex of the freezer and it is not retained, which,  conditionally on $\mathbf{F}_{p,n}(m)$, happens with probability $(1-p) \cdot 2 \cdot  \frac{G_{p,n}(m)\left(n-G_{p,n}(m)\right)}{n(n-1)}$.
\end{enumerate}
\end{proof}

\smallskip

In the rest of the section we recall some background on uniform random forests (Section \ref{sec:URF}) in order to evaluate the asymptotics for $n$ large of the above conditional distributions (Section \ref{sec:asymp_conditioned}) and to get estimates on the largest jump of $G_{p,n}$ (Section \ref{sec:largest_jump}).

\subsection{Background on uniform random forests}
\label{sec:URF}

We gather here the information we need on the enumeration of uniform random forests, following Kolchin  \cite{kolchin86} and Britikov \cite{britikov88}, and on the sizes of the largest connected components of the forest, following Luczak-Pittel \cite{LuczakPittel92} and Bernikovich-Pavlov \cite{BernikovichPavlov11}. A key point is that the sizes of the connected components of a uniform random forest can be interpreted as the increments of a conditioned random walk. For $x\in (0,\mathrm{e}^{-1}]$ consider the probability measure
\begin{equation}\label{def:mu_x}
\mu_x(k)=\frac{k^{k-2}}{k!} \cdot \frac{x^k}{T(x)},\quad k\geq 1, \quad \text{with } \quad T(x)=\sum_{k\geq 1}\frac{k^{k-2}}{k!}x^k.
\end{equation}
Setting $x=\theta e^{-\theta}$ with $\theta \in (0,1]$, one checks that $T(x)=\theta-\theta^2/2~$, that the expectation of $\mu_x$ is  $2/(2-\theta)$ and its variance $2\theta/(1-\theta)(2-\theta)^2$ (see e.g. Lemma \ref{lm:lambert_fct} in the Appendix).

The following result dates back at least to Kolchin  \cite{kolchin86} and Britikov \cite{britikov88} and was formulated as is by Contat-Curien \cite{ContatCurien23}.

\begin{proposition}[\cite{kolchin86},\cite{britikov88},\cite{ContatCurien23}]
\label{prop_forest_rw}
Let $N \in \mathbb N$, $M \in \mathbb Z_+$ with $M\leq N-1$ and for $x\in (0,e^{-1}]$, \linebreak $\big(S_i^{(x)}:0\leq i\leq N-M\big)$ be a random walk with i.i.d. increments of law $\mu_x$, started from $S_0^{(x)}=0$.
\begin{enumerate}[topsep=0cm]
\item[\emph{1)}] Whatever $x\in (0,e^{-1}]$, the cardinal of $\mathcal W(N,M)$ is given by
	$$\#\mathcal{W}(N,M)=\frac{N!}{(N-M)!} \cdot \frac{T(x)^{N-M}}{x^N} \cdot \mathbb{P}\big(S_{N-M}^{(x)}=N\big).$$
\item[\emph{2)}] If $W(N,M)$ is a uniform random forest of $\mathcal W(N,M)$ and $\mathcal{C}_1,\ldots,\mathcal{C}_{N-M}$ denote the sizes of its connected components indexed in a uniform random order, then, whatever $x\in (0,\mathrm{e}^{-1}]$, the vector $\left(\mathcal{C}_1,...,\mathcal{C}_{N-M}\right)$ has the same law as the increments of $\big(S_i^{(x)}:0\leq i\leq N-M\big)$ conditioned on $S_{N-M}^{(x)}=N$. Moreover, conditionally on their sizes, the connected components are independent uniform Cayley trees.
\end{enumerate}
\end{proposition}

Britikov \cite{britikov88} used the first point to estimate the asymptotic of $\#\mathcal{W}(N,M)$ in different regimes, using for each of them an appropriate value of $x$ (depending possibly on $N,M$) to obtain relevant estimates. The case $x=\mathrm{e}^{-1}$ is of particular interest in the critical regime: the resulting measure $\mu_{\mathrm{e}^{-1}}$ is then heavy-tailed, in the domain of attraction of a $3/2$-stable law, with $\mu_{\mathrm{e}^{-1}}(k)\sim \sqrt{\frac{2}{\pi}}k^{-5/2}$ as $k\to\infty$, and its expectation is equal to $2$. Let 
\begin{equation*}\label{def:p1}
	p_1(x)=\frac{1}{\pi}\int_{0}^{\infty}\mathrm{e}^{-\frac{2}{3}t^{3/2}}\cos\left(xt+\frac{2}{3}t^{3/2}\right)\mathrm{d}t
\end{equation*}
be the density of the corresponding $3/2$-stable law. Britikov's result reads as follows. 

\begin{proposition}[Britikov \cite{britikov88}]
\label{lm:Britikov}
Let $\omega=(2M-N)/N^{2/3}$, with $N \in \mathbb N$, $M \in \mathbb Z_+$, $M\leq N-1$.
\begin{enumerate}[topsep=0cm]
\item[\emph{1)}] \emph{(Subcritical regime)}
When $\omega \to -\infty$,  
$$\#\mathcal{W}(N,M)=\left(1+o(1)\right)\cdot \frac{N^{2M}}{2^M M!} \cdot \left(1-\frac{2M}{N}\right)^{1/2}.$$
\item[\emph{2)}]  \emph{(Near-critical regime)}
When $N \rightarrow \infty$ and $\omega$ is bounded,
$$\#\mathcal{W}(N,M)=\left(1+o(1)\right)\cdot \frac{N^{N-1/6}}{2^{N-M}(N-M)!} \cdot \sqrt{2\pi} \cdot p_1\left(\omega\right).$$ 
\item[\emph{3)}]  \emph{(Supercritical regime)} When $\omega \to \infty$, 
$$ \#\mathcal{W}(N,M)=\left(1+o(1)\right)\cdot  \frac{N^{N-2}}{2^{N-M-1}(N-M-1)!} \cdot \left(\frac{2M}{N}-1\right)^{-5/2}.$$
\end{enumerate}
\end{proposition}

\smallskip

Note that $\omega \to -\infty$ or $\omega \to \infty$ implies $N \rightarrow \infty$.

Using these estimates, Luczak and Pittel \cite{LuczakPittel92}  studied the asymptotics of the largest components of uniform random forests in each of the three regimes, showing similar, yet different, behaviors to the Erd\H{o}s-R\'enyi graph: a phase transition occurs according to whether  $M/N<1/2$ (with largests components of order $\ln(N)$), $M/N \sim 1/2$ (with largests components of order $N^{2/3}$) and $N/M>1/2$ (where a giant component emerges); however in the supercritical regime, removing the giant tree results in a critical random forest, whereas removing the giant component in the Erd\H{o}s-Rényi graph gives a subcritical Erd\H{o}s-Rényi graph. We specify some of Luczak and Pittel's results in the subcritical regime -- which themselves are based on results of Erd\H{o}s and R\'enyi \cite{ErdosRenyi60} in the subcritical regime of their model -- as we shall need them later. 

\begin{proposition}[Luczak-Pittel \cite{LuczakPittel92}, Theorem 3.1 (ii)]
\label{prop:largest_trees}
For $i\in \mathbb N$, let $L_i(N,M)$ denote the size of the $i$-th largest tree in a uniform random forest with $N$ vertices and $M$ edges. Then, when $\big(N,M/N\big) \rightarrow (\infty,c)$, with $c \in (0,1/2)$,
$$
\frac{L_i(N,M)}{\ln(N)} \overset{\mathbb P}\longrightarrow \frac{1}{2c-1-\ln(2c)}.
$$
\end{proposition}

\subsection{Asymptotics of expected conditional jumps}
\label{sec:asymp_conditioned}

Combining Proposition \ref{lm:Britikov} with Proposition \ref{lm_transitions} gives the asymptotics of Corollary \ref{corol:est_transitions} below. In this statement, we specify the variable of the Landau notation $o(1)$ by writing $o_l(1)$ for a (deterministic) quantity that vanishes as $l$ is large (this function may differ in each assertion) . 
Also, for all $n \in \mathbb N, m \in \mathbb Z_+$ such that $V_{p,n}(m)\geq 1$ and all $k \in \mathbb N, k <V_{p,n}(m)$, we set 
$$\Omega_{p,n}(m)=\frac{2E_{p,n}(m)-V_{p,n}(m)}{(V_{p,n}(m))^{2/3}}, \qquad \Omega_{p,n}^{(k)}(m)=\frac{2E_{p,n}(m)-V_{p,n}(m)-k+2}{\left(V_{p,n}(m)-k\right)^{2/3}}$$
and $\Omega_{p,n}(m)=0$ when $V_{p,n}(m)=0$, $\Omega_{p,n}^{(k)}(m)=0$ when $k \geq V_{p,n}(m)$.

\begin{corollary}\label{corol:est_transitions} Let $\epsilon:\mathbb Z_+ \rightarrow \mathbb R$ denote a positive function that vanishes at infinity. 
\begin{enumerate}[topsep=0cm]
\item[\emph{1)}] When $(\Omega_{p,n}(m),E_{p,n}(m))\to (-\infty,\infty)$,  
\begin{eqnarray*}
\mathbb{P}\left( \Delta G_{p,n}(m)=k ~ | ~\mathbf{F}_{p,n}(m)\right) &=& \big(1+o_{(\Omega_{p,n}(m),E_{p,n}(m))}(1)\big) \cdot \frac{k^{k-1}}{k!} \cdot \left(\frac{2E_{p,n}(m)}{V_{p,n}(m)}\right)^{k-1}\\
&& \hspace{0cm} \times  ~\mathrm{e}^{-2k\frac{E_{p,n}(m)}{V_{p,n}(m)}} \cdot \left(\frac{k-1+2pG_{p,n}(m)}{n}\right)\frac{n-G_{p,n}(m)}{n},
\end{eqnarray*}	
where the function $o_{\cdot}(1)$ is uniform over all $k$ such that $1\leq k\leq \epsilon(E_{p,n}(m))\left(E_{p,n}(m) \right)^{1/2}$.
\item[\emph{2)}] Let $c>0$ be some arbitrary constant.  When $\left\vert \Omega_{p,n}(m)\right\vert \leq c $ and $V_{p,n}(m) \rightarrow \infty$,
\begin{eqnarray*}
&& \hspace{-1cm}
	\mathbb{P}\left( \Delta G_{p,n}(m)=k ~ | ~\mathbf{F}_{p,n}(m)\right) \\
	&=& \big(1+o_{V_{p,n}(m)}(1)\big) \cdot \frac{2k^{k-1}}{k!e^k} \cdot \frac{V_{p,n}(m)-E_{p,n}(m)}{n} \cdot \left(\frac{k-1+2pG_{p,n}(m)}{n}\right) \cdot \frac{p_1\left(\Omega_{p,n}^{(k)}(m)\right)}{p_1\left(\Omega_{p,n}(m)\right)} 
\end{eqnarray*}	
where the function $o_{\cdot}(1)$ is uniform over all $k$ such that $1\leq k\leq \epsilon(V_{p,n}(m))\left(V_{p,n}(m)\right)^{1/2}$. 
\item[\emph{3)}] When $\Omega_{p,n}(m)\to \infty$,
\begin{eqnarray*}
	&& \hspace{-2cm} \mathbb{P}\left( \Delta G_{p,n}(m)=k ~ | ~\mathbf{F}_{p,n}(m)\right) \\
	&=&\big(1+o_{\Omega_{p,n}(m)}(1)\big) \cdot \frac{2k^{k-1}}{k!\mathrm{e}^k}\cdot \left(\frac{k-1+2pG_{p,n}(m)}{n}\right)\frac{{V_{p,n}(m)-E_{p,n}(m)-1}}{n}
\end{eqnarray*}
where the function $o_{\cdot}(1)$ is uniform over all $k$ such that $1\leq k\leq \epsilon(V_{p,n}(m))\left(V_{p,n}(m)\right)^{1/2}$.
\end{enumerate}	
\end{corollary} 

\medskip

\begin{remark}
	A consequence of 2) and 3), together with Stirling's formula and the fact that the function $p_1$ is strictly positive on $\mathbb{R}$, is that for any $A>0$ and any function $\epsilon$ that vanishes at infinity, there exists a constant $c_{A,\epsilon,p}>0$ such that for every $V_{p,n}(m),E_{p,n}(m)$ verifying $2E_{p,n}(m)-V_{p,n}(m)\geq -AV_{p,n}(m)^{2/3}$ with $V_{p,n}(m)$ large enough, and then every $1\leq k\leq \epsilon\left(V_{p,n}(m)\right)\left(V_{p,n}(m)\right)^{1/2}$,
	\begin{equation}\label{eq:crit_surcrit}
		\mathbb{P}\left(\Delta G_{p,n}(m)=k ~ | ~\mathbf{F}_{p,n}(m)\right)\geq \frac{c_{A,\epsilon,p}}{k^{3/2}} \cdot \frac{G_{p,n}(m)\left(V_{p,n}(m)-E_{p,n}(m)-1\right)}{n^2}.
	\end{equation}	
\end{remark}

\medskip

\begin{proof} We shall repeatedly use the following consequence of Stirling's formula: as $l \in \mathbb N \rightarrow \infty$, uniformly for all integers $k \in \big[ 1,\epsilon(l) \sqrt l \big]$ 
$$
\frac{l!}{(l-k)!}=(1+o_l(1)) \cdot l^k.
$$	
1) When $\Omega_{p,n}(m)\to -\infty$, one has $~V_{p,n}(m)\to \infty~$ and so $~\Omega_{p,n}^{(k)}(m)= \Omega_{p,n}(m)\big(1+o_{V_{p,n}(m)}(1)\big)$, uniformly in $1\leq k\leq \left(V_{p,n}(m)\right)^{1/2}$. Applying the subcritical regime estimate of Proposition \ref{lm:Britikov} together with Stirling's formula, we thus get
\begin{spreadlines}{5pt}
\begin{align*}
	\frac{\#\mathcal{W}(V_{p,n}(m)-k,E_{p,n}(m)-k+1)}{\#\mathcal{W}(V_{p,n}(m),E_{p,n}(m))}&=\big(1+o_{\Omega_{p,n}(m)}(1)\big)\cdot 2^{k-1} \cdot \frac{\left(V_{p,n}(m)-k\right)^{2(E_{p,n}(m)-k+1)}}{V_{p,n}(m)^{2E_{p,n}(m)}}\\
         & \hspace{-0.5cm} \times  \frac{E_{p,n}(m)!}{\left(E_{p,n}(m)-k+1\right)!} \cdot \left(\frac{\left(2E_{p,n}(m)-V_{p,n}(m)-k+2\right)V_{p,n}(m)}{(2E_{p,n}(m)-V_{p,n}(m))(V_{p,n}(m)-k)}\right)^{1/2} \\
	&=\big(1+o_{(\Omega_{p,n}(m),E_{p,n}(m))}(1)\big) \cdot 2^{k-1} \cdot V_{p,n}(m)^{-2k+2} E_{p,n}(m)^{k-1} \\
	& \hspace{-0.5cm} \times \mathrm{e}^{-2kE_{p,n}(m)/V_{p,n}(m)}   
\end{align*}
\end{spreadlines}
when $(\Omega_{p,n}(m),E_{p,n}(m))\to (-\infty,\infty)$, uniformly for all $1\leq k\leq \epsilon(E_{p,n}(m))\left(E_{p,n}(m)\right)^{1/2}$. Together with Proposition \ref{lm_transitions} this leads to 
\begin{align*}
 \mathbb{P}\left( \Delta G_{p,n}(m)=k | \mathbf{F}_{p,n}(m)\right)&=\big(1+o_{(\Omega_{p,n}(m),E_{p,n}(m))}(1)\big) \cdot \frac{k^{k-1}}{k!} \cdot \left(2E_{p,n}(m)\right)^{k-1}V_{p,n}(m)^{-2k+2}\\
 & \qquad \times \mathrm{e}^{-2kE_{p,n}(m)/V_{p,n}(m)}\frac{V_{p,n}(m)!}{\left(V_{p,n}(m)-k\right)!} \cdot \left(\frac{k-1+2pG_{p,n}(m)}{n(n-1)}\right)\\
 &= \big(1+o_{(\Omega_{p,n}(m),E_{p,n}(m))}(1)\big) \cdot \frac{k^{k-1}}{k!}  \cdot  \left(\frac{2E_{p,n}(m)}{V_{p,n}(m)}\right)^{k-1} \cdot \mathrm{e}^{-2kE_{p,n}(m)/V_{p,n}(m)} \\
& \qquad \times \frac{\left(k-1+2pG_{p,n}(m)\right)\left(n-G_{p,n}(m)\right)}{n^2}
\end{align*}
uniformly for all $1\leq k\leq \epsilon(E_{p,n}(m))(E_{p,n}(m))^{1/2}$, where we used again Stirling's formula and that $V_{p,n}(m)=n-G_{p,n}(m)$.

2) The proof is similar. Observe that under the hypotheses we make here we have $\big\lvert\Omega_{p,n}^{(k)}(m)\big\rvert\leq c+o_{V_{p,n}(m)}(1)$, uniformly in $1\leq k\leq (V_{p,n}(m))^{1/2}$. 
 We can thus apply the asymptotics in the critical regime of Proposition \ref{lm:Britikov} to estimate both $\#\mathcal{W}\left(V_{p,n}(m)-k,E_{p,n}(m)-k+1\right)$ and \linebreak $\#\mathcal{W}\left(V_{p,n}(m),E_{p,n}(m)\right)$, and then plug them in the first identity of Proposition \ref{lm_transitions} to get the result.
 
3) When $\Omega_{p,n}(m)\to \infty$, again $~V_{p,n}(m)\to \infty~$ and $~\Omega_{p,n}^{(k)}(m)= \Omega_{p,n}(m)\big(1+o_{V_{p,n}(m)}(1)\big)$, uniformly in $1\leq k\leq \left(V_{p,n}(m)\right)^{1/2}$. So we now apply the asymptotics of the supercritical regime of Proposition \ref{lm:Britikov}, together with Proposition \ref{lm_transitions}, to get the expected result. Note that here one may have $E_{p,n}(m)=V_{p,n}(m)-1$, in which case the cardinal $\#\mathcal{W}(V_{p,n}(m)-k,E_{p,n}(m)-k+1)$ is null for each $k$, as well as the probability $\mathbb{P}\left( \Delta G_{p,n}(m)=k  | \mathbf{F}_{p,n}(m)\right)$. 
\end{proof}

\subsection{Estimates on the jumps of $G_{p,n}$}
\label{sec:largest_jump}

We will also need to control the jumps of $G_{p,n}$, which, by Proposition \ref{prop:freeforest}, are related to the sizes of the trees in uniform random forests $W(V_{p,n}(m),E_{p,n}(m))$. In that aim, we settle here estimates on the size $L_1(N,M)$  of the \emph{largest} tree in a uniform random forest  with $N$ vertices and $M$ edges, in the spirit of what has been done by Luczak-Pittel \cite{LuczakPittel92} and Bernikovich-Pavlov \cite{BernikovichPavlov11}. We refine slightly their results, relying on Proposition \ref{prop_forest_rw} and its notation. The first lemma below  concerns the subcritical regime, and the second the supercritical regime.

\begin{lemma}\label{lm:forest:sub}
Fix $\varepsilon>0$ and consider  a function $B$ such that $B(N)\to \infty $ as $N\to \infty$. Then for $N$ large enough and all $M$ verifying $2M-N\leq -\varepsilon N$,
$$\mathbb{P}\left(L_1(N,M)\geq B(N)\right)~\leq~\frac{C_{\varepsilon}\cdot N^{2}}{B(N)^2} \cdot \exp\left(-\frac{\varepsilon^2 B(N)}{2}\right)$$
where $C_{\varepsilon}\in (0,\infty)$ only depends on $\varepsilon$.
\end{lemma}

For the proof, we use the following local limit theorem stemming from Britikov \cite{britikov88}. Recall the notation $\mu_x$ (\ref{def:mu_x}) and $S^{(x)}$ from Proposition \ref{prop_forest_rw}, and for $\theta\in (0,1)$, set
$$m(\theta)=\frac{2}{2-\theta}\quad \text{and}\quad \sigma^2(\theta)=\frac{2\theta}{(1-\theta)(2-\theta)^2}.$$

\smallskip

\begin{lemma}[Britikov \cite{britikov88}, Lemma 5]\label{lm:Britikov_sub}
	Let $x=\theta\mathrm{e}^{-\theta}$ for some $\theta\in (0,1)$ which may depend on $N,M$. Assume that $N-M \rightarrow \infty$,
	$(N-M)\theta\to \infty$ and $(N-M)^{1/3}(1-\theta) \to \infty$. Then, if $z=\frac{N-m(\theta)(N-M)}{\sigma(\theta)(N-M)^{1/2}}$ lies in some finite interval, one has
	$$\sigma(\theta)(N-M)^{1/2}\mathbb{P}\left(S_{N-M}^{(x)}=N\right)=\left(1+o(1)\right)\frac{1}{\sqrt{2\pi}}\mathrm{e}^{-z^2/2}.$$
\end{lemma}

\textbf{Proof of Lemma \ref{lm:forest:sub}}.
Throughout the proof we consider $M\geq B(N)-1~$ since otherwise $\mathbb{P}(L_1(N,M)\geq B(N))$ is null and the statement is trivially true. 
By Proposition \ref{prop_forest_rw}, if $X_i^{(x)},i\geq 1,$ are i.i.d. random variables with law $\mu_x$, whatever $x \in (0,e^{-1}]$, one has 
\begin{align}
		\mathbb{P}\left(L_1(N,M)\geq B(N)\right)&=\sum_{r=B(N)}^{M+1}\mathbb{P}\left(\max_{1\leq i\leq N-M} X_i^{(x)}=r~\Big\vert ~
		 S_{N-M}^{(x)}=N\right) \notag\\
		 &\leq (N-M)\sum_{r=B(N)}^{M+1}\frac{\mu_x(r) \cdot \mathbb{P}\left(S_{N-M-1}^{(x)}=N-r\right)}{\mathbb{P}\left(S_{N-M}^{(x)}=N\right)}\notag\\
		 &\leq \frac{N-M}{\mathbb{P}\Big(S_{N-M}^{(x)}=N\Big)} \cdot \sum_{r=B(N)}^{M+1}\mu_x(r).
		 \label{eq:bornesup_foret}
\end{align} 

Following Britikov  \cite{britikov88}, the strategy is then to choose wisely $x\in (0,\mathrm{e}^{-1}]$. We take here $x=\theta e^{-\theta}$ with $\theta=2M/N$. Note that under our hypotheses on $M$, one has $\theta\in (0,1)$ and then $m(\theta)=N/(N-M)$.  And also,
$(N-M)\theta \geq M\to \infty$ (since $M\geq B(N)-1$) and $\left(N-M\right)^{1/3}\left(1-\theta\right)\geq  \varepsilon\left(N-M\right)^{1/3}\to \infty$ as $N \rightarrow \infty$.
Lemma \ref{lm:Britikov_sub} thus yields that for  $N$ large enough and every $M \geq B(N)-1$ verifying $2M-N\leq -\varepsilon N$ 
\begin{equation*}\label{eq:equiv_marche}
	\mathbb{P}\left(S_{N-M}^{(x)}=N\right)\geq\frac{1}{2\sqrt{2\pi}}\cdot \frac{(N-M)^{1/2}(N-2M)^{1/2}}{M^{1/2}N}.
\end{equation*}   
Plugging this bound in \eqref{eq:bornesup_foret}, and using again that $2M-N\leq -\varepsilon N$, we get
\begin{align*}
\mathbb{P}\left( L_1(N,M)\geq B(N)\right)&\leq C_1\cdot N M^{1/2}
\sum_{r=B(N)}^{M+1}\frac{r^{r-2}}{r!} \cdot \frac{x^r}{T(x)}\\
&\leq C_1\cdot N M^{1/2} \sum_{r=B(N)}^{M+1}\frac{r^{r-2}}{r!}\left(\frac{2M}{N}\right)^r \mathrm{e}^{-\frac{2Mr}{N}} \cdot \frac{N^2}{2M(N-M)},
\end{align*}
with $C_1\in (0,\infty)$ depending only on $\varepsilon$. We used Lemma \ref{lm:lambert_fct} to obtain $T(x)$. Next, by Stirling's formula, still under our hypotheses on $M$, this leads to 
\begin{align*}
	\mathbb{P}\left(L_1(N,M) \geq B(N)\right) &\leq \frac{C_2\cdot N^{2}}{B(N)^{1/2}} \sum_{r=B(N)}^{M+1} r^{-5/2}\mathrm{e}^{r}\left(\frac{2M}{N}\right)^r \mathrm{e}^{-\frac{2Mr}{N}} \\
	&\leq  \frac{C_2\cdot N^{2}}{B(N)^{1/2}} \sum_{r=B(N)}^{M+1} r^{-5/2} e^{-r(1-2M/N)^2/2}
\end{align*}
for some $C_2 \in (0,\infty)$ depending only on $\varepsilon$, and all $N$ large enough, where for the second inequality we used that $\mathrm{e}^x(1-x)\leq \mathrm{e}^{-x^2/2}$ for $x\in [0,1]$.  
Finally we get 
\begin{align*}
\mathbb{P}\left(L_1(N,M) \geq B(N)\right) &\leq \frac{C_2\cdot N^{2}}{B(N)^{1/2}}\exp\left(-\left(1-\frac{2M}{N}\right)^2\frac{B(N)}{2}\right)\sum_{r=B(N)}^{M}r^{-5/2}
\end{align*}
which leads to the upper bound of the statement. $\hfill \square$

\bigskip

\begin{lemma}\label{lm:forest:super}
Consider two functions $\omega,g$ such that, as $N \rightarrow \infty$, $\omega(N)\to \infty $, $\omega(N)=o\big(N^{1/3}\big)$, $g(N)\to \infty$ and $g(N)=o\left(\omega(N)\right).$ Then, for $N$ large enough and every $M$ verifying $2M-N=\omega(N)N^{2/3}$,
$$\mathbb{P}\big(L_1(N,M)<g(N)N^{2/3}\big)\leq C\cdot \omega(N)^{5/2}N^{2/3} \cdot \exp\left(-\frac{\omega(N)}{g(N)}\right),$$
for some $C\in (0,\infty)$ independent of $N$.
	\end{lemma} 

\begin{proof} We use here Proposition \ref{prop_forest_rw} with the measure $\mu_{\mathrm{e}^{-1}}$ and let $X_i,i\geq 1$, be i.i.d. random variables with law $\mu_{\mathrm{e}^{-1}}$. The proof is inspired by \cite{BernikovichPavlov11} for similar results in the case of unlabelled forests.  We recall that the expectation of $\mu_{\mathrm{e}^{-1}}$ is equal to 2 and introduce the centered random variables $Y_i=X_i-2$, $i\geq 1$, as well as $\tilde{S}_N=\sum_{i=1}^{N}Y_i  $ and
$ \tilde{S}_N^{(r-2)}=\sum_{i=1}^N Y_i^{(r-2)} $, with $Y_i^{(r-2)}=Y_i\mathbbm{1}_{\{Y_i\leq r-2\}}$, for $N,r \geq 1$. From Proposition \ref{prop_forest_rw}, for any $r\in \mathbb N$:
\begin{align*}
\mathbb{P}\left(L_1(N,M)\leq r\right)&=\frac{\mathbb{P}\left(\max_{i\leq N-M}Y_i\leq r-2,\tilde{S}_{N-M}=\omega(N)N^{2/3}\right)}{\mathbb{P}\left(\tilde{S}_{N-M}=\omega(N)N^{2/3}\right) }\\
&\leq~\frac{\mathbb{P}\left(\tilde{S}_{N-M}^{(r-2)}=\omega(N)N^{2/3}\right)}{\mathbb{P}\left(\tilde{S}_{N-M}=\omega(N)N^{2/3}\right)}.
\end{align*}
The measure $\mu_{\mathrm{e}^{-1}}$ being in the domain of attraction of a $3/2$-stable law, the local limit theorem (see e.g. \cite{GnedenkoKolmogorov54}) yields, under the assumption $2M-N=\omega(N)N^{2/3}$ with $\omega(N)\to \infty $ and $\omega(N)=o\big(N^{1/3}\big)$,
\begin{equation}
\label{eq:forest}
\mathbb{P}\left(\tilde{S}_{N-M}=\omega(N)N^{2/3}\right)=\frac{\left(1+o(1)\right)}{\sqrt{2\pi}} \cdot \omega(N)^{-5/2}N^{-2/3}.
\end{equation}
We now want to get an upper bound for $\mathbb{P}\big(\tilde{S}_{N-M}^{(r-2)}=\omega(N)N^{2/3}\big)$ when $r=g(N)N^{2/3}$ with $g(N) \rightarrow \infty$ and $g(N)=o(\omega(N))$. Since $\big\vert r^{-1} Y_1^{(r-2)}\big \vert \leq 1$ and $\mathrm{e}^x\leq 1+x+x^2$ for all $x \in [-1,1]$, we have that
\begin{align*}
\mathbb{P}\left(\tilde{S}_{N-M}^{(r-2)}=\omega(N)N^{2/3}\right)&\leq \mathrm{e}^{-r^{-1} \omega(N)N^{2/3}}\mathbb{E}\left[\mathrm{e}^{r^{-1} Y_1^{(r-2)}} \right]^{N-M}\\
&\leq  \mathrm{e}^{-r^{-1} \omega(N)N^{2/3}}\left(\mathbb{E}\left[1+r^{-1} Y_1^{(r-2)}+\left(r^{-1} Y_1^{(r-2)}\right)^2\right]\right)^{N-M}.
\end{align*}
The distribution of $Y_1$ yields the existence of $a,b \in (0,\infty)$ such that for every $r$ large enough
$$\mathbb{E}\left[Y_1^{(r-2)}\right]\leq \frac{-a}{\sqrt{r}} \quad \text{ and }\quad \mathbb{E}\Big[\left(Y_1^{(r-2)}\right)^2\Big]\leq b\sqrt{r}$$ 
which then leads to
\begin{align*}
	\mathbb{P}\left(\tilde{S}_{N-M}^{(r-2)}=\omega(N)N^{2/3}\right)&\leq \mathrm{e}^{-r^{-1} \omega(N)N^{2/3}}\left(1+br^{-3/2}\right)^{N-M}\\
	&\leq \mathrm{e}^{-\omega(N)/g(N)}\mathrm{e}^{bNr^{-3/2}/2}.
\end{align*}
Together with (\ref{eq:forest}) and since $N r^{-3/2}=g(N)^{-3/2}\to 0$ as $N\to \infty$, we get the expected upper bound for $\mathbb{P}\big(L_1(N,M)\leq g(N)N^{2/3}\big)$.
\end{proof}

\section{Properties of the fluid limit functions}
\label{sec:fluid}

In the Introduction, the function $g_p:[0,\infty)\rightarrow [0,1)$, which will describe the fluid limit of the gel in the $p$-frozen model, was defined on $ [1/2,\infty)$ in Definition \ref{def:gel_function} as the inverse of the function $f_p:[0,1) \rightarrow [1/2,\infty)$ given for $t \in [0,1)$ by
\begin{equation}
\label{def:f_p}
f_p(t)~=~\frac{1}{2}+\frac{t}{2p}\int_{0}^1 \frac{u^{\frac{1}{p}}}{1-tu} \mathrm du~=~\frac{1}{2}\sum_{n=0}^{\infty}\frac{t^n}{1+pn}
\end{equation}
and for $t \in [0,1/2]$ by $g_p(t)=0$. All other functions $d_p,v_p,e_p,r_t$ and $t_{p,k},k\geq 1$ were defined from this function $g_p$. We propose in Section \ref{sub:functions} an alternative definition of $g_p$ and of the couple $(g_p,d_p)$ as solutions to a (system of) differential equation(s) and develop several properties of these functions. One difficulty when we will implement in the next section the \emph{differential equation method}  to determine the fluid limits of the processes $G_{p,n}$ and $D_{p,n}$ is that the differential equations characterizing $g_p$ and $(g_p,d_p)$ are \emph{not Lipschitz} in the neighborhood of $t=1/2$, which is the source of technical difficulties in the proof of Theorem \ref{thm:fluid limit}. For this reason, we need to approximate these differential equations by smoother ones, which is done in Section \ref{section:approximation}.  Last, Section \ref{sec:t_k} is devoted to a system of differential equations satisfied by the functions $t_{p,k},k\geq 1$.

\subsection{The functions $g_p,d_p,v_p,e_p$ and $r_p$}
\label{sub:functions}

We will see in the forthcoming Lemma \ref{lm:equadiff} and Lemma \ref{lm:syst_edo} (these lemmas are proved in a more general context and therefore postponed to the next section) that the function $g_p:[1/2,\infty)\rightarrow [0,1)$ is the unique strictly increasing solution to the equation 
\begin{equation}
\tag{$E_{(0)}$}
\label{eq:f}
	g'(t)=\frac{2pg(t)\left(1-g(t)\right)}{1-2t \left(1-g(t)\right)}, \quad t> 1/2, \quad \text{with} \quad g(1/2)=0,
\end{equation}
and that there is a unique solution of strictly increasing functions to the system of equations
\begin{equation}
\tag{$\tilde E_{(0)}$}
\label{eq:syst_EDO_0}
\begin{array}{l}	
	\left\{
	\begin{array}{ll}
		g'(t)=\displaystyle \frac{2pg(t)\left(1-g(t)\right)^2}{1-2t+g(t)+2d(t)} \\
		d'(t)=2(1-p)g(t)\left(1-g(t)\right)+g(t)^2
	\end{array},
	\quad t>\frac{1}{2}, \quad  \text{ with } g(1/2)=d(1/2)=0
	\right.
	\end{array}
\end{equation}
which is denoted $\big(g_p,d_p\big)$, with $g_p$ as above. There is no conflict of notation here, since Proposition \ref{prop:propfunctions} below shows that $d_p$ indeed corresponds to its definition via $g_p$ in (\ref{def:d_p}).

Given the relations (\ref{def:V_{p,n}E_{p,n}}), the functions $v_p,e_p, r_p$ are defined for $t \geq 0$ by
$$
v_p(t)=1-g_p(t); \quad  \quad e_p(t)=t-g_p(t)-d_p(t); \quad  \quad  r_p(t)=\frac{e_p(t)}{v_p(t)},
$$
and, again, Proposition \ref{prop:propfunctions} shows that this corresponds to their definitions in the statement of Theorem \ref{thm:forest}.

\begin{proposition}
\label{prop:propfunctions}
\begin{enumerate}
\item[\emph{1)}] For $t \in [0,1/2]$, $g_p(t)=d_p(t)=0$, $v_p(t)=1$, $e_p(t)=r_p(t)=t$.
\item[\emph{2)}] The functions $g_p,d_p,v_p,e_p,r_p$ are infinitely differentiable on $(1/2,\infty)$, with \linebreak $g_p'(1/2^+)=2(1+p)=-v'_p(1/2^+)$,  $~  d_p'(1/2^+)=0$, $~  e_p'(1/2^+)=-1-2p$, $~  r_p'(1/2^+)=-p$.  
\item[\emph{3)}] As $t \rightarrow \infty$,  $~ 1-g_p(t)=v_p(t) \sim e^{-2pt}$, $~ d_p(t)-t+1 \sim e^{-2pt}$, $~ e_p(t) \sim te^{-4pt}$, $~ r_p(t)  \sim te^{-2pt}$.
\item[\emph{4)}] The ratio function rewrites $r_p(t)=t(1-g_p(t))$, and therefore $e_p(t)=t(1-g_p(t))^2$ and \newline $d_p(t)=t-g_p(t)-t(1-g_p(t))^2$, for $t\geq 0$.
\item[\emph{5)}] While the functions $g_p,d_p,v_p$ are monotonic on $(0,\infty)$ ($g_p$ and $d_p$ are increasing, $v_p$ is decreasing), the functions $r_p$ and $e_p$ are increasing on $(0,1/2]$ and decreasing on $[1/2,\infty)$. In particular, $r_p(t)<1/2$ for $t\neq 1/2$.
\item[\emph{6)}] The function $g_p$ is concave on $[0,\infty)$. The fonction $d_p$ is convex on $[0,\infty)$.
\end{enumerate}
\end{proposition}

\begin{proof} Most assertions of this corollary are easy to check by using the differential equations defining $g_p$ and $d_p$ and the relations between the different functions. We leave their proof to the reader. We wish however to point out that the identity $r_p(t)=t(1-g_p(t))$, $t\geq 0$ stated in 4) is shown in the proof of the forthcoming Lemma \ref{lm:syst_edo}, and we detail here the two following points: 

5) The function $r_p$ is decreasing on $[1/2,\infty)$ (note that this implies that $e_p=r_pv_p$ is also decreasing on $[1/2,\infty)$). Indeed, to see this use that $r_p(t)=t(1-g_p(t))$ and note that this function is decreasing on $[1/2,\infty)$ if and only if $t\mapsto f_p(t)(1-t)$ is decreasing on $[0,1)$. Using the series representation (\ref{def:f_p}) of $f_p$, we get that
$$
f_p(t)(1-t)=\frac{1}{2}-\frac{p}{2}  \sum_{n=0}^{\infty} \frac{t^{n+1}}{(1+pn)(1+p(n+1))}
$$
which is clearly decreasing. 

6) The concavity of $g_p$ is a consequence of the convexity of $f_p$ on $[0,1)$, which is an immediate consequence of the series representation of $f_p$. To see the convexity of $d_p$, note that \linebreak $d''_p=2g_p'(1-p+2pg_p)$ on $[1/2,\infty)$ (and 0 otherwise), which is positive.
\end{proof}

Next, in order to establish the asymptotic behavior of the absorption times $A^{(k)}_{p,n}$, $A^{(k+)}_{p,n}$ stated in Theorem \ref{thm:absorption}, we also emphasize the following identity.

\begin{lemma}
\label{lem:identintegral}
Recalling that $\gamma_E$ denotes Euler's constant and $\psi$ the digamma function, defined by $\psi(x)=\frac{\Gamma'(x)}{\Gamma(x)}$ with $\Gamma$ the gamma function, we have
$$
(1-p)\int_0^{\infty} (1-g_p(t)) \mathrm dt ~ = ~ \frac{\psi(1/p) +\gamma_E}{2}.
$$
The function $\psi(1/p) +\gamma_E$ is decreasing in $p$, equal to 0 when $p=1$ and to $1$ when $p=1/2$, and $\psi(1/p) \sim -\ln p$ when $p \rightarrow 0$.
\end{lemma}

\begin{proof} Using that $g_p(t)=0$ for $t\in [0,1/2]$ and that $(g_p(t))_{t\geq 1/2}$ is solution to (\ref{eq:f}), we have
\begin{eqnarray*}
\int_0^{\infty} (1-g_p(t)) \mathrm dt&=&\frac{1}{2}+\int_{1/2}^{\infty} \frac{1-2t(1-g_p(t))}{2pg_p(t)} g_p'(t) \mathrm dt \\
&\underset{s=g_p(t)}=& \frac{1}{2}+ \int_0^1  \frac{1-2f_p(s)(1-s)}{2ps} \mathrm ds \\
&=& \frac{1}{2}+\frac{1}{2p} \int_0^1 \left(1-  \sum_{n=1}^{\infty} \frac{s^{n-1}(1-s)}{1+pn} \right)\mathrm ds \\
&=&  \frac{1}{2}+\frac{1}{2p}- \frac{1}{2p} \left(\sum_{n=1}^{\infty} \frac{1}{(1+pn)n}-\sum_{n=1}^{\infty} \frac{1}{(1+pn)(n+1)} \right).
\end{eqnarray*}

We then use that 
$$
\psi(x+1)=-\gamma_E+\sum_{n=1}^{\infty} \frac{x}{n(n+x)} \quad \text{ for }x\geq 0,
$$
(see e.g. \cite{spouge94}) and that $~\psi(x+1)=\psi(x)+1/x~$ for $x>0$ (a trivial consequence of the relation $\Gamma(x+1)=x\Gamma(x)$) to get
\begin{eqnarray*}
\int_0^1 (1-g_p(t)) \mathrm dt&=& \frac{1}{2}+\frac{1}{2p}\left(\frac{1}{1-p}\left(\psi(1/p) +\gamma_E\right)-\left(\psi(1/p)+\gamma_E+p \right) \right) \\
&=&\frac{\psi(1/p)+\gamma_E}{2(1-p)}.
\end{eqnarray*}
Last, the above series representation of $\psi$ shows that it is increasing, with $\psi(1)+\gamma_E=0$ and $\psi(2)+\gamma_E=1$. Moreover $\psi(x) \sim \ln(x)$ when $x \rightarrow \infty$, which gives the asymptotic behavior of the integral when $p\rightarrow 0$.
\end{proof}

\subsection{Approximation}
\label{section:approximation}

In order to introduce more smoothness, we consider the following equation for any $\varepsilon \geq 0$, generalizing thus (\ref{eq:f}): 
\begin{equation}
\tag{$E_{(\varepsilon)}$}
\label{def:edo_approx}
	g'(t)=\frac{2pg(t)\left(1-g(t)\right)}{1-2(t+\varepsilon)\left(1-g(t)\right)}, \quad t>1/2.
\end{equation}

For $a\in [0,1)$, an increasing continuously differentiable function $g:\left[1/2,\infty\right)\to [a,1)$ such that $g\left(1/2\right)=a$ and $g$ verifies $(E_{(\varepsilon)})$ is called a solution to $(E_{(\varepsilon)})$ starting from $a$. This implies in particular that $g(t)\to 1$ as $t\to \infty$ and that $1-2(t+\varepsilon)(1-g(t))>0$ for all $t>1/2$.

\begin{lemma}\label{lm:equadiff}
\begin{enumerate}[topsep=0cm]
\item[\emph{1)}] For $\varepsilon \geq 0$ and $a\in \left(\frac{2\varepsilon}{1+2\varepsilon},1\right)$,  there exists a unique solution to $(E_{(\varepsilon)})$ starting from $a$. We denote it here by $g_p^{(\varepsilon,a)}$.
\item[\emph{2)}] Let $\frac{2\varepsilon}{1+2\varepsilon}<a<b<1$, then $~g_p^{(\varepsilon,a)}(t)<g_p^{(\varepsilon,b)}(t)$ for all $t \geq \frac{1}{2}$ and 
$$\sup_{t\geq {1}/{2}}\big\vert g_p^{(\varepsilon,a)}(t)-g_p^{(\varepsilon,b)}(t)\big\vert ~\leq ~\vert b-a\vert.$$

\vspace{-0.4cm}

\item[\emph{3)}] Take $0<\varepsilon_1<\varepsilon_2$ and $a>\frac{2\varepsilon_2}{1+2\varepsilon_2}$. Then $g^{(\varepsilon_1,a)}_{p}(t)< g^{(\varepsilon_2,a)}_{p}(t)$ for all $t\geq \frac{1}{2}$.
\item[\emph{4)}] There exists a unique solution to $(E_{(0)})$ starting from $0$, which is our function $g_p$ defined as the inverse of the function $f_p$ (\ref{def:f_p}). In particular $g_{p}(\cdot+\varepsilon)$ is the solution to $(E_{(\varepsilon)})$ starting from $g_{p}({1}/{2}+\varepsilon)$.
\end{enumerate}
\end{lemma}
\begin{proof}
	1) We could use the Cauchy-Lipschitz theorem but prefer to give here a direct proof "by hands", that gives explicitly the inverse of $g_p^{(\varepsilon,a)}$ and adapts immediately to prove the point 4) - for which Cauchy-Lipschitz does not apply. Assume that $g_p^{(\varepsilon,a)}$ exists and let $f_p^{(\varepsilon,a)}:[a,1)\to \left[1/2,\infty\right)$ denotes its inverse. Then $f_p^{(\varepsilon,a)}$ is solution to the linear differential equation
	$$f'(t)=\frac{1-2(f(t)+\varepsilon)(1-t)}{2pt(1-t)}, \quad t\in [a,1)$$ 
and one easily checks that it writes
$$f_p^{(\varepsilon,a)}(t)=\frac{1}{2}-\varepsilon + \frac{\varepsilon a^{1/p}}{t^{1/p}}+\frac{1}{2pt^{1/p}}\int_{a}^t  \frac{u^{1/p}}{1-u}\mathrm{d}u.$$
This shows that $g_p^{(\varepsilon,a)}$ is uniquely determined, if it exists. Its existence will be proved if we show that $f_p^{(\varepsilon,a)}$ is strictly monotone, that is $\big(f_p^{(\varepsilon,a)}(t)+\varepsilon \big)(1-t)<1/2$ for all $t\in [a,1)$. In that aim, note that
\begin{align*}
	f_p^{(\varepsilon,a)}(t)+\varepsilon&=\frac{1}{2}+ \frac{\varepsilon a^{1/p}}{t^{1/p}}+\frac{t}{2p}\int_{a/t}^1 \frac{u^{1/p}}{1-tu}\mathrm{d}u\\
	&\leq \frac{1}{2}+ \frac{\varepsilon a^{1/p}}{t^{1/p}}+\frac{t}{2p(1-t)}\int_{a/t}^1 u^{1/p}\mathrm{d}u\\
	&\leq \frac{1}{2}+ \frac{\varepsilon a^{1/p}}{t^{1/p}}+\frac{t}{2p(1-t)} \cdot \frac{1-\left(a/t\right)^{1+\frac{1}{p}}}{1+1/p}
\end{align*} 
and then
\begin{align*}
	\big(f_p^{(\varepsilon,a)}(t)+\varepsilon \big)(1-t)&\leq \frac{1-t}{2}+\varepsilon a^{1/p}t^{-1/p}(1-t)+\frac{t}{2(p+1)}\big(1-\left(a/t\right)^{1+\frac{1}{p}}\big)\\
	&=\frac{1}{2}+\frac{h(t)}{t^{1/p}}
\end{align*}      
with 
$$h(t)=\varepsilon a^{1/p}(1-t)-\frac{pt^{1+\frac{1}{p}}}{2(p+1)}-\frac{a^{1+\frac{1}{p}}}{2(p+1)}$$
a decreasing function on $[a,1)$. Consequently, for every $t\in [a,1)$ we have $h(t)\leq h(a)$ \linebreak $=a^{1/p}\left(\varepsilon-a\left(\varepsilon +1/2\right)\right) $ and thus $\big(f^{(\varepsilon,a)}_p(t)+\varepsilon\big)(1-t)<1/2 $ as soon as $a>\frac{2\varepsilon}{1+2\varepsilon}$.   

2) The function $f^{(\varepsilon,a)}_p$ defined above is in fact well-defined for all $t \in (0,1)$, and for a fixed $t$, $a \in \left(\frac{2\varepsilon}{1+2\varepsilon},1 \right) \mapsto f^{(\varepsilon,a)}_p(t)$ is decreasing. This implies that $g_p^{(\varepsilon,a)}(t)<g_p^{(\varepsilon,b)}(t)$ for $t \geq 1/2$. 
Then note that for $s>1/2$, the function
$$
x \in  \left(1-\frac{1}{2s},1\right)\mapsto G(s,x):=\frac{2px(1-x)}{1-2s(1-x)}
$$           
 is positive, decreasing. Finally write for $t>1/2$
\begin{eqnarray*}
0 \leq g_p^{(\varepsilon,b)}(t) -g_p^{(\varepsilon,a)}(t) &=& b-a+\int_{1/2}^{t}  G(s,g_p^{(\varepsilon,b)}(s))-G(s,g_p^{(\varepsilon,a)}(s)) ~ \mathrm ds \\
&\leq & b-a.
\end{eqnarray*}
	
3) Here we just use that for any $t \geq a$, the function,
$$
\varepsilon \in (0,\infty) \mapsto \frac{1}{2}-\varepsilon + \frac{\varepsilon a^{1/p}}{t^{1/p}}+\frac{1}{2pt^{1/p}}\int_{a}^t  \frac{u^{1/p}}{1-u}\mathrm{d}u
$$
is decreasing.

4) We proceed as in point 1) and show similarly that $f_p(t)(1-t)<1/2$ for every $t\in (0,1)$ which leads to the result.
\end{proof}

We now turn to an approximation of (\ref{eq:syst_EDO_0}). For $\varepsilon\geq 0$ consider
\begin{equation}
\tag{$\tilde{E}_{(\varepsilon)}$}
\label{eq:syst_EDO}
	\left\{
	\begin{array}{ll}
		g'(t)=\frac{2pg(t)\left(1-g(t)\right)^2}{1-2(t+\varepsilon)+g(t)+2d(t)},\\
		d'(t)=2(1-p)g(t)\left(1-g(t)\right)+g(t)^2
	\end{array}
	\quad t>\frac{1}{2}.
	\right.
\end{equation}
We call solution to $(\tilde E_{(\varepsilon)})$ starting from $(a,b)\in [0,1)\times [0,\infty)$, a couple of strictly increasing continuously differentiable functions $(g,d):\left[1/2,\infty\right) \to \left[a,1\right)\times [b,\infty)$ such that $g\left({1}/{2}\right)=a$, $d\left({1}/{2}\right)=b$ and $(g,d)$ verifies $(\tilde{E}_{(\varepsilon)})$, which implicitly means that $1-2(t+\varepsilon)+g(t)+2d(t)>0$ for every $t>1/2$.

For $\varepsilon\geq 0$ and $\left(a,b\right)\in [0,1)\times [0,\infty)$, we set
$$\delta_{(a,b)}(\varepsilon)=\frac{\varepsilon-b-a^2/2}{\left(1-a\right)^2}$$
(which may be negative).

\begin{lemma}
\label{lm:syst_edo}
Consider $\varepsilon\geq0$ and $(a,b)\in [0,1)\times [0,\infty)$ such that $\delta_{(a,b)}(\varepsilon)\geq 0$.
\begin{enumerate}[topsep=0cm]
\item[\emph{1)}] If $(g,d)$ is a solution to $(\tilde{E}_{(\varepsilon)})$ starting from $(a,b)$ then $g$ is a solution to $(E_{(\delta_{(a,b)}(\varepsilon))})$ starting from $a$.
\item[\emph{2)}] If either $a>\frac{2\delta_{(a,b)}(\varepsilon)}{1+2\delta_{(a,b)}(\varepsilon)}$ or $a=0$ and $b=\varepsilon$, there exists a unique solution to $(\tilde{E}_{(\varepsilon)})$ starting from $\left(a,b\right)$. In particular, there exists a unique solution to
$(\tilde{E}_{(0)})$ starting from $(0,0)$, denoted by $(g_p,d_p)$, with $g_p$ the inverse of $f_p$ (\ref{def:f_p}). 
\end{enumerate}
\end{lemma}
\begin{proof}
1) Consider $(g,d)$ a solution to $(\tilde{E}_{(\varepsilon)})$ starting from $(a,b)$ and set
$$
 r(t)=\frac{t+\varepsilon-g(t)-d(t)}{1-g(t)}, \quad t \geq 1/2.
$$
The first part of  \eqref{eq:syst_EDO} rewrites for $t>1/2$
	$$g'(t)=\frac{2pg(t)(1-g(t))}{1-2r(t)}.$$  Our goal is to prove that $r(t)=(t+\delta_{(a,b)}(\varepsilon))(1-g(t))$ for every $t \geq 1/2$ which will yield the claim. In that aim set for $t\geq 1/2$, $$h(t)=r(t)-(t+\delta_{(a,b)}(\varepsilon))(1-g(t)).$$ By definition of $\delta_{(a,b)}(\varepsilon)$, $h\left(1/2\right)=0$. Then, for $t > \frac{1}{2}$
	\begin{align*}
		h'(t)&=\frac{1-g'(t)-d'(t)}{1-g(t)}+\frac{g'(t)}{1-g(t)}r(t)-\left(1-g(t)\right)+\big(t+\delta_{(a,b)}(\varepsilon)\big)g'(t)\\
		&=-h(t)\frac{g'(t)}{1-g(t)}
	\end{align*} 
	which implies that $h$ is identically zero. 
	
2) When $a>2\delta_{(a,b)}(\varepsilon)/(1+2\delta_{(a,b)}(\varepsilon))$ or $a=\delta_{(a,b)}(\varepsilon)=0$, Lemma \ref{lm:equadiff} gives the existence of a solution $g$ to $(E_{(\delta_{(a,b)}(\varepsilon))})$ starting from $a$. Defining then $d$ from $g$ by $$d(t)=b+\int_{1/2}^{t}\left(2(1-p)g(s)\left(1-g(s)\right)+g^2(s)\right)\,\mathrm{d}s, \quad t\geq \frac{1}{2},$$
one sees that $(g,d)$ is a solution $(\tilde{E}_{(\varepsilon)})$ starting from $(a,b)$ (using the same strategy as above with the ratio function $r$).

To prove the uniqueness, we use 1) together with Lemma \ref{lm:equadiff} which gives the uniqueness of a solution $g$ to $(E_{(\delta_{(a,b)}(\varepsilon))})$ starting from $a$ (this function being $g_p$ when $a=\delta_{(a,b)}(\varepsilon)=0$). The function $d$ is then uniquely determined from $g$.
\end{proof}

\subsection{The functions $t_{p,k}$}
\label{sec:t_k}

We now turn to the functions $t_{p,k}$ arising as the scaling limits of the number of trees of size $k$, $k\geq 1$, which are defined from the function $g_p$ by
$$
t_{p,k}(t)=\frac{k^{k-2}}{k!}\left(2t\right)^{k-1}\left(1-g_p(t)\right)^k \mathrm{e}^{-2kt\left(1-g_p(t)\right)}.
$$

\bigskip

\begin{proposition}
The sequence of functions $(t_{p,k},k\geq 1)$ is the unique solution to the following system of differential equations:
\begin{equation*}
\label{eq_lim_trees}
	\left\{
	\begin{array}{ll}
		t_k'(t)=\sum_{i+j=k}ijt_i(t)t_j(t)-2kt_k(t)\left(1-(1-p)g_p(t)\right),  \quad t\geq 0, \\
		t_1(0)=1; \quad t_k(0)=0 \quad \text{for }k\geq 2.
	\end{array}
	\right.
\end{equation*}
\end{proposition}

\begin{proof}
For the uniqueness of solutions to this system, note that if $(t_k,k\geq 1)$ is a solution to this equation, then $t_1$ is the solution to a linear differential equation of the first order, and so it is uniquely determined by its initial condition $t_1(0)=1$. Then we proceed by induction on $k$, noticing that, given the functions $t_i, i \leq k-1$, the function $t_k$ is also solution to a linear differential equation of the first order and so is uniquely determined by its value at $t=0$.

To prove the existence, we just have to check that the functions $t_{p,k}, k\geq 1$ are solutions. Regarding the initial conditions, we clearly have that $t_{p,1}(0)=1$ since $g_p(0)=0$, and $t_{p,k}(0)=0$ for $k\geq 2$. Then, setting $u_k(t)=t_k(t)/(t_1(t))^k$ for $k \geq 1$, we note that the system of equations rewrites
$$\left\{
	\begin{array}{ll}
	        t_1'(t)=-2t_1(t)(1-(1-p)g_p(t)), \quad t_1(0)=1 \\
		u_k'(t)=\sum_{i+j=k}iju_i(t)u_j(t), \quad u_k(0)=0 \text{ for }k\geq 2
	\end{array}
	\right. \qquad t\geq 0.$$
We immediately see from the definition of $t_{p,1}$ and (\ref{eq:f}) that $t_{p,1}'(t)=-2t_{p,1}(t)(1-(1-p)g_p(t))$, $\forall t \geq 0$. Next, consider the functions
$$
u_{p,k}(t):=\frac{t_{p,k}(t)}{(t_{p,1}(t))^k}=\frac{k^{k-2}}{k!}\left(2t\right)^{k-1}, \quad t \geq 0
$$
and write for $k\geq 2$
\begin{eqnarray*}
\sum_{i+j=k}iju_{p,i}(t)u_{p,j}(t)&=& (2t)^{k-2} \cdot \sum_{i=1}^{k-1}\frac{i^{i-1}}{i!} \cdot \frac{(k-i)^{k-i-1}}{(k-i)!}  \\
&\underset{(\text{by }(\ref{eq:ijk}))}=& \frac{2 k^{k-3}}{(k-2)!} \cdot (2t)^{k-2}.
\end{eqnarray*}
The left-hand side is equal to $u'_{p,k}(t)$, which shows that the functions $t_{p,k},k\geq 1$ are indeed solutions to the system of differential equations of the statement of the proposition.
\end{proof}

\section{Convergence to the fluid limit}
\label{sec:mise_en_place}	
	
This section is devoted to the proof of Theorem \ref{thm:fluid limit} and the ensuing results Theorem \ref{thm:forest}, Corollary \ref{prop:cc} and Corollary \ref{cor:CC1}. As announced in the Introduction, we will use the differential equation method as developed by Wormald (see Theorem \ref{thmDEM} in the Appendix) for processes with relatively small one-step jumps, which are approximated by sufficiently smooth functions. To this end, we start by implementing preliminary results in Section \ref{sec:prelimSection4}. Among other things, we recall there that the processes $G_{p,n}$ and $D_{p,n}$ are respectively of order $n^{2/3}$ and $n^{1/3}$ at time $\lfloor n/2 \rfloor$. So, immediately, these processes divided by $n$ and accelerated in time by a factor $n$ converge to 0 uniformly on the interval $[0,1/2]$, and in the rest of the section we can focus on the interval $[1/2,\infty$). Since the differential equations (\ref{eq:f}) and (\ref{eq:syst_EDO_0}) involving $g_p$ and $d_p$ do not satisfy the necessary Lipschitz assumptions at the critical point $t=1/2$, Wormald's theorem cannot be applied directly around that point. We will therefore start by proving the fluid limit for the processes beyond time $(1/2+\varepsilon)n$, for $\varepsilon>0$, see Section \ref{sec:DEM_eps}, after having setting up suitable estimates for the process $G_{p,n}$ at time $(1/2+\varepsilon)n$ in Section \ref{section:starting_pt}.  We will then proceed by approximation, letting $\varepsilon$ tends to $0$, to conclude the proof of Theorem \ref{thm:fluid limit} in Section \ref{sec:proofFL}. Theorem \ref{thm:forest} and Corollaries \ref{prop:cc} and \ref{cor:CC1}, are then proved in Sections \ref{sec:proofF} and \ref{sec:typ_L_trees} respectively.

We recall that $\mathbf{F}_{p,n}$ denotes the filtration generated by the process $(G_{p,n},D_{p,n})$.
	
	\subsection{Preliminaries}
	\label{sec:prelimSection4}
	
	\subsubsection{Approximating a process via its conditional jumps}
	
	A key point underlying our proofs and the differential equation method is to approximate a sequence of processes  with \emph{small} (in $n$) \emph{variations} by its conditional jumps. Formally, we will need the following consequence of Azuma-Hoeffding inequality.
	
	\begin{lemma}\label{lm:approx:esp}
	Let $(u_n)_{n\in \mathbb N}$ be a deterministic sequence of positive real numbers. For each $n \in \mathbb N$, let $\left(Y_n(m)\right)_{m \in \mathbb Z_+}$ be a stochastic process starting from $Y_n(0)=0$ and such that $\left\vert \Delta Y_n(m)\right\vert \leq u_n~$ for every $m \in \mathbb Z_+$. Let then $(\mathbf{G}_n(m))_{m \in \mathbb Z_+}$ denote the filtration generated by the process $Y_n$, and consider $T_n$ a stopping time such that $T_n\leq An$ almost surely, for some deterministic $A>0$ independent of $n$. Then, for all $\varepsilon>0$,
	$$
\mathbb P\Bigg(\bigg|Y_n\left(\lfloor An\rfloor\right)-Y_n\left( T_n\right)-\sum_{m= T_n}^{\lfloor An\rfloor-1} \mathbb{E}\left[\Delta Y_n(m)|\mathbf{G}_n(m)\right]\bigg| \geq \varepsilon\Bigg) \leq 2 \exp\left(-\frac{\varepsilon^2}{8 An u_n^2} \right).
	$$
In particular, when $u_n=o(n^{1/2})$, 	
$$
\frac{Y_n\left(\lfloor An\rfloor\right)-Y_n\left( T_n\right)-\sum_{m= T_n}^{\lfloor An\rfloor-1} \mathbb{E}\left[\Delta Y_n(m)|\mathbf{G}_n(m) \right] }{n}~ \underset{n \rightarrow \infty}{\overset{\mathbb P}\longrightarrow}~ 0.
$$
\end{lemma}
\begin{proof}
	For each $n\in \mathbb N$, introduce the process $M_n$ defined by $M_n(0)=0$ and
	$$M_n(m)=\sum_{k=0}^{m-1}\left(\Delta Y_n(k)-\mathbb{E}\left[\Delta Y_n(k)|\mathbf{G}_n(k)\right]\right), \quad m\geq 1.$$
This defines a martingale (with respect to the filtration $\mathbf G_n$) with bounded jumps: $\left\vert \Delta M_n(m)\right\vert \leq 2u_n$, $\forall m \in \mathbb Z_+$.
The stopped process defined by $M_n^{T_n}(m)=M_n\left(T_n\wedge m\right)$ is also a martingale, and so is $M_n-M_n^{T_n}$. The jumps of this last martingale are uniformly bounded by $2u_n$ and we conclude by applying Azuma-Hoeffding inequality to $M_n-M_n^{T_n}$ at time $\lfloor An\rfloor$.
\end{proof}

\subsubsection{The processes  before time $n/2$}

From \cite{ContatCurien23,viau25}, the size of the gel $G_{p,n}$ and of the number of discarded edges $D_{p,n}$ are respectively of order $n^{2/3}$ and $n^{1/3}$ at time $\lfloor n/2 \rfloor$ (in the sense that appropriately normalized those random variables have a limit in distribution). Which we summarize roughly as
\begin{equation}
\label{bound1/2}
\sup_{t \leq 1/2} \frac{G_{p,n}(\lfloor nt \rfloor)}{n^{2/3}} = \frac{G_{p,n}(\lfloor n/2 \rfloor)}{n^{2/3}} =O_{\mathbb P}(1) ; \qquad \sup_{t \leq 1/2} \frac{D_{p,n}(\lfloor nt \rfloor)}{n^{2/3}} = \frac{D_{p,n}(\lfloor n/2 \rfloor)}{n^{2/3}}  =o_{\mathbb P}(1).
\end{equation}

\subsubsection{The forest part of $\mathrm{F}_{p,n}$ is never too supercritical}

An important point in our approach is to evaluate the \emph{criticality} of the forest part of the graphs $\mathrm F_{p,n}(m), m \in \mathbb Z_+$, with the vocabulary of uniform random forests of Section \ref{sec:URF}, see in particular Proposition \ref{lm:Britikov} and the following paragraph. From the relations \eqref{def:V_{p,n}E_{p,n}}, the number of vertices $V_{p,n}(m)$ and number of edges $E_{p,n}(m)$ in the forest part of the graph $\mathrm{F}_{p,n}(m)$ verify
\begin{equation}\label{eq:rel_E_V}
		2E_{p,n}(m)-V_{p,n}(m)=2m-n-G_{p,n}(m)-2D_{p,n}(m).
\end{equation}
The sub/sur/criticality is determined by the asymptotic position of this quantity relatively to \linebreak $(V_{p,n}(m))^{2/3}$.
Clearly, the forest part of $\mathrm{F}_{p,n}(m)$ is (sub)critical when $m \leq n/2$ and truly subcritical whem $m=\lfloor n(1/2-\varepsilon)\rfloor$ for some $\varepsilon>0$.
	
The lemma below shows that for $m\geq n/2$ also the forest part of the graph $\mathrm{F}_{p,n}$ cannot be "too supercritical".  It will be crucial in the proof of the key Proposition	\ref{prop:starting_pt}, which in turn implies that the forest part of the graph $\mathrm F_{p,n}(m)$ is subcritical with high probability when $m=\lfloor n(1/2+\varepsilon)\rfloor$ for some $\varepsilon>0$.
Heuristically, the idea of the proof of the lemma below is that if the forest were supercritical at some time, it would contain a giant tree, which would freeze quickly with high probability, which then would lower down the criticality.

	\begin{lemma}\label{lm:supercrit}
		For every $A\in \left(1/2,1\right)$, there exists $c_{A}>0$ such that for  $n$ large enough
		$$\mathbb{P}\left(\exists m\in \left[n/2,An\right]:\, 2E_{p,n}(m)-V_{p,n}(m)\geq \ln(n)^3 n^{2/3}\right) \leq \mathrm{e}^{-c_A\ln(n)^2}.$$
	\end{lemma}
	
\textbf{\textit{Remark.}} This lemma can in fact be extended to all $A>1/2$. To do this, the proof should be refined. As we will only need the version with $A\in \left(1/2,1\right)$ in the following, we will leave it as is.	
	
	\begin{proof}
		Fix $A\in \left(1/2,1\right)$. Consider the stopping time
		$$T_n:=\inf\left\{m\geq n/2:\, 2E_{p,n}(m)-V_{p,n}(m)\geq \ln(n)^3 n^{2/3} \right\}$$
and the random time
$$S_n:=\inf\left\{s\in \left[n/2, T_n\right]\cap \mathbb Z_+ : \frac{\ln(n)^3 n^{2/3}}{2}\leq 2E_{p,n}(m)-V_{p,n}(m)<\ln(n)^3n^{2/3}\quad\forall m \in [s,T_n) \right\}$$
(with the convention $\inf\{\emptyset\}=\infty$). Using that the positive jumps of the process $m \mapsto 2E_{p,n}(m)-V_{p,n}(m)$ are smaller or equal to 2, 
we see that for $n$ not too small $\left\{T_n<\infty \right\}\subset \left\{S_n<\infty \right\}$ and that when $T_n<\infty$, $2E_{p,n}(S_n)-V_{p,n}(S_n)\leq \ln(n)^3n^{2/3}/2+2$. This in turn implies that
for all $m\in [S_n,S_n+n^{2/3}]$, 
$$\frac{\ln(n)^3 n^{2/3}}{2}~\leq~2E_{p,n}(m)-V_{p,n}(m)~\leq 2E_{p,n}(S_n)-V_{p,n}(S_n)+2n^{2/3}~\leq~\frac{\ln(n)^3 n^{2/3}}{2}+2n^{2/3}+2.$$
Consequently,
		\begin{equation}
			\{ T_n\leq An\}\subset\bigcup_{s=\left\lceil\frac{n}{2}\right\rceil}^{\left\lfloor An\right\rfloor}A_n(s)\label{eq:subs:surcrit}
		\end{equation} 
with, for $s \in \mathbb Z_+$, $$A_n(s)=\left\{\forall m\in \left[s,s+n^{2/3}\right],\, \frac{\ln(n)^3 n^{2/3}}{2}\leq 2E_{p,n}(m)-V_{p,n}(m)\leq\frac{\ln(n)^3 n^{2/3}}{2}+2n^{2/3}+2  \right\}.$$
Heuristically, on $A_n(s)$ the forest part of the graph is supercritical over the time interval $[s,s+n^{2/3}]$ and this supercriticality hardly varies. In particular, the vertices of a tree of size larger than $2n^{2/3}+5$ at time $s$, if there are any, will not be frozen at time $\lfloor s+n^{2/3} \rfloor$. 

Let $T^{(1)}_{p,n}(s)$ be the largest tree in the forest at time $s$ (if several trees have the largest size, we choose such a tree at random) and $\# T^{(1)}_{p,n}(s)$ its size. We use the splitting
		\begin{align*}
			A_n(s)\subset &\bigg\{\# T^{(1)}_{p,n}(s)>\ln(n)n^{2/3}, \text{the vertices of }T^{(1)}_{p,n}(s)  \text{ are not frozen at time } \lfloor s+n^{2/3} \rfloor \bigg\}\\ & \cup  \bigg\{\# T^{(1)}_{p,n}(s)\leq \ln(n)n^{2/3},\,  \frac{\ln(n)^3n^{2/3}}{2} \leq 2E_{p,n}(s)-V_{p,n}(s)\leq \ln(n)^3n^{2/3} \bigg\}
		\end{align*}
to obtain a relevant upper bound for $\mathbb P(A_n(s))$. Note that at each step in the process $\mathrm{F}_{p,n}$, the probability that the new arriving edge creates a cycle in a given tree of size $k$ (sending therefore this tree in the gel) is $\frac{k(k-1)}{n(n-1)}$. Consequently,		
		\begin{align*}
			& \mathbb P\left(\# T^{(1)}_{p,n}(s)>\ln(n)n^{2/3}, \text{the vertices of }T^{(1)}_{p,n}(s)  \text{ are not frozen at time } \lfloor s+n^{2/3} \rfloor \right)\notag\\
			& \hspace{2cm}\leq \left(1-\frac{\ln(n)^2n^{4/3}}{n(n-1)}+\frac{\ln(n)n^{2/3}}{n(n-1)}\right)^{n^{2/3}}~\underset{n\to \infty}{\sim} \mathrm{e}^{-\ln(n)^2}.\label{eq:largest_treE_{p,n}ot_freeze}
		\end{align*}
	Next, recalling that $L_1(N,M)$ denotes the size of the largest tree in a uniform random forest with $N$ vertices and $M$ edges, Proposition \ref{prop:freeforest} yields
		\begin{eqnarray*}
			&&\mathbb{P}\left(\# T^{(1)}_{p,n}(s)\leq \ln(n)n^{2/3},\, \frac{\ln(n)^3n^{2/3}}{2} \leq 2E_{p,n}(s)-V_{p,n}(s)\leq \ln(n)^3n^{2/3}  \right)\\
			&\leq & \mathbb{E}\left[\mathbb{P}\left(L_1\left(V_{p,n}(s),E_{p,n}(s)\right)\leq  \frac{\ln(n)}{(1-A)^{2/3}} \cdot V_{p,n}(s)^{2/3}, \right.\right.\\
			&& \hspace{1cm}\left. \left. \frac{\ln(n)^3}{2}\cdot V_{p,n}(s)^{2/3} \leq 2E_{p,n}(s)-V_{p,n}(s)\leq \frac{\ln(n)^3}{(1-A)^{2/3}} \cdot V_{p,n}(s)^{2/3}~\Big\vert~\mathbf{F}_{p,n}(s)\right)\right]
		\end{eqnarray*}
where we used that $(1-A)n \leq V_{p,n}(s) \leq n$ for $s\leq \lfloor An \rfloor$ and recall that $\mathbf{F}_{p,n}$ is the filtration generated by $(V_{p,n},E_{p,n})$. Recall also that $A<1$. 	
Lemma \ref{lm:forest:super}, whose assumptions  are clearly satisfied here, implies that the conditional probability above is smaller than $c_1 (\ln(n))^{15/2} n^{2/3} e^{-c_2\ln(n)^2}$ for deterministic constants $c_1,c_2\in (0,\infty)$ (that depend on $A$) and all $n$ large enough, whatever $s \leq \lfloor An\rfloor$. All in all, we have shown that $$\mathbb P(A_n(s)) \leq e^{-c_3 \ln(n)^2}$$
for some $c_3\in (0,\infty)$ and all $n$ large enough, whatever $s \leq \lfloor An\rfloor$.

Finally, applying the union bound to (\ref{eq:subs:surcrit}) gives the result. 
	\end{proof}

	\subsection{The gel just beyond the critical window}\label{section:starting_pt}
	
	The aim of this section is to provide the following bounds for the gel size just beyond the critical window. This will allow us to implement the stochastic differential method in the next section. 
	
	\begin{proposition}\label{prop:starting_pt}
		For every $p\in (0,1]$, there are constants $c(p),C(p) \in (0,\infty)$ such that for $\varepsilon>0$ small enough
		$$\mathbb{P}\left(G_{p,n}\left(\left\lfloor\frac{n}{2}+\varepsilon n\right\rfloor\right) \in \big[(2+c(p))\varepsilon n~;~C(p)\varepsilon n \big]\right)~\underset{n\to \infty} \longrightarrow~1.  $$ 
	\end{proposition}
Observe with (\ref{eq:rel_E_V}) that this implies that the forest part of the graph at time $\lfloor (1/2 +\varepsilon) n \rfloor$ verifies with high probability
	\begin{align}
		2E_{p,n}\left(\left\lfloor\frac{n}{2}+\varepsilon n\right\rfloor\right)-V_{p,n}\left(\left\lfloor\frac{n}{2}+\varepsilon n\right\rfloor\right)&=2\varepsilon n-G_{p,n}\left(\left\lfloor\frac{n}{2}+\varepsilon n\right\rfloor\right)-2D_{p,n}\left(\left\lfloor\frac{n}{2}+\varepsilon n\right\rfloor\right)\notag\\
		&\leq -c(p)\varepsilon V_{p,n}\left(\left\lfloor\frac{n}{2}+\varepsilon n\right\rfloor\right)\notag
	\end{align}
which indicates that the forest is in the subcritical regime. Since the largest tree in a subcritical forest is "small" (Lemma \ref{lm:forest:sub}), so are the jumps of $G_{p,n}$. This will help us to verify the \emph{boundedness hypothesis} of the differential equation method (Theorem \ref{thmDEM}). 

The proof of Proposition \ref{prop:starting_pt} is quite involved and divided into several steps. The difficulty lies in proving the lower bound $(2+c(p))\varepsilon n$. For the upper bound, we can simply use a comparison with the standard  Erd\H{o}s-Rényi graph, which yields a constant $C(p)$ which could be any number larger than $4$. So $C(p)$ does not depend on $p$ with our approach (we let however the notation depends on $p$ since in principle one could obtain the statement with any constant larger than $2(1+p)$).

\subsubsection{Proof of the lower bound}

Since $p>0$ we can choose $c(p)>0$ and define $\tilde c(p)$ such that:
$$
2+c(p) <\frac{2(1+p)^2}{1+2p} \quad \text{ and } \quad \tilde{c}(p)=\frac{2+c(p)}{2(1+p)}~\in~(0,1). 
$$
Our goal is to show that for $\varepsilon\in \left(0,1/2\right)$ sufficiently small,  the probability of the event
$$A_n(\varepsilon):=\left\{G_{p,n}\left( \left\lfloor\frac{n}{2}+\varepsilon n\right\rfloor \right) < (2+c(p))\varepsilon n \right\}$$
converges to 0 as $n \rightarrow \infty$.

\textbf{Heuristics and preparatory work.} Our aim is to approximate  $G_{p,n}\left(\left\lfloor \frac{n}{2}+\varepsilon n\right\rfloor\right)$ by the sum of the conditional expectations of its jumps $\mathbb{E}\left[\Delta G_{p,n}(m)\vert \mathbf{F}_{p,n}(m)\right]$, using 
Lemma \ref{lm:approx:esp}, and to evaluate these expectations with the help of Corollary \ref{corol:est_transitions}. There are several obstacles on our way to implement this properly. We will therefore have to position ourselves a little beyond the exit of the critical window and consider only small jumps. Heuristically, we will use the approximation
\begin{equation*}
	G_{p,n}\left(\left\lfloor \frac{n}{2}+\varepsilon n\right\rfloor\right)\approx G_{p,n}\left(\left\lfloor \frac{n}{2}+\tilde{c}(p)\varepsilon n\right\rfloor\right)+\sum_{m=\lfloor \frac{n}{2}+\tilde{c}(p)\varepsilon n\rfloor}^{\left\lfloor \frac{n}{2}+\varepsilon n\right\rfloor-1} \mathbb{E}\left[\Delta G_{p,n}(m)\mathbbm{1}_{\{\Delta G_{p,n}(m)\leq n^{1/4}\}}\vert \mathbf{F}_{p,n}(m)\right]\label{eq:heurist_approx}
\end{equation*}
to get a lower bound for  $G_{p,n}\left(\left\lfloor \frac{n}{2}+\varepsilon n\right\rfloor\right)$, using Lemma \ref{lm:supercrit} to get a lower bound for $G_{p,n}\left(\left\lfloor \frac{n}{2}+\tilde{c}(p)\varepsilon n\right\rfloor\right)$ and Corollary \ref{corol:est_transitions} to estimate the conditional  expectations of the jumps. Even if Lemma \ref{lm:supercrit} only gives a crude lower bound, we will be able to show that the drift of the process over the time interval $\left[\frac{n}{2}+\tilde{c}(p)\varepsilon n,\frac{n}{2}+\varepsilon n\right]$ is large enough to compensate for this shortfall. We have chosen $\tilde{c}(p)$ in such a way that it is both sufficiently large to move away from the critical window, but also small enough so that the information about the drift over $\left[\frac{n}{2}+\tilde{c}(p)\varepsilon n,\frac{n}{2}+\varepsilon n\right]$ can provide the compensation. This will be detailed in the proof of Lemma \ref{lm:min_esp}, but we already observe here that 
\begin{equation}\label{eq:inclusion_An}
	A_n(\varepsilon)\subset \left\{\forall m\in \left[\frac{n}{2}+\tilde{c}(p)\varepsilon n,\frac{n}{2}+\varepsilon n   \right],\, G_{p,n}(m)\leq 2(1+p)\left(m-\frac{n}{2}\right)  \right\}.
\end{equation} 

Let us now introduce a series of events that all have a probability converging to 1 as $n\rightarrow \infty$ and on which it will be easier to work:  

\begin{enumerate}[leftmargin=0.5cm]
\item[$\bullet$]
$  B_n(\varepsilon)${\small{$\displaystyle ~=\left\{\left\vert \sum_{m=\left\lfloor \frac{n}{2}+\tilde{c}(p)\varepsilon n\right\rfloor}^{\left\lfloor\frac{n}{2}+\varepsilon n\right\rfloor-1}\Big(\Delta G_{p,n}(m)\mathbbm{1}_{\{\Delta G_{p,n}(m)\leq n^{1/4}\}}-\mathbb{E}\left[\Delta G_{p,n}(m)\mathbbm{1}_{\{\Delta G_{p,n}(m)\leq n^{1/4}\}}\vert \mathbf{F}_{p,n}(m)   \right]\Big) \right\vert \leq \varepsilon^2 n    \right\}$}}
\item[$\bullet$]  $ \displaystyle C_n(\varepsilon)=\left\{\left\vert \sum_{m=\left\lfloor \frac{n}{2}\right\rfloor }^{\left\lfloor\frac{n}{2}+\varepsilon n\right\rfloor-1}\Big(\Delta D_{p,n}(m)-\mathbb{E}\left[\Delta D_{p,n}(m)\vert \mathbf{F}_{p,n}(m)   \right]\Big) \right\vert \leq \varepsilon^2 n    \right\}\cap \left\{D_{p,n}\left(\frac{n}{2}\right)\leq \varepsilon^2 n \right\}$

\item[$\bullet$] $\displaystyle D_n(\varepsilon)=\left\{\forall m\in \left[\frac{n}{2},\frac{n}{2}+\varepsilon n\right],\, G_{p,n}(m)+2D_{p,n}(m)\geq 2m-n-\ln(n)^3n^{2/3} \right\}.$
\end{enumerate}
 
Lemma \ref{lm:approx:esp} and (\ref{bound1/2}) imply that $\mathbb P(B_n(\varepsilon))\rightarrow 1$, $\mathbb P(C_n(\varepsilon))\rightarrow 1$ as $n \rightarrow \infty$, while  Lemma \ref{lm:supercrit} and the identity (\ref{eq:rel_E_V}) imply that $\mathbb P(D_n(\varepsilon))\rightarrow 1$. We can therefore focus on $$\tilde{A}_n(\varepsilon):=A_n(\varepsilon)\cap B_n(\varepsilon)\cap C_n(\varepsilon)\cap D_n(\varepsilon)$$ to prove that  $\mathbb{P}\left(A_n(\varepsilon)\right) \rightarrow 0$ as $n \rightarrow \infty$.

We finish this preparatory part by a few remarks. First, since we have taken $\varepsilon$ in $\left(0,1/2\right)$, $V_{p,n}(m) \geq \left(1/2-\varepsilon\right)n$, $\forall m \in \left[0,\frac{n}{2}+\varepsilon n\right]$, which improves on $A_n(\varepsilon)$ in
\begin{equation}\label{eq:def_H_n}
V_{p,n}(m) \geq \left(1-(2+c(p))\varepsilon\right)n, \quad \forall m \in \left[0,\frac{n}{2}+\varepsilon n\right]. 
\end{equation} 
Next, since (by Proposition \ref{lm_transitions})
$$\mathbb{E}\left[\Delta D_{p,n}(m)\vert \mathbf{F}_{p,n}(m)\right]=2(1-p)\frac{G_{p,n}(m)\left(n-G_{p,n}(m)\right)}{n(n-1)}+\frac{G_{p,n}(m)\left(G_{p,n}(m)-1\right)}{n(n-1)}, \quad \forall m\geq 0,$$  
there exists $\mathrm{c_D} \in (0,\infty)$ such that 
\begin{equation}\label{eq:maj_D_{p,n}}
	D_{p,n}\left(\left\lfloor\frac{n}{2}+\varepsilon n\right\rfloor\right)\leq \mathrm{c_D} \varepsilon^2n \quad \text{ on } A_n(\varepsilon)\cap C_n(\varepsilon). 
\end{equation}
Then, using this upper bound (\ref{eq:maj_D_{p,n}}), we see that for $n$ large enough and all $m \geq \frac{n}{2}+\tilde{c}(p)\varepsilon n$,
\begin{equation}
\label{eq:min_G_{p,n}}
\frac{G_{p,n}(m)}{n}~\geq~2\tilde{c}(p)\varepsilon -(1+\mathrm{c_D})\varepsilon^2 \quad \text{ on } A_n(\varepsilon)\cap C_n(\varepsilon)\cap D_n(\varepsilon).
\end{equation}
Last, using the identity $E_{p,n}(m)=m-G_{p,n}(m)-D_{p,n}(m)$, we see that for $\varepsilon>0$ small enough, on the event $\tilde{A}_n(\varepsilon)$
\begin{equation}
\label{eq:borne_inf_edges}
	E_{p,n}(m)~\geq\frac{n}{2}+\tilde{c}(p)\varepsilon n-(2+c(p))\varepsilon n -\mathrm{c_D}\varepsilon^2 n~\geq~\left(\frac{1}{2}-\sqrt \varepsilon\right)n 
\end{equation}
for all $m\in \left[\frac{n}{2}+\tilde{c}(p)\varepsilon n, \frac{n}{2}+\varepsilon n  \right]$. 

\bigskip

We now turn to the proof of the following lower bound for the conditional expectations of $\Delta G_{p,n}$ on $\tilde{A}_n(\varepsilon)$, which will allows us to show that $\mathbb P(\tilde A_n(\varepsilon))\rightarrow 0$, and consequently $\mathbb P( A_n(\varepsilon))\rightarrow 0$, as $n \rightarrow \infty$, for sufficiently small $\varepsilon>0$.

\begin{lemma}\label{lm:min_esp}
	There exists $c \in (0,\infty)$ such that for $\varepsilon>0$ small enough, and then $n$ large enough and every $m\in\left[\frac{n}{2}+\tilde{c}(p)\varepsilon n,\frac{n}{2}+\varepsilon n\right]$,
	$$	\mathbb{E}\left[\Delta G_{p,n}(m)\mathbbm{1}_{\{\Delta G_{p,n}(m)\leq n^{1/4}\}}\vert \mathbf{F}_{p,n}(m)   \right]\mathbbm{1}_{\tilde{A}_n(\varepsilon)}\geq 
	\big(2(1+p)-{c}\varepsilon \big)\mathbbm{1}_{\tilde{A}_n(\varepsilon)}$$  \end{lemma} 

\begin{proof}
Fix $\varepsilon \in (0,1/2)$ small enough such that the conclusion of (\ref{eq:borne_inf_edges}) holds on $\tilde{A}_n(\varepsilon)$.
	Corollary \ref{corol:est_transitions} and its consequence \eqref{eq:crit_surcrit} provide $A>0$ (depending on $\varepsilon$) such that for $n$ large enough (the threshold also depends on $\varepsilon$) and all $m\in \left[\frac{n}{2}+\tilde{c}(p)\varepsilon n, \frac{n}{2}+\varepsilon n  \right]$, on the event $\tilde{A}_n(\varepsilon)$:
	
	$\quad \bullet$	for $m\geq 1$  such that $2E_{p,n}(m)-V_{p,n}(m)\leq -AV_{p,n}(m)^{2/3}$ and every $k \in \{1,...,V_{p,n}(m)^{1/4}\} $ 
		\begin{equation}
			\mathbb{P}\left(\Delta G_{p,n}(m)=k\vert \mathbf{F}_{p,n}(m)\right)~\geq~(1-\varepsilon) \frac{k^{k-1}}{k!}\left(\frac{2E_{p,n}(m)}{V_{p,n}(m)}\right)^{k-1}\mathrm{e}^{-2k\frac{E_{p,n}(m)}{V_{p,n}(m)}}\frac{2pG_{p,n}(m)(n-G_{p,n}(m))}{n^2} \label{eq:min_probcdtion}
		\end{equation}
		
	$\quad \bullet$ for $m\geq 1$ such that $2E_{p,n}(m)-V_{p,n}(m)>-AV_{p,n}(m)^{2/3}$ and every $k \in \{1,...,V_{p,n}(m)^{1/4} \}$
		\begin{equation}\label{eq:min_prob_cdti_crit+}
			\mathbb{P}\left(\Delta G_{p,n}(m)=k\vert \mathbf{F}_{p,n}(m)\right)~\geq~\frac{c_A} {k^{3/2}} \cdot \frac{G_{p,n}(m)\left(V_{p,n}(m)-E_{p,n}(m)\right)}{n^2}
		\end{equation}
		where $c_A \in (0,\infty)$ only depends on $A$.

	Here we used that, on $\tilde{A}_n(\varepsilon)$, $E_{p,n}(m)$ and $V_{p,n}(m)$ are deterministically large as soon as $n$ is large, thanks to \eqref{eq:borne_inf_edges} and (\ref{eq:def_H_n}), and also that $V_{p,n}(m)-E_{p,n}(m)-1\geq (V_{p,n}(m)-E_{p,n}(m))/2$ for $n$ large enough, since $V_{p,n}(m)-E_{p,n}(m)=n-m+D_{p,n}(m)\geq (1/2-\varepsilon) n$ when $m \leq n/2+\varepsilon n$. 
	
We will need the existence, ensured by Lemma \ref{lm:partial_sum}, of $\delta>0$ and $N_0 \in \mathbb N$ (depending both on $\varepsilon$) such that: 
\begin{eqnarray}
	&& S_{N_0}(2z)=\sum_{k= 1}^{N_0}\frac{k^k}{k!}\left(2z\right)^{k-1}\mathrm{e}^{-2kz}~\geq~\frac{1-\varepsilon}{1-2z}, \quad \forall z\in \left[0,\frac{1}{2}-\delta\right] \label{eq:min_{F}_{p,n}_1}\\
	&&	2p\cdot (1-\varepsilon)\cdot \tilde{c}(p)\varepsilon \cdot (1-(2+c(p))\varepsilon)\cdot S_{N_0}(2z)~\geq~2(1+p), \quad \forall z\in \left[\frac{1}{2}-\delta,\frac{1}{2}\right] \label{eq:min_{F}_{p,n}_2}\\
	&&	\sum_{k=1}^{N_0}\frac{c_A}{\sqrt{k}} \left(\frac{1}{2}-\varepsilon\right)\tilde{c}(p)\varepsilon~\geq~2(1+p). \label{eq:min_E_crit_sucrit}
\end{eqnarray}
	
We then distinguish three cases to show the expected inequality of the statement for \linebreak $m\in\left[\frac{n}{2}+\tilde{c}(p)\varepsilon n,\frac{n}{2}+\varepsilon n\right]$:	

{\underline{Case 1}, when}  $2E_{p,n}(m)-V_{p,n}(m)\leq -2\delta V_{p,n}(m)$: using \eqref{eq:min_probcdtion} and \eqref{eq:min_{F}_{p,n}_1} we get that for $n$ large enough, on $\tilde{A}_n(\varepsilon)$,
		\begin{align*}
			\mathbb{E}\left[\Delta G_{p,n}(m)\mathbbm{1}_{\{\Delta G_{p,n}(m)\leq n^{1/4}\}}| \mathbf{F}_{p,n}(m)\right]
			&\geq \left(1-\varepsilon\right)^22p\frac{G_{p,n}(m)\left(n-G_{p,n}(m)\right)}{\left(1-2E_{p,n}(m)/V_{p,n}(m)\right)n^2}\\
			&=\left(1-\varepsilon\right)^22p\frac{G_{p,n}(m)\left(n-G_{p,n}(m)\right)^2}{\left(n-2m+G_{p,n}(m)+2D_{p,n}(m)\right)n^2}.
		\end{align*}
Observe that on $\tilde{A}_n(\varepsilon)$, by definition of $A_n(\varepsilon)$,
		$\left(1-G_{p,n}(m)/n\right)^2\geq \left(1-(2+c(p))\varepsilon\right)^2,$ 
		and then use  \eqref{eq:inclusion_An} and \eqref{eq:maj_D_{p,n}} to see that the inequality above implies
		\begin{align*}
			& \mathbb{E}\left[\Delta G_{p,n}(m)\mathbbm{1}_{\{\Delta G_{p,n}(m)\leq n^{1/4}\}}| \mathbf{F}_{p,n}(m)\right] \\
			& \quad \geq \left(1-\varepsilon\right)^2\left(1-(2+c(p))\varepsilon\right)^2 2p\left(1+\frac{2m-n-2D_{p,n}(m)}{n-2m+G_{p,n}(m)+2D_{p,n}(m)}\right) \\
		 	& \quad \geq \left(1-\varepsilon\right)^2\left(1-(2+c(p))\varepsilon\right)^2 2p\left(1+\frac{2m-n-2\mathrm{c_D}\varepsilon^2 n}{p(2m-n)+2\mathrm{c_D}\varepsilon^2n}\right).
		\end{align*}
		Since $m\geq n/2+\tilde{c}(p)\varepsilon n$, the desired inequality follows in this case for some well chosen constant ${c} \in (0,\infty)$, independent of $\varepsilon$ and $n$.

{\underline{Case 2}, when} $-2\delta V_{p,n}(m)\leq 2E_{p,n}(m)-V_{p,n}(m)\leq -AV_{p,n}(m)^{2/3}$: \eqref{eq:min_probcdtion} yields that for $n$ large enough, on $\tilde{A}_n(\varepsilon)$,  
		\begin{align*}
			\mathbb{E}\left[\Delta G_{p,n}(m)\mathbbm{1}_{\{\Delta G_{p,n}(m)\leq n^{1/4}\}}| \mathbf{F}_{p,n}(m)\right] &\geq (1-\varepsilon) 2p\frac{G_{p,n}(m)\left(n-G_{p,n}(m)\right)}{n^2}S_{N_0}\left(\frac{2E_{p,n}(m)}{V_{p,n}(m)}\right).
		\end{align*}
Assume now that $\varepsilon$ is small enough such that $(1+\mathrm c_D)\varepsilon \leq 1$. On $ \tilde{A}_n(\varepsilon)$, one then has $G_{p,n}(m)\geq \tilde{c}(p)\varepsilon n$	and $n-G_{p,n}(m)\geq \left(1-(2+c(p))\varepsilon\right)n$ by (\ref{eq:min_G_{p,n}}) and the definition of $A_n(\varepsilon)$ respectively (taking $n$ larger if necessary). The above inequality therefore leads, together with \eqref{eq:min_{F}_{p,n}_2}, to the lower bound
	\begin{align*}
	\mathbb{E}\left[\Delta G_{p,n}(m)\mathbbm{1}_{\{\Delta G_{p,n}(m)\leq n^{1/4}\}}| \mathbf{F}_{p,n}(m)\right]
	&\geq 2(1+p).
	\end{align*}

{\underline{Case 3}, when} $ 2E_{p,n}(m)-V_{p,n}(m) > -AV_{p,n}(m)^{2/3}$: \eqref{eq:min_prob_cdti_crit+} implies that for sufficiently large $n$, on $\tilde{A}_n(\varepsilon)$,
		\begin{align*}
			\mathbb{E}\left[\Delta G_{p,n}(m)\mathbbm{1}_{\{\Delta G_{p,n}(m)\leq n^{1/4}\}}| \mathbf{F}_{p,n}(m)\right] &\geq \sum_{k=1}^{N_0}\frac{c_A}{\sqrt{k}}\frac{G_{p,n}(m)(V_{p,n}(m)-E_{p,n}(m))}{n^2}	\\
			&\geq \sum_{k=1}^{N_0}\frac{c_A}{\sqrt{k}}\left(\frac{1}{2}-\varepsilon\right)\tilde{c}(p)\varepsilon
		\end{align*}
	where we used that for $m\in \left[\frac{n}{2}+\tilde{c}(p)\varepsilon n,\frac{n}{2}+\varepsilon n\right]$, on $\tilde{A}_n(\varepsilon)$, $G_{p,n}(m)\geq \tilde{c}(p)\varepsilon n$ by (\ref{eq:min_G_{p,n}}) and $V_{p,n}(m)-E_{p,n}(m)=n-m+D_{p,n}(m)\geq \left(1/2-\varepsilon\right)n$. We conclude using \eqref{eq:min_E_crit_sucrit}.
\end{proof}

\textbf{End of the proof of the lower bound of Proposition \ref{prop:starting_pt}.}  We take $\varepsilon>0$ small enough such that the conclusion of Lemma \ref{lm:min_esp} holds.
As announced in the heuristic introduction, we use that
\begin{align*}
		G_{p,n}\left(\left\lfloor\frac{n}{2}+\varepsilon n\right\rfloor\right)&\geq \left(G_{p,n}\left(\left\lfloor\frac{n}{2}+\tilde{c}(p)\varepsilon n\right\rfloor\right)+ \sum_{m=\left\lfloor\frac{n}{2}+\tilde{c}(p)\varepsilon n\right\rfloor}^{\left\lfloor\frac{n}{2}+\varepsilon n\right\rfloor-1}\Delta G_{p,n}(m)\mathbbm{1}_{\{\Delta G_{p,n}(m)\leq n^{1/4}\}} \right)
	\end{align*}
to get, with the definition of $B_n(\varepsilon)$ and (\ref{eq:min_G_{p,n}}), that on $\tilde{A}_n(\varepsilon)$
	\begin{align*}
		G_{p,n}\left(\left\lfloor\frac{n}{2}+\varepsilon n\right\rfloor\right) &\geq \left(2\tilde{c}(p)\varepsilon n +\sum_{m=\left\lfloor\frac{n}{2}+\tilde{c}(p)\varepsilon n\right\rfloor }^{\left\lfloor\frac{n}{2}+\varepsilon n\right\rfloor-1}\mathbb{E}\left[\Delta G_{p,n}(m)\mathbbm{1}_{\{\Delta G_{p,n}(m)\leq n^{1/4}\}}\vert \mathbf{F}_{p,n}(m)   \right]-(2+\mathrm{c_D})\varepsilon^2n    \right) \\
		&\geq  \left(2\tilde{c}(p)\varepsilon n +2(1+p)\left(1-\tilde{c}(p)\right)\varepsilon n -(2+\mathrm{c_D}+C)\varepsilon^2n    \right)
	\end{align*}
with $C \in (0,\infty)$, where we used Lemma \ref{lm:min_esp} to get the second inequality.
Since $G_{p,n}\left(\left\lfloor\frac{n}{2}+\varepsilon n\right\rfloor\right) \leq (2+c(p))\varepsilon n$ on $\tilde{A}_n(\varepsilon)$ by definition of ${A}_n(\varepsilon)$, this implies that 
	$$(2+c(p)) \geq  2\tilde{c}(p)  +2(1+p)\left(1-\tilde{c}(p)\right)  -(2+\mathrm{c_D}+C)\varepsilon.$$
Now, note that by choice and definition of $c(p),\tilde c(p)$, one has $(2+c(p))<2\tilde{c}(p)  +2(1+p)(1-\tilde{c}(p))$. The above reasoning therefore  implies that $\tilde{A}_n(\varepsilon)=\emptyset$ for $\varepsilon$ sufficiently small and all  $n$ large enough. So finally, since $\tilde{A}_n(\varepsilon)=A_n(\varepsilon)\cap B_n(\varepsilon)\cap C_n(\varepsilon)\cap D_n(\varepsilon)$ and the events $B_n(\varepsilon), C_n(\varepsilon), D_n(\varepsilon)$ all have a probability that converges to 1 as $n \rightarrow \infty$, we have indeed that $\mathbb P(A_n(\varepsilon)) \rightarrow 0$ for small values of $\varepsilon$.
  	
	\subsubsection{Proof of the upper bound}
	
The proof of the upper bound is easier as we can used a comparison with the standard Erd\H{o}s-Rényi model. Indeed, as already observed, for any $m$, $G_{p,n}(m)$ is stochastically smaller than $C_{\mathrm{ER},n}(m)$, the total number of vertices  involved at time $m$ in cyclic components of the Erd\H{o}s-Rényi random graph. For $t>1/2$, with high probability as $n \rightarrow \infty$,  the giant component of the Erd\H{o}s-Rényi graph at time $\lfloor nt\rfloor$ is a cyclic component and the number of vertices involved in \emph{other} cyclic components is bounded (see e.g. \cite[Theorem 6.11]{bollobas01}). Consequently, one has \begin{equation*}
	\frac{C_{\mathrm{ER},n}(\left(\lfloor nt\rfloor\right)}{n}~\overset{\mathbb{P}}\longrightarrow~g_{\mathrm{ER}}(t)=g_1(t).
\end{equation*}
Since $g_{\mathrm{ER}}(1/2)=0$ and $g_{\mathrm{ER}}'$ is bounded on $[1/2,\infty)$ (by $g_{\mathrm{ER}}'(1/2^+)=4$), one has $g_{\mathrm{ER}}(1/2+\varepsilon) < C \varepsilon$ for some finite $C>4$ and all $\varepsilon \in (0,1]$. The result follows.

This proof has the advantage of being concise, but relies on results on the standard model. Alternatively it is possible to set up a self-contained proof, based on a similar approach to that used for the lower bound, with the additional difficulty of controlling potential large jumps. 

\subsection{Differential equation method for the process beyond time $(1/2+\varepsilon)n$}
\label{sec:DEM_eps}
		
Throughout this section we let $\varepsilon>0$, $c(p),C(p)$ be such that Proposition \ref{prop:starting_pt} holds and keep a certain amount of flexibility with $\varepsilon$, allowing us to choose it arbitrarily small if necessary. Our aim is to prove, via Wormald's theorem  (Theorem \ref{thmDEM}), that the process $$ \left(\frac{G_{p,n}\left(\left\lfloor nt+\varepsilon n\right\rfloor\right)}{n},\frac{D_{p,n}\left(\left\lfloor nt+\varepsilon n\right\rfloor\right)}{n}\right)_{t\geq 1/2}$$ converges in probability to a fluid limit which will approximate, when $\varepsilon$ goes to $0$, the couple of functions $(g_p,d_p)$, solution to (\ref{eq:syst_EDO_0}).  The convergence of  $ \left(G_{p,n}\left(\left\lfloor t n\right\rfloor\right)/n,D_{p,n}\left(\left\lfloor t n\right\rfloor\right)/n\right)_{t\geq 1/2}$ to $(g_p,d_p)$ as settled in Theorem \ref{thm:fluid limit} will mainly follow by using the triangular inequality, see the next section. 

Here we will work on the event 
\begin{equation}\label{def:I_n(epsilon)}
I_n(\varepsilon)=\left\{(2+c(p))\varepsilon n\leq G_{p,n}\left(\left\lfloor\frac{n}{2}+\varepsilon n\right\rfloor\right) \leq C(p)\varepsilon n \right\}\cap \left\{\varepsilon^3n\leq D_{p,n}\left(\left\lfloor\frac{n}{2}+\varepsilon n\right\rfloor\right)\leq (2C(p)+1)\varepsilon^2 n \right\}
\end{equation}
whose probability tends to $1$ as $n$ goes to infinity (at least for $\varepsilon$ small enough) as a consequence of Proposition \ref{prop:starting_pt}, together with (\ref{bound1/2}), Proposition \ref{lm_transitions} and Lemma \ref{lm:approx:esp} (regarding the bounds on $D_n$, we proceed similarly as around (\ref{eq:maj_D_{p,n}})). We will see in Lemma \ref{cor:edo_aleat} that on this event, provided that $\varepsilon$ is sufficiently small, the equation (\ref{eq:syst_EDO}) defined in Section \ref{section:approximation}, has a unique solution starting at time $t=1/2$ from  $\left(G_{p,n}\left(\left\lfloor\frac{n}{2}+\varepsilon n\right\rfloor\right)/n,D_{p,n}\left(\left\lfloor\frac{n}{2}+\varepsilon n\right\rfloor\right)/n\right)$, for all $n \geq 1$, which legitimates the following definition. 
 
\begin{definition}\label{def:f_nd_n_epsilon}
	For $\varepsilon>0$ small enough and all $n\geq 1$, we define the (random) couple of  functions $\big(g_{p,n}^{(\varepsilon)}(t),d_{p,n}^{(\varepsilon)}(t)\big)_{t\geq 1/2}$ as:
	\vspace{-0.3cm}
	$$\begin{array}{ll}
		\bullet \text{ the solution to \eqref{eq:syst_EDO} starting from } \displaystyle\left(\frac{G_{p,n}\left(\left\lfloor\frac{n}{2}+\varepsilon n\right\rfloor\right)}{n}, \frac{D_{p,n}\left(\left\lfloor\frac{n}{2}+\varepsilon n\right\rfloor\right)}{n}\right) & \text{on }I_n(\varepsilon)\\
		\bullet ~\left(g_p\left(t+\varepsilon\right),d_p(t+\varepsilon)\right)_{t\geq 1/2} &\text{on } I_n(\varepsilon)^c.
	\end{array}$$
\end{definition}

With this definition, we have the following convergence.
\begin{proposition}\label{prop:DEM}
For $\varepsilon>0$ small enough, 	
$$\left(\left(\left\lvert \frac{G_{p,n}\left(\lfloor nt+\varepsilon n\rfloor\right)}{n}-g_{p,n}^{(\varepsilon)}(t) \right\vert, \left\lvert \frac{D_{p,n}\left(\lfloor nt+\varepsilon n\rfloor\right)}{n}-d_{p,n}^{(\varepsilon)}(t) \right\vert \right)\right)_{t\geq \frac{1}{2}}~\underset{n \rightarrow \infty }{\overset{\mathbb{P}}{\longrightarrow}}~0$$	
for the topology of uniform convergence on compacts.
\end{proposition}

The proof of Proposition \ref{prop:DEM} consists in verifying the different hypotheses of Theorem \ref{thmDEM}, which is done in Section \ref{proof:Prop46}. To achieve this, we start by setting up some preliminary steps.

\subsubsection{Preliminaries}

\textbf{Definition of the domain.} To start with, we need to introduce a domain where the target functions are confined and within which the increments of the process are well approximated by the derivatives of the target functions. Technically, to proceed directly with Theorem \ref{thmDEM}, it is simpler to work with processes starting from 0, so we consider the shifted process defined for $t \geq 0$ by
\begin{equation*}\label{def:F_n_tilde}
	\big(\tilde{g}_{p,n}^{(\varepsilon)}(t),\tilde{d}_{p,n}^{(\varepsilon)}(t)\big)=\big(g_{p,n}^{(\varepsilon)}(t+1/2),d_{p,n}^{(\varepsilon)}\left(t+1/2\right)\big).
\end{equation*} 

The following lemma lays the foundations.

\begin{lemma}\label{cor:edo_aleat} For $\varepsilon>0$ sufficiently small:
\begin{enumerate}[topsep=0cm]
\item[\emph{1)}] For all $n\geq 1$, there exists on $I_n(\varepsilon)$ a unique solution to \eqref{eq:syst_EDO} starting at time $t=1/2$ from $\left(G_{p,n}\left(\left\lfloor\frac{n}{2}+\varepsilon n\right\rfloor\right)/n, D_{p,n}\left(\left\lfloor\frac{n}{2}+\varepsilon n\right\rfloor\right)/n\right)$, so the couples $\big(g_{p,n}^{(\varepsilon)},d_{p,n}^{(\varepsilon)}\big)$ of Definition \ref{def:f_nd_n_epsilon} and its shifted version $\big(\tilde{g}_{p,n}^{(\varepsilon)},\tilde{d}_{p,n}^{(\varepsilon)}\big)$ are indeed well-defined.
\item[\emph{2)}] There exists a deterministic constant $\bar \kappa_{\varepsilon}>0$  and, for all $A>0$, a deterministic constant $\bar K_{\varepsilon,A}\in \left(C(p)\varepsilon,1\right)$ such that simultaneously for all $t\in \left[0,A+1\right]$ and all $n\geq 1$: 

 \emph{(a)} \hspace{0.1cm} $\tilde{g}_{p,n}^{(\varepsilon)}(t) < \bar K_{\varepsilon,A}$
 
 \emph{(b)} \hspace{0.1cm} $\displaystyle \frac{1-\bar K_{\varepsilon,A}}{2}< \frac{t+1/2+\varepsilon-\tilde{g}_{p,n}^{(\varepsilon)}(t)-\tilde{d}_{p,n}^{(\varepsilon)}(t)}{1-\tilde{g}_{p,n}^{(\varepsilon)}(t)}< \frac{1}{2}-\bar \kappa_{\varepsilon}.$
\end{enumerate}
\end{lemma}

\begin{proof} We use Lemma \ref{lm:syst_edo}. In that aim, introduce
	$$a_n(\varepsilon)=\frac{G_{p,n}\left(\left\lfloor\frac{n}{2}+\varepsilon n\right\rfloor\right)}{n}, \,\,\, b_n(\varepsilon)=\frac{D_{p,n}\left(\left\lfloor\frac{n}{2}+\varepsilon n\right\rfloor\right)}{n} \,\,\, \text{ and }\,  \delta_{n}(\varepsilon)=\frac{\varepsilon-b_n(\varepsilon)-a_n(\varepsilon)^2/2}{\left(1-a_n(\varepsilon)\right)^2}.$$ 
	
1) By definition of $I_n(\varepsilon)$, it is clear that for $\varepsilon>0$ small enough, $0<\delta_{n}(\varepsilon)<a_n(\varepsilon)/2$ on $I_n(\varepsilon).$ The existence and uniqueness of a solution to \eqref{eq:syst_EDO} follows from Lemma \ref{lm:syst_edo} 2). Moreover Lemma \ref{lm:syst_edo} 1) says that $g_{p,n}^{(\varepsilon)}$ is the unique solution to $(E_{(\delta_n(\varepsilon))})$ (this equation is defined, as  (\ref{eq:syst_EDO}), in Section \ref{section:approximation}) starting from $a_n(\varepsilon)$, which will be useful below.

2) First observe that for any $\varepsilon>0$, the couple $(g_p(\cdot+1/2+\varepsilon),d_p(\cdot+1/2+\varepsilon))$ verifies the desired inequalities $(a)$ and $(b)$ for well-chosen constants $\bar \kappa^{0}_{\varepsilon}>0$, $\bar K^{0}_{\varepsilon,A} \in (0,1)$, see Proposition \ref{prop:propfunctions} for details. 

Then observe that for $\varepsilon>0$ small enough, $\delta_n(\varepsilon)< 2\varepsilon$ on $I_n(\varepsilon)$. Consequently, if $\hat{g}_{p,(\varepsilon)}$ denotes the (well-defined) unique solution to $(E_{(2\varepsilon)})$ starting from $C(p)\varepsilon$, since moreover, still on $I_n(\varepsilon)$, ${g}_{p,n}^{(\varepsilon)}$ is the solution to $(E_{(\delta_n(\varepsilon))})$ starting from $a_n(\varepsilon)\leq C(p)\varepsilon$, we have by Lemma \ref{lm:equadiff} that $\tilde{g}_{p,n}^{(\varepsilon)}(t)=g_{p,n}^{(\varepsilon)}(t+1/2)\leq g_{p,n}^{(\varepsilon)}(A+3/2) \leq \hat{g}_{p,(\varepsilon)}(A+3/2)$ for all $t\in \left[0,A+1\right]$.
This yields the upper bound (a) with $\bar K_{\varepsilon,A}=\max(\bar K^{0}_{\varepsilon,A},\hat{g}_{p,(\varepsilon)}(A+1)) \in (0,1)$.

Next, working again on $I_n(\varepsilon)$, we have seen in the proof of Lemma \ref{lm:syst_edo} (up to a time shift of $1/2$) that for every $t\geq 0$
\begin{equation*}\label{eq:r(t)}
	\frac{t+1/2+\varepsilon-\tilde{g}_{p,n}^{(\varepsilon)}(t)-\tilde{d}_{p,n}^{(\varepsilon)}(t)}{1-\tilde{g}_{p,n}^{(\varepsilon)}(t)}= \left(t+1/2+\delta_n(\varepsilon)\right)\big(1-\tilde{g}_{p,n}^{(\varepsilon)}(t)\big)=:\tilde{r}_{p,n}^{(\varepsilon)}(t),
\end{equation*}
leading to the lower bound of $(b)$. Moreover, observe that $\tilde{r}_{p,n}^{(\varepsilon)}$ verifies 
$$\tilde{r}_{p,n}^{(\varepsilon)}(0)\leq \frac{1}{2}-\frac{c(p)\varepsilon}{2}\quad\, \text{and } \quad \left(\tilde{r}_{p,n}^{(\varepsilon)}\right)'(t)=\left(1-\tilde{g}_{p,n}^{(\varepsilon)}(t)\right)\left(1-(t+1/2+\delta_n(\varepsilon))\frac{2p\tilde{g}_{p,n}^{(\varepsilon)}(t)}{1-2\tilde{r}_{p,n}^{(\varepsilon)}(t)}\right), \, t> 0.$$
For $\varepsilon>0$ small enough, by continuity of $\tilde{r}_{p,n}^{(\varepsilon)}$, the time $T:=\inf\big\{t\geq 0:\tilde{r}_{p,n}^{(\varepsilon)}(t)\geq 1/2-\varepsilon^2 \big\}$ (with the usual convention $\inf\{\emptyset\}=\infty$) is strictly positive. It $T$ were finite, there would exist $\eta \in (0,T)$ such that for every $t\in [T-\eta,T]$, $\tilde{r}_{p,n}^{(\varepsilon)}(t)\geq 1/2-2\varepsilon^2$. As $\tilde{g}_{p,n}^{(\varepsilon)}$ is increasing and verifies $\tilde{g}_{p,n}^{(\varepsilon)}(0)\geq (2+c(p))\varepsilon$, it is easy to see that for $\varepsilon>0$ small enough, this would lead to $\big(\tilde{r}_{p,n}^{(\varepsilon)}\big)'(t)\leq 0$ for any $t\in [T-\eta,T]$, which would contradict the definition of $T$. Thus, $T=\infty$ on $I_n(\varepsilon)$ and the claim follows with $\bar \kappa_{\varepsilon}=\min(\bar \kappa^{0}_{\varepsilon},\varepsilon^2)$. 
\end{proof}

Now set
\begin{equation*}
	K_{\varepsilon,A}=\frac{1+\bar K_{\varepsilon,A}}{2}  \quad \text{ and }\quad \kappa_{\varepsilon}=\frac{\bar \kappa_{\varepsilon}}{2},
\end{equation*}
and consider the following domain.

\begin{definition}\label{def:domaine}
	For $\varepsilon>0$ small enough and $A>0$, 
	$$	D_{\varepsilon,A}=\left\{(t,g,d)\in \left(-\varepsilon, A+1\right)\times(0,K_{\varepsilon,A})\times (0,A+1)~:~\frac{1-K_{\varepsilon,A}}{2}<\frac{t+1/2+\varepsilon-d-g}{1-g}<\frac{1}{2}-\kappa_{\varepsilon} \right\}.$$

Additionally, define  for $(t,g,d)\in D_{\varepsilon,A}$,
$$F_{\varepsilon}\left(t,g,d\right)=\frac{2pg(1-g)^2}{1-2(t+1/2+\varepsilon)+g+2d}.$$
\end{definition}

\bigskip

\textbf{Some consequences.} This way, the process $\big(\big(t,\tilde{g}_{p,n}^{(\varepsilon)}(t),\tilde{d}_{p,n}^{(\varepsilon)}(t)\big):t\in \left[0,A\right]\big)$ is confined in a compact subset of $D_{\varepsilon,A}$, according to Lemma \ref{cor:edo_aleat} and the definition of the event $I_n(\varepsilon)$ (which gives strictly positive lower bounds for the processes $\tilde{g}_{p,n}^{(\varepsilon)},\tilde{d}_{p,n}^{(\varepsilon)}$).
The function $F_{\varepsilon}$ is chosen so that $$g'(t+1/2)=F_{\varepsilon}\left(t,g(t+1/2),d(t+1/2)\right),\quad t>0$$ for any couple of functions $(g,d)$ solution to (\ref{eq:syst_EDO}). Moreover, observe that for all $(t,g,d)\in D_{\varepsilon,A}$ the denominator in $F_{\varepsilon}(t,g,d)$ belongs to the interval $\left[(1-K_{\varepsilon,A})2 \kappa_{\varepsilon},K_{\varepsilon,A}\right] \subset (0,1)$,  so that $F_{\varepsilon}$ is well-defined, bounded and Lipschitz continuous on $D_{\varepsilon,A}$. 

Let us now define, for $\varepsilon>0$, the time-translated processes:
\begin{equation*}
\begin{array}{ll}
G_{p,n}^{(\varepsilon)}(m) =G_{p,n}(\lfloor m+(1/2+\varepsilon) n\rfloor) & \qquad D_{p,n}^{(\varepsilon)}(m)=D_{p,n}(\lfloor m+(1/2+\varepsilon) n\rfloor) \\
V_{p,n}^{(\varepsilon)}(m)=V_{p,n}(\lfloor m+(1/2+\varepsilon) n\rfloor) & \qquad E_{p,n}^{(\varepsilon)}(m)=E_{p,n}(\lfloor m+(1/2+\varepsilon) n\rfloor)
\end{array}
\qquad m \geq 0
\end{equation*}
and $\big({\mathbf{F}}^{(\varepsilon)}_{p,n}(m)\big)_{m\in \mathbb Z_+}$ be the associated filtration.
Consider their modifications, 
	\begin{equation*}\label{def:X_tilde}
		\Big(\tilde{G}^{(\varepsilon)}_{p,n}(m),\tilde{D}^{(\varepsilon)}_{p,n}(m)\Big):=
		\left\{
		\begin{array}{ll}
		\left(G_{p,n}(\lfloor m+(1/2+\varepsilon) n\rfloor),D_{p,n}(\lfloor m+(1/2+\varepsilon) n\rfloor)\right) & \text{ on } I_n(\varepsilon)\\
			n\cdot\left(g_p\left(\frac{m}{n}+1/2+\varepsilon\right),d_p\left(\frac{m}{n}+1/2+\varepsilon\right)\right)&\text{ on } I_n(\varepsilon)^c,
		\end{array}
		\right.
	\end{equation*} 
as well as $\tilde{V}^{(\varepsilon)}_{p,n}(m)=n-\tilde{G}^{(\varepsilon)}_{p,n}(m)$, $\tilde{E}^{(\varepsilon)}_{p,n}(m)= \left\lfloor m+(1/2+\varepsilon)n \right\rfloor -\tilde{G}^{(\varepsilon)}_{p,n}(m)-\tilde{D}^{(\varepsilon)}_{p,n}(m)$. Note that these processes are adapted to the filtration ${\mathbf{F}}^{(\varepsilon)}_{p,n}$. These modifications eliminates potential initialization issues since, as already observed, Lemma \ref{cor:edo_aleat} and the definition of $D_{\varepsilon,A}$ imply that $$\left(0, \frac{\tilde{G}^{(\varepsilon)}_{p,n}(0)}{n},\frac{\tilde{D}^{(\varepsilon)}_{p,n}(0)}{n}\right)\in D_{\varepsilon,A}.$$
Consider then the exit time 
\begin{equation}
\label{def:exittime} 
H_{D_{\varepsilon,A}}:=\inf\left\{m\geq 0:\, \left(\frac{m}{n},\frac{\tilde{G}^{(\varepsilon)}_{p,n}(m)}{n},\frac{\tilde{D}^{(\varepsilon)}_{p,n}(m)}{n}\right)\notin D_{\varepsilon,A}  \right\},
\end{equation}
which is a stopping-time with respect to the filtration ${\mathbf F}^{(\varepsilon)}_{p,n}$.
Note the following consequences of the definition of $D_{\varepsilon,A}$: for every $m \in \big[0 ,H_{D_{\varepsilon,A}} \big)$ 
\begin{equation}\label{eq:A1}
	\tilde{V}^{(\varepsilon)}_{p,n}(m)> \left(1-K_{\varepsilon,A}\right)n \tag{A1}
\end{equation}
and for $n$ small enough (depending on $\varepsilon$ an $A$)
\begin{equation}\label{eq:A2}
	\frac{1-{K}_{\varepsilon,A}}{3}<\frac{\tilde{E}^{(\varepsilon)}_{p,n}(m)}{\tilde{V}^{(\varepsilon)}_{p,n}(m)}<\frac{1}{2}-\kappa_{\varepsilon},\tag{A2}
\end{equation}
which implies that the forest part of $\tilde{G}^{(\varepsilon)}_{p,n}(m)$ is subcritical as long as $m\in \big[0,H_{D_{\varepsilon,A}}\big)$, and more precisely that
\begin{equation}\label{eq:A4}
\tilde \Omega_{p,n}^{(\varepsilon)}(m):=\frac{2\tilde E_{p,n}^{(\varepsilon)}(m)-\tilde V_{p,n}^{(\varepsilon)}(m)}{\tilde V_{p,n}^{(\varepsilon)}(m)^{2/3}} \leq -2\kappa_{\varepsilon}(1-K_{\varepsilon,A})^{1/3} \cdot n^{1/3}.\tag{A3}
\end{equation} 

Combining \eqref{eq:A1} with the lower bound of \eqref{eq:A2} also yields for $n$ small enough and any $m\in \big[0,H_{D_{\varepsilon,A}}\big)$,
\begin{equation}\label{eq:A3}
	\tilde{E}^{(\varepsilon)}_{p,n}(m)\geq \frac{(1-K_{\varepsilon,A})^2}{3} \cdot n.\tag{A4}
\end{equation} 

\bigskip

\textbf{Approximation of the conditional jumps.} We have now the material to prove the following lemma, which is a key point of the proof of Proposition \ref{prop:DEM}. 

	\begin{lemma}\label{lm:trend:hyp}
		As $n\to \infty$, uniformly over $m\in \big[0,H_{D_{\varepsilon,A}}\big)$, one has
		$$\left\vert\mathbb{E}\left[\Delta \tilde{G}^{(\varepsilon)}_{p,n}(m) | {\mathbf{F}}^{(\varepsilon)}_{p,n}(m)  \right] -
		F_{\varepsilon}\left(\frac{m}{n},\frac{\tilde{G}^{(\varepsilon)}_{p,n}(m)}{n},\frac{\tilde{D}^{(\varepsilon)}_{p,n}(m)}{n}\right)\right\vert\leq \lambda(n)$$ 
		for a deterministic function $\lambda(n){=}o(1)$.
	\end{lemma}
	\begin{proof}
For $n\geq 1$ and $m\in \big[0,H_{D_{\varepsilon,A}}\big)$, using the definition of $F_{\varepsilon}$, we have that
\begin{align}
		&\bigg\vert\mathbb{E}\Big[\Delta \tilde{G}^{(\varepsilon)}_{p,n}(m) | {\mathbf{F}}^{(\varepsilon)}_{p,n}(m)  \Big] -	F_{\varepsilon}\left(\frac{m}{n},\frac{\tilde{G}^{(\varepsilon)}_{p,n}(m)}{n},\frac{\tilde D_{p,n}(m)}{n}\right) \bigg\vert \label{3terms}\\
			\leq & \left\vert\mathbb{E}\left[\Delta G_{p,n}^{(\varepsilon)}(m) | {\mathbf{F}}^{(\varepsilon)}_{p,n}(m)  \right] -\frac{2pG_{p,n}^{(\varepsilon)}(m)(n-G_{p,n}^{(\varepsilon)}(m))}{n^2\left(1-2E_{p,n}^{(\varepsilon)}(m)/V_{p,n}^{(\varepsilon)}(m)\right)}\right\vert\mathbbm{1}_{\{I_n(\varepsilon)\}} \notag \\	
	+& \left \vert \frac{2pG_{p,n}^{(\varepsilon)}(m)(n-G_{p,n}^{(\varepsilon)}(m))^2}{n^3\left(V_{p,n}^{(\varepsilon)}(m)-2\big(m/n+1/2+\varepsilon-G_{p,n}^{(\varepsilon)}(m)-D_{p,n}^{(\varepsilon)}(m)\big)\right)} -\frac{2pG_{p,n}^{(\varepsilon)}(m)(n-G_{p,n}^{(\varepsilon)}(m))^2}{n^3\big(V_{p,n}^{(\varepsilon)}(m)-2E_{p,n}^{(\varepsilon)}(m)\big)} \right\vert \mathbbm{1}_{\{I_n(\varepsilon)\}} \notag\\		
	+&\left\lvert n\left(g_p\left(\frac{m+1}{n}+\frac{1}{2}+\varepsilon\right)-g_p\left(\frac{m}{n}+\frac{1}{2}+\varepsilon\right)\right) -F_{\varepsilon}\left(\frac{m}{n},g_p\left(\frac{m}{n}+\frac{1}{2}+\varepsilon\right),d_p\left(\frac{m}{n}+\frac{1}{2}+\varepsilon\right)\right) \right\rvert \mathbbm{1}_{\{I_n(\varepsilon)^c\}}. \notag
\end{align}

We will show that each of the three terms in the right-hand side of (\ref{3terms}) is deterministically bounded by a function independent of $m$ that converges to 0 as $n \rightarrow \infty$.
		
1) We start by bounding from above the last term: since $\left(g_p(\cdot+\varepsilon),d_p(\cdot+\varepsilon)\right)$ is a solution to \eqref{eq:syst_EDO}, one has for every $m \in \big[0,H_{D_{\varepsilon,A}}\big)$,
\begin{align*}
n\left(g_p\left(\frac{m+1}{n}+\frac{1}{2}+\varepsilon\right)-g_p\left(\frac{m}{n}+\frac{1}{2}+\varepsilon\right)\right)&=\int_{m}^{m+1}F_{\varepsilon}\left(\frac{s}{n},g_p\left(\frac{s}{n}+\frac{1}{2}+\varepsilon\right),d_p\left(\frac{s}{n}+\frac{1}{2}+\varepsilon\right)\right)\mathrm{d}s\\
&=F_{\varepsilon}\left(\frac{m}{n},g_p\left(\frac{m}{n}+\frac{1}{2}+\varepsilon\right),d_p\left(\frac{m}{n}+\frac{1}{2}+\varepsilon\right)\right)+O\left(\frac{1}{n}\right)
\end{align*}
where the $O\left(1/n\right)$ is uniform over $m \in \big[0,H_{D_{\varepsilon,A}}\big)$, as a consequence of the Lipschitz continuity of $F_{\varepsilon}$ over $D_{\varepsilon,A}$ and the fact that the derivatives of $g_p$ and $d_p$ are bounded. 

2) Regarding the middle term  in the right-hand side of (\ref{3terms}), we note, using e.g. (\ref{eq:A4}), that it is bounded from above by a constant times 
$$
\frac{m/n+1/2+\varepsilon-\lfloor m/n+1/2+\varepsilon \rfloor}{n^2} \leq \frac{1}{n^2}
$$
for all $m\in \big[0,H_{D_{\varepsilon,A}}\big)$.

3) Last, on the event $I_n(\varepsilon)$, for every $m\in \big[0,H_{D_{\varepsilon,A}}\big)$,
		\begin{align*}
		  &\left\vert\mathbb{E}\left[\Delta G_{p,n}^{(\varepsilon)}(m) | {\mathbf{F}}^{(\varepsilon)}_{p,n}(m)  \right] -\frac{2pG_{p,n}^{(\varepsilon)}(m)(n-G_{p,n}^{(\varepsilon)}(m))}{n^2\big(1-2E_{p,n}^{(\varepsilon)}(m)/V_{p,n}^{(\varepsilon)}(m)\big)}\right\vert\\
			\leq &\left\vert \mathbb{E}\left[\Delta G_{p,n}^{(\varepsilon)}(m)\mathbbm{1}_{\{\Delta G_{p,n}^{(\varepsilon)}(m)\leq V_{p,n}^{(\varepsilon)}(m)^{1/4}\}} | {\mathbf{F}}^{(\varepsilon)}_{p,n}(m)  \right]-\frac{2p G_{p,n}^{(\varepsilon)}(m)(n-G_{p,n}^{(\varepsilon)}(m))}{n^2}S_{V_{p,n}^{(\varepsilon)}(m)^{1/4}}\left(\frac{2E_{p,n}^{(\varepsilon)}(m)}{V_{p,n}^{(\varepsilon)}(m)}\right) \right\vert \\
			&+  \mathbb{E}\left[\Delta G_{p,n}^{(\varepsilon)}(m)\mathbbm{1}_{\{\Delta G_{p,n}^{(\varepsilon)}(m)> V_{p,n}^{(\varepsilon)}(m)^{1/4}\}} | {\mathbf{F}}^{(\varepsilon)}_{p,n}(m)  \right] \\
			&+ 2p \left\vert \frac{1}{1-2E_{p,n}^{(\varepsilon)}(m)/V_{p,n}^{(\varepsilon)}(m)}-S_{V_{p,n}^{(\varepsilon)}(m)^{1/4}}\left(\frac{2E_{p,n}^{(\varepsilon)}(m)}{V_{p,n}^{(\varepsilon)}(m)}\right)  \right\vert 
		\end{align*}
		where $S_{N}$ is defined in \eqref{def:S_N} for $N \in \mathbb N$. It remains to show that on the event $I_n(\varepsilon)$ and uniformly over $m<H_{D_{\varepsilon,A}}$, each of the three terms in the right-hand side of this inequality is smaller than a deterministic function of $n$ that converges to 0 as $n \rightarrow \infty$. In the lines below we implicitly work on   $I_n(\varepsilon)$ and with  $m\in \big[0,H_{D_{\varepsilon,A}}\big)$.
		
	\hspace{0.5cm} $\bullet$ By \eqref{eq:A3}, $E_{p,n}^{(\varepsilon)}(m)$ is greater than a constant times $n$, and by \eqref{eq:A4}, $\Omega_{p,n}^{(\varepsilon)}(m)$ is smaller than a negative constant times $n^{1/3}$. Since moreover $V_{p,n}^{(\varepsilon)}(m)\leq n$, Corollary \ref{corol:est_transitions} 1) applies and implies that uniformly in $1 \leq k \leq V_{p,n}^{(\varepsilon)}(m)^{1/4}$,
	\begin{align*}
		& \left\vert \mathbb{P}\big(\Delta G_{p,n}^{(\varepsilon)}(m)=k | {\mathbf{F}}^{(\varepsilon)}_{p,n}(m)\big)-\frac{k^{k-1}}{k!}\left(\frac{2E_{p,n}^{(\varepsilon)}(m)}{V_{p,n}^{(\varepsilon)}(m)}\right)^{k-1}\mathrm{e}^{-2k\frac{E_{p,n}^{(\varepsilon)}(m)}{V_{p,n}^{(\varepsilon)}(m)}}\left(\frac{2pG_{p,n}^{(\varepsilon)}(m)\left(n-G_{p,n}^{(\varepsilon)}(m)\right)}{n^2}\right)\right\vert  \\
		&\leq \left(o_{\big(E_{p,n}^{(\varepsilon)}(m),\Omega_{p,n}^{(\varepsilon)}(m)\big)}(1)+O_{\big(E_{p,n}^{(\varepsilon)}(m),\Omega_{p,n}^{(\varepsilon)}(m)\big)}(1) \cdot \frac{n^{1/4}}{n}\right) \cdot \frac{k^{k-1}}{k!}\left(\frac{2E_{p,n}^{(\varepsilon)}(m)}{V_{p,n}^{(\varepsilon)}(m)}\right)^{k-1}\mathrm{e}^{-2k\frac{E_{p,n}^{(\varepsilon)}(m)}{V_{p,n}^{(\varepsilon)}(m)}}.
	\end{align*}
	There thus exists a deterministic $\lambda_1(n)$ (independent on $m \in \big[0,H_{D_{\varepsilon,A}}\big)$) that converges to 0 as $n \rightarrow \infty$ such that	
	\begin{align*}
		\Bigg\vert \mathbb{E}&\Big[\Delta G_{p,n}^{(\varepsilon)}(m)\mathbbm{1}_{\{\Delta G_{p,n}^{(\varepsilon)}(m)\leq V_{p,n}^{(\varepsilon)}(m)^{1/4}\}} | {\mathbf{F}}^{(\varepsilon)}_{p,n}(m)  \Big]-\frac{2p G_{p,n}^{(\varepsilon)}(m)\big(n-G_{p,n}^{(\varepsilon)}(m)\big)}{n^2}S_{V_{p,n}^{(\varepsilon)}(m)^{1/4}}\left(2\frac{E_{p,n}^{(\varepsilon)}(m)}{V_{p,n}^{(\varepsilon)}(m)}\right) \Bigg\vert\\
		&\leq \lambda_1(n) \cdot \sum_{k\geq 1}\frac{k^{k}}{k!}\left(\frac{2E_{p,n}^{(\varepsilon)}(m)}{V_{p,n}^{(\varepsilon)}(m)}\right)^{k-1}\mathrm{e}^{-2k\frac{E_{p,n}^{(\varepsilon)}(m)}{V_{p,n}^{(\varepsilon)}(m)}} \leq \lambda_1\left(n\right) \cdot (2\kappa_{\varepsilon})^{-1}
	\end{align*}
	where we used (\ref{lm:Borel:law}) and the upper bound of the inequality \eqref{eq:A2} to get the last line. 
		
	\hspace{0.5cm} $\bullet$ As we just implicitly said, \eqref{eq:A2} implies
$2E_{p,n}^{(\varepsilon)}(m)-V_{p,n}^{(\varepsilon)}(m) \leq -2\kappa_{\varepsilon}V_{p,n}^{(\varepsilon)}(m)$. Recall also that, by (\ref{eq:A1}),$V_{p,n}^{(\varepsilon)}(m)$ is deterministically bounded from below by a constant times $n$. Lemma \ref{lm:forest:sub} and the free forest property of Proposition \ref{prop:freeforest} thus imply that for $n$ large enough
\begin{align}
\label{maj:esp_sauts}
	\mathbb{E}\bigg[\Delta G_{p,n}^{(\varepsilon)}(m)\mathbbm{1}_{\{\Delta G_{p,n}^{(\varepsilon)}(m)>V_{p,n}^{(\varepsilon)}(m)^{1/4}\}}\vert {\mathbf{F}}^{(\varepsilon)}_{p,n}(m)\Big]
	&\leq n\,\mathbb{P}\Big(\Delta G_{p,n}^{(\varepsilon)}(m)>V_{p,n}^{(\varepsilon)}(m)^{1/4} \vert {\mathbf{F}}^{(\varepsilon)}_{p,n}(m) \Big)\notag\\
	&\leq c_1 nV_{p,n}^{(\varepsilon)}(m)^{3/2}\exp\big(-c_2V_{p,n}^{(\varepsilon)}(m)^{1/4}\big) \notag\\
	&\leq c_1n^{5/2}\exp\big(-c'_2 n^{1/4}\big),
\end{align}	
where $c_1,c_2,c'_2$ belong to $(0,\infty)$ and only depend on $\varepsilon$ and $A$.

	\hspace{0.5cm}  $\bullet$ Using again that the ratio $E_{p,n}^{(\varepsilon)}(m)/V_{p,n}^{(\varepsilon)}(m)$ belongs to the compact $\left[0,1/2-\kappa_{\varepsilon}\right]$, we can bound from above the last term by a deterministic function that converges to 0 as $n\rightarrow \infty$, thanks to Lemma \ref{lm:partial_sum} 1) and the fact, by \eqref{eq:A1}, that $V_{p,n}^{\varepsilon}(m) \geq (1-K_{\varepsilon,A})n$.
\end{proof}

\subsubsection{Proof of Proposition \ref{prop:DEM}}
\label{proof:Prop46}

We have now the material to apply Theorem \ref{thmDEM} to the process $\big(\tilde{G}^{(\varepsilon)}_{p,n},\tilde{D}^{(\varepsilon)}_{p,n}\big)$ introduced in the preliminaries, working on the domain $D_{\varepsilon,A}$ of Definition \ref{def:domaine} for some fixed $A>0$. Let us check all the hypotheses of the theorem. First, both processes $\tilde{G}^{(\varepsilon)}_{p,n},\tilde{D}^{(\varepsilon)}_{p,n}$ are uniformly bounded by $n$. Then, recall the definition of $F_{\varepsilon}$  (Definition \ref{def:domaine}) and set 
$$
G(t,g,d)=2(1-p)g(1-g)+g^2, \quad \text{ for }(t,g,d) \in D_{\varepsilon,A}.
$$ 
The functions $F_{\varepsilon}$ and $G$ are Lipschitz continuous on $D_{\varepsilon,A}$: this is obvious for $G$ and was discussed for $F_{\varepsilon}$ in the paragraph after Definition \ref{def:domaine}. Besides, by construction, 
$D_{\varepsilon,A}$ contains the closure of the set 
$$\left\{\left(0,g,d\right):\mathbb{P}\left(\big(\tilde{G}^{(\varepsilon)}_{p,n}(0),\tilde{D}^{(\varepsilon)}_{p,n}(0)\big)=n\cdot(g,d)\right)\neq 0 \text{ for some }n \right\}.$$
Next, recall the definition (\ref{def:exittime}) of the exit time $H_{D_{\varepsilon,A}}$, and check:

	\begin{enumerate}[leftmargin=0.8cm, topsep=0cm]
			\item[$\bullet$] \textbf{The boundedness hypothesis.}
			Since $g_p'$ is bounded on $\left[1/2,A+3/2\right]$, we have for $n$ large enough
	\begin{equation*}
		\mathbb{P}\left(\exists m \in \big[0,H_{D_{\varepsilon,A}}\big):\Delta \tilde{G}^{(\varepsilon)}_{p,n}(m)>n^{1/4}   \right)=\mathbb{P}\left(\exists m \in \big[0,H_{D_{\varepsilon,A}}\big):\Delta G_{p,n}^{(\varepsilon)}(m)>n^{1/4},I_n(\varepsilon) \right).
	\end{equation*}
	Consequently, proceeding in a similar way to (\ref{maj:esp_sauts}), recalling also (\ref{eq:A1}), we deduce that
\begin{align*}
	\mathbb{P}\left(\exists m \in \big[0,H_{D_{\varepsilon,A}}\big)\right.&\Big.:\Delta \tilde{G}^{(\varepsilon)}_{p,n}(m)>n^{1/4}   \Big) \\
&\leq \sum_{m=0}^{\lfloor(A+1)n\rfloor}\mathbb{E}\left[\mathbb{P}\left(\Delta {G}^{(\varepsilon)}_{p,n}(m)>n^{1/4}, m < H_{D_{\varepsilon,A}}\vert {\mathbf{F}}^{(\varepsilon)}_{p,n}(m)\right)\mathbbm 1_{I_n(\varepsilon)}\right]\\
				&\leq  c_1n^{5/2}\exp\big(-c_2 n^{1/4}\big),
			\end{align*}
			where $c_1,c_2$ belong to $(0,\infty)$ and only depend on $\varepsilon$ and $A$. Besides, $\lvert \Delta \tilde{D}^{(\varepsilon)}_{p,n}(m)\rvert\leq 1$ for all $m\geq 0$. This therefore yields the boundedness hypothesis of Theorem \ref{thmDEM}, with, keeping the notation introduced there, $\beta(n)=n^{1/4}$ and $\gamma(n)=c'_1n^{5/2}\exp\big(-c'_2 n^{1/4}\big)$.
			\item[$\bullet$] \textbf{The trend hypothesis.} By Lemma \ref{lm:trend:hyp}, for every $m\in \big[0,H_{D_{\varepsilon,A}}\big)$
			\begin{align*}
	\left\vert\mathbb{E}\left[\Delta \tilde{G}^{(\varepsilon)}_{p,n}(m) | {\mathbf{F}}^{(\varepsilon)}_{p,n}(m)  \right] -F_{\varepsilon}\left(\frac{m}{n},\frac{\tilde{G}^{(\varepsilon)}_{p,n}(m)}{n},\frac{\tilde{D}^{(\varepsilon)}_{p,n}(m)}{n}\right)\right\vert
	&\leq \lambda(n)
	\end{align*} 
with $\lambda(n)=o(1)$. With a similar (but much simpler) approach,  we get that simultaneously for all $m\in \big[0,H_{D_{\varepsilon,A}}\big)$, thanks to Proposition \ref{lm_transitions} and the definition of $d_p$ via \eqref{eq:syst_EDO},
			$$\left\vert \mathbb{E}\left[\Delta \tilde{D}^{(\varepsilon)}_{p,n}(m) | {\mathbf{F}}^{(\varepsilon)}_{p,n}(m)  \right] -G\left(\frac{m}{n},\frac{\tilde{G}^{(\varepsilon)}_{p,n}(m)}{n},\frac{\tilde{D}^{(\varepsilon)}_{p,n}(m)}{n}\right) \right \vert= O\left(\frac{1}{n} \right)$$
for a deterministic $O(1/n)$ independent of $m\in \big[0,H_{D_{\varepsilon,A}}\big)$.			
		
		\end{enumerate}
Finally recall from Definition \ref{def:f_nd_n_epsilon} and Lemma \ref{cor:edo_aleat} that $\big(\tilde{g}_{p,n}^{(\varepsilon)},\tilde{d}_{p,n}^{(\varepsilon)}\big)$ is the unique solution to equation (\ref{eq:syst_EDO}) starting from $\big(\tilde{G}^{(\varepsilon)}_{p,n}(0)/n,\tilde{D}^{(\varepsilon)}_{p,n}(0)/n\big)$ and that $\big(\big(t,\tilde{g}_{p,n}^{(\varepsilon)}(t),\tilde{d}_{p,n}^{(\varepsilon)}(t)\big):t\in \left[0,A\right]\big)$ is confined in a compact subset of $D_{\varepsilon,A}$. All in all, all this implies that the conclusion (b) of Theorem \ref{thmDEM} holds with $\eta(n)=\max(\lambda(n)+n^{-1/2}+2n\gamma(n);n^{-1/13})=o(1)$ and $\sigma(n)=A$, yielding that, with high probability as $n\to \infty$,
\begin{equation*}
	\tilde{G}^{(\varepsilon)}_{p,n}(m)=n\cdot \tilde{g}_{p,n}^{(\varepsilon)}(m/n)+o(n) \quad\text{ and }\quad  \tilde{D}^{(\varepsilon)}_{p,n}(m)=n\cdot \tilde{d}_{p,n}^{(\varepsilon)}(m/n)+o(n),
\end{equation*}
uniformly over $m\in [0,An]$. Since $\big(\tilde{G}^{(\varepsilon)}_{p,n},\tilde{D}^{(\varepsilon)}_{p,n}\big)$ coincides with $\big(G_{p,n}^{(\varepsilon)},D_{p,n}^{(\varepsilon)}\big)$ on $I_n(\varepsilon)$, whose probability tends to $1$ as $n$ tends to infinity, and since $\tilde{g}_{p,n}^{(\varepsilon)}\left(\left\lfloor nt\right\rfloor/n\right)-\tilde{g}_{p,n}^{(\varepsilon)}(t)=O\left(1/n\right)$ uniformly on compact sets, this in turn yields that
\begin{equation*}
		\left(\left(\left\lvert \frac{G_{p,n}\left(\lfloor nt+(1/2+\varepsilon) n\rfloor\right)}{n}-\tilde{g}_{p,n}^{(\varepsilon)}(t) \right\vert, \left\lvert \frac{D_{p,n}\left(\lfloor nt+(1/2+\varepsilon) n\rfloor\right)}{n}-\tilde{d}_{p,n}^{(\varepsilon)}(t) \right\vert \right)\right)_{t\geq 0}~\underset{n \rightarrow \infty }{\overset{\mathbb{P}}{\longrightarrow}}~0
\end{equation*}
for the topology of uniform convergence on compacts, which is equivalent to the statement of the proposition.

	\subsection{Proof of Theorem \ref{thm:fluid limit}}
	\label{sec:proofFL}		

\emph{Convergence of the rescaled gel $G_{p,n}(\lfloor n~\cdot \rfloor)/n$.} Our goal is to prove that for any (large) $A>0$ and any (small) $\delta>0$,
$$
\mathbb P\left(\sup_{t \in [0,A]}\left \vert \frac{G_{p,n}\left(\lfloor nt\rfloor\right)}{n} -g_p(t) \right\vert >\delta\right) \underset{n \rightarrow \infty}\longrightarrow 0. 
$$
Recall the definition of the event $I_n(\varepsilon)$ in (\ref{def:I_n(epsilon)}) and that for $\varepsilon>0$ sufficiently small $\mathbb P(I_n(\varepsilon))\rightarrow 1$ as $n \rightarrow \infty$. It is therefore sufficient to prove that for a well-chosen, small, $\varepsilon>0$ (that may depend on $A$ and $\delta$),
\begin{equation}
\label{cv_X_A}
\mathbb P\left(\sup_{t \in [0,A]}\left \vert \frac{G_{p,n}\left(\lfloor nt\rfloor\right)}{n} -g_p(t) \right\vert >\delta, I_n(\varepsilon)\right) \underset{n \rightarrow \infty}\longrightarrow 0. 
\end{equation}
In that aim, consider $g_{p,n}^{(\varepsilon)}$ as defined in Definition \ref{def:f_nd_n_epsilon} and write
\begin{eqnarray*}
\sup_{t \in [0,A]}\left \vert \frac{G_{p,n}\left(\lfloor nt\rfloor\right)}{n} -g_p(t) \right\vert &\leq& \sup_{t \in \left[0,\frac{1}{2}+\varepsilon\right]}\left \vert \frac{G_{p,n}\left(\lfloor nt\rfloor\right)}{n} -g_p(t) \right\vert +\sup_{t \in \left[\frac{1}{2}+\varepsilon,A\right]}\left \vert \frac{G_{p,n}\left(\lfloor nt\rfloor\right)}{n} -g_p(t) \right\vert \\
&\leq &  \frac{G_{p,n}\big(\big\lfloor n\big(\frac{1}{2}+\varepsilon\big)\big\rfloor\big)}{n} +g_p\left(\frac{1}{2}+\varepsilon\right) + \sup_{t \in \left[\frac{1}{2},A-\varepsilon\right]}\left \vert \frac{G_{p,n}\left(\lfloor n(t+\varepsilon)\rfloor\right)}{n} -g^{(\varepsilon)}_{p,n}(t) \right\vert \\
&&+ \sup_{t \in \left[\frac{1}{2},A-\varepsilon\right]}\left \vert  g^{(\varepsilon)}_{p,n}(t) -g_p(t+\varepsilon) \right\vert
\end{eqnarray*}
where we used that the process $G_{p,n}$ and the function $g_p$ are non-decreasing. 

On $I_n(\varepsilon)$, since $G_{p,n}\big(\big\lfloor n(1/2+\varepsilon)\big\rfloor\big) \leq C(p)\varepsilon n$ and $g_p$ is continuous on $[1/2,\infty)$, we have that for $\varepsilon$ (determinist and independent of $n$) small enough
$$
 \frac{G_{p,n}\big(\big\lfloor n(\frac{1}{2}+\varepsilon)\big\rfloor\big)}{n} +g_p\left(\frac{1}{2}+\varepsilon\right) \leq C(p)\varepsilon+g_p\left(\frac{1}{2}+2\varepsilon\right) \leq \frac{\delta}{4}.
$$
(We will need the ``$+2\varepsilon$" later on.) Besides, from Lemma  \ref{cor:edo_aleat} and its proof, we know that on $I_n(\varepsilon)$, for $\varepsilon$ (determinist, independent of $n$) sufficiently small, $g_{p,n}^{(\varepsilon)}$ is the unique solution to the equation $(E_{(\delta_n(\varepsilon))})$ starting from $g_{p,n}^{(\varepsilon)}(1/2)= G_{p,n}\big(\lfloor n(1/2+\varepsilon)\rfloor\big)/n$, where 
$$0<\delta_n(\varepsilon)=\frac{\varepsilon-d_{p,n}^{(\varepsilon)}(1/2)-g_{p,n}^{(\varepsilon)}(1/2)^2/2}{(1-g_{p,n}^{(\varepsilon)}(1/2))^2} \leq 2 \varepsilon$$
(with $d_{p,n}^{(\varepsilon)}(1/2)=D_{p,n}\big(\big\lfloor n(1/2+\varepsilon)\big\rfloor\big)/n$). Since $g(\cdot+\delta_n(\varepsilon))$ is also a solution to $(E_{(\delta_n(\varepsilon)}))$, Lemma \ref{lm:equadiff} 2) then yields
\begin{align*}
	\sup_{t\in \left[\frac{1}{2},A-\varepsilon\right]} \left \vert g_{p,n}^{(\varepsilon)}\left(t\right)- g_p\left(t+\delta_n(\varepsilon)\right)\right\vert&\leq \left\vert g_{p,n}^{(\varepsilon)}\left(\frac{1}{2}\right)- g_p\left(\frac{1}{2}+\delta_n(\varepsilon)\right)\right\vert \\
	 &\leq  \frac{G_{p,n}\big(\big\lfloor n(\frac{1}{2}+\varepsilon)\big\rfloor\big)}{n}+g_p\left(\frac{1}{2}+2 \varepsilon \right) \leq  \frac{\delta}{4}.
\end{align*}
To complete, the mean value theorem gives for $\varepsilon  \in (0, 1]$,
$$\sup_{t\in \left[\frac{1}{2},A-\varepsilon\right]}\left\vert g_p\left(t+\delta_n(\varepsilon)\right) -g_p\left(t+\varepsilon\right)\right\vert\leq \sup_{t\in \left[\frac{1}{2},A+1\right]}\lvert g_p'(t)\rvert  \varepsilon.$$
All in all, we have proved so far that on $I_n(\varepsilon)$, for $\varepsilon$ (determinist, independent of $n$) sufficiently small,
$$
\sup_{t \in [0,A]}\left \vert \frac{G_{p,n}\left(\lfloor nt\rfloor\right)}{n} -g_p(t) \right\vert \leq \frac{3 \delta}{4} + \sup_{t \in \left[\frac{1}{2},A-\varepsilon\right]}\left \vert \frac{G_{p,n}\left(\lfloor n(t+\varepsilon)\rfloor\right)}{n} -g^{(\varepsilon)}_{p,n}(t) \right\vert.
$$
By Proposition \ref{prop:DEM}, the supremum in the right-hand side converges in probability to 0, leading to (\ref{cv_X_A}).

\bigskip

\emph{Convergence of the rescaled number of discarded edges $D_{p,n}(\lfloor n~\cdot \rfloor)/n$.} We proceed similarly. For fixed $A>0$ and $\delta>0$, our goal is to prove that for $\varepsilon$ small enough
$$
\mathbb P\left(\sup_{t \in [0,A]}\left \vert \frac{D_{p,n}\left(\lfloor nt\rfloor\right)}{n} -d_p(t) \right\vert >\delta, I_n(\varepsilon)\right) \underset{n \rightarrow \infty}\longrightarrow 0. 
$$
As above, we use the triangular inequality to get
\begin{eqnarray*}
\sup_{t \in [0,A]}\left \vert \frac{D_{p,n}\left(\lfloor nt\rfloor\right)}{n} -d_p(t) \right\vert 
&\leq &  \frac{D_{p,n}\big(\big\lfloor n\big(\frac{1}{2}+\varepsilon\big)\big\rfloor\big)}{n} +d_p\left(\frac{1}{2}+\varepsilon\right) + \sup_{t \in \left[\frac{1}{2},A-\varepsilon\right]}\left \vert \frac{D_{p,n}\left(\lfloor n(t+\varepsilon)\rfloor\right)}{n} -d^{(\varepsilon)}_{p,n}(t) \right\vert \\
&&+ \sup_{t \in \left[\frac{1}{2},A-\varepsilon\right]}\left \vert  d^{(\varepsilon)}_{p,n}(t) -d_p(t+\varepsilon) \right\vert.
\end{eqnarray*}
Regarding the initialization at time $1+2+\varepsilon$, we proceed as with $G_{p,n}$, recalling that $D_{p,n}(1/2+\varepsilon)$ is smaller than a constant times $\varepsilon^2 n$ on $I_n(\varepsilon)$ and that $d_p(1/2)=0$. Recall next that for $t\geq \frac{1}{2}$,
$$d_p'(t+\varepsilon)=2(1-p)g_p(t+\varepsilon)(1-g_p(t+\varepsilon))+g_p^2(t+\varepsilon), \quad \big(d_{p,n}^{(\varepsilon)}\big)'(t)=2(1-p)g_{p,n}^{(\varepsilon)}(t)\big(1-g_{p,n}^{(\varepsilon)}(t)\big)+\big(g_{p,n}^{(\varepsilon)}\big)^2(t).$$
Since we have seen that $\sup_{t\in\left[\frac{1}{2},A-\varepsilon \right]}\big\lvert g_{p,n}^{(\varepsilon)}\left(t\right)-g_p(t+\varepsilon) \big\rvert$ can be made arbitrarily small on $I_n{(\varepsilon)}$ and since the functions $g_{p,n}^{(\varepsilon)},g_p$ are positive and bounded from above by 1, we have that, on $I_n(\varepsilon)$, for $\varepsilon$ (determinist, independent of $n$) sufficiently small,
$$
\sup_{t \in [0,A]}\left \vert \frac{D_{p,n}\left(\lfloor nt\rfloor\right)}{n} -d_p(t) \right\vert \leq \frac{3 \delta}{4} + \sup_{t \in \left[\frac{1}{2},A-\varepsilon\right]}\left \vert \frac{D_{p,n}\left(\lfloor n(t+\varepsilon)\rfloor\right)}{n} -d^{(\varepsilon)}_{p,n}(t) \right\vert
$$
and we conclude with Proposition \ref{prop:DEM}.

\subsection{Proof of Theorem \ref{thm:forest}}
\label{sec:proofF}

Theorem \ref{thm:forest} is a consequence of Theorem \ref{thm:fluid limit}. The convergence of the triplet \linebreak $\big(V_{p,n}(\lfloor n~\cdot \rfloor)/n, E_{p,n}(\lfloor n~\cdot \rfloor)/n, R_{p,n}(\lfloor n~\cdot \rfloor)\big)$ follows immediately from Theorem \ref{thm:fluid limit}, together with the relations (\ref{def:V_{p,n}E_{p,n}}) and the properties of the functions $v_p,e_p,r_p$ highlighted in Proposition \ref{prop:propfunctions}.

Regarding the number of trees $N_{p,n}^{(k)}$ of size $k$, $k \in \mathbb N$, we first note that since the process
$$\sum_{k\geq 1}\frac{kN_{p,n}^{(k)}(\lfloor n~\cdot \rfloor)}{n}=\frac{V_{p,n}(\lfloor n~\cdot \rfloor)}{n}~\underset{n\rightarrow \infty}{\overset{\mathbb P}\longrightarrow}v_p=1-g_p=\sum_{k\geq 1} kt_{p,k}$$
it is sufficient to prove  separately the convergence for each $k \in \mathbb N$ of $~kN_{p,n}^{(k)}(\lfloor n \cdot \rfloor)/n~$ to $t_{p,k}$ to get the convergence  of the sequence $(kN_{p,n}^{(k)}(\lfloor n \cdot \rfloor)/n,k\geq 1)$ for the norm $\|\cdot \|_1$ in $\ell^1$. 

Then, note from the dynamic of the $p$-frozen model that for all $k \in \mathbb N$ and all $m \in \mathbb N$
\begin{equation}\label{eq:esp_N_k}
	\mathbb{E}\left[\Delta N_{p,n}^{(k)}(m) | \mathbf{F}^{(N)}_{p,n}(m)\right]=\sum_{i+j=k}ij\frac{N_{p,n}^{(i)}(m)N_{p,n}^{(j)}(m)}{n^2}-2k\frac{N_{p,n}^{(k)}(m)\left(n-(1-p)G_{p,n}(m)\right)}{n^2}+O\left(\frac{1}{n}\right)
\end{equation}
where $ \mathbf{F}^{(N)}_{p,n}$ denotes the filtration generated by the sequence $(N_{p,n}^{(k)})_{k\geq 1}$, 
for a deterministic $O(1/n)$ that only depends on $k$, not on $m$. We will proceed by induction on $k$ to get the convergence of $N_{p,n}^{(k)}(\lfloor n \cdot\rfloor)/n$ to the function $t_{p,k}$, relying on an approximation of the gel $G_{p,n}$ by its fluid limit. 

We prove in detail the initial step, when $k=1$, the induction step will then proceed similarly. Fix $A>0$ and define for any integer $m \in \left[0,An\right]$
$$Y_n^{(1)}(m)=\frac{N_{p,n}^{(1)}(m)}{n}-t_{p,1}\left(\frac{m}{n}\right),$$
where we recall that 
$$
t_{p,1}(t)=\left(1-g_p(t)\right)e^{-2t \left(1-g_p(t)\right)}, \qquad t\geq 0.
$$
It is easy to see, using that $g_p$ is solution to the equation (\ref{eq:f}), that $t_{p,1}'(t)=-2t_{p,1}(t) \left(1-(1-p)g_p(t) \right)$ for all $t\geq 0$ -- which is bounded as well as its derivative since $g_p'$ is bounded -- so we have for every $m \in \mathbb N$, by Taylor's expansion,
\begin{align}
\label{deriv:t_p}
	 t_{p,1}\left(\frac{m+1}{n}\right)-t_{p,1}\left(\frac{m}{n}\right)=-\frac{2}{n}\cdot t_{1,p}\left(\frac{m}{n}\right)\left(1-(1-p)g_p\left(\frac{m}{n}\right)\right)+O\left(\frac{1}{n^2}\right).
\end{align}
In particular the jumps of $Y_n^{(1)}$ are deterministically bounded by a $O(1/n)$.
Then  consider for  $\delta>0$ the event
$$J_n(\delta)=\left\{ \sup_{0\leq m\leq \lfloor An \rfloor}\left\lvert \frac{G_{p,n}(m)}{n}-g_p\left(\frac{m}{n}\right)\right\rvert \leq \delta \right\}\cap \left\{\sup_{0\leq m\leq \lfloor An \rfloor} \left\lvert  Y_n^{(1)}(m)-\sum_{i=0}^{m-1}	\mathbb{E}\left[\Delta Y_n^{(1)}(i) | \mathbf{F}^{(N)}_{p,n}(i)\right] \right\rvert \leq \delta \right\}.$$
The convergence in probability of the rescaled process $G_{p,n}(\lfloor n \cdot \rfloor )/n$ to $g_p$ and Lemma \ref{lm:approx:esp} imply that $\mathbb P(J_n(\delta)) \rightarrow 1$ as $n \rightarrow \infty$. Next, on $J_n(\delta)$, using (\ref{eq:esp_N_k}) for $k=1$ and (\ref{deriv:t_p}), we see that for every $n$ deterministic large enough and every $m\in \left[0,An\right]$, we have 
\begin{align*}
	\left\lvert\mathbb{E}\left[\Delta Y_n^{(1)}(m)\vert \mathbf{F}^{(N)}_{p,n}(m)\right]\right\rvert
	&\leq \frac{2}{n}\left\lvert Y_n^{(1)}(m)\right\rvert+\frac{3\delta}{n},
\end{align*}
which in turn leads to
\begin{align*}
	\lvert Y_n^{(1)}(i)\rvert \leq  \frac{2}{n}\sum_{i=0}^{m-1}\lvert Y_n^{(1)}(j)\rvert + \delta (3A+1).
\end{align*}
We then conclude with the discrete version of Grönwall's Lemma that
$\sup_{0\leq m\leq \lfloor An \rfloor}\lvert Y_n^{(1)}(m)\rvert\leq$ \linebreak $\delta (3A+1) \mathrm{e}^{2A}$. Since $\delta$ can be chosen arbitrarily small, this proves the convergence in probability to 0 of $\sup_{0\leq m\leq \lfloor An \rfloor}\lvert Y_n^{(1)}(m) \rfloor$, which in turn gives the convergence in probability of  $~N_{p,n}^{(1)}(\lfloor n \cdot\rfloor)/n~$ to the function $t_{p,1}$ for the topology of uniform convergence on $[0,A]$, using that $\sup_{t\geq0} \lvert t_{p,1}(\lfloor nt\rfloor/n)-t_{p,1}(t) \rvert=O(1/n)$.

The proof of the induction step to pass from $k$ to $k+1$ is similar, using that $\big(G_{p,n}/n,N_{p,n}^{(1)}/n,...,N_{p,n}^{(k)}/n\big)$ is approximated by $\left(g_p\left(\cdot/n\right),t_{p,1}\left(\cdot /n\right),...,t_{p,k-1}\left(\cdot/n\right)\right)$, together with \eqref{eq:esp_N_k} and the definition and properties of the fonctions $t_{p,i}, i \geq 1$. Details are left to the reader.

\subsection{Largest and typical trees}
\label{sec:typ_L_trees}

\textbf{Proof of Corollary \ref{prop:cc}.} Let $t\in [0,\infty) \backslash \{1/2\}$. By  Theorem \ref{thm:forest}, $V_{p,n}(\lfloor nt \rfloor)/n$ converges in probability towards $1-g_p(t)>0$ and the ratio $R_{p,n}(t)=E_{p,n}(t)/V_{p,n}(t)$ converges in probability  towards $r_p(t)=t(1-g_p(t))$ which lies in $(0,1/2)$ since $t \neq 1/2$, by Proposition \ref{prop:propfunctions}. Together with the free forest property of Proposition \ref{prop:freeforest} and Proposition \ref{prop:largest_trees}, this implies that 
$$
\frac{\# T^{(i)}_{p,n}(\lfloor nt \rfloor)}{\ln(n)} \underset{n \rightarrow \infty}{\overset{\mathbb P}\longrightarrow} \frac{1}{2t(1-g_p(t))-1-\ln(2t(1-g_p(t)))}, \quad \forall i \in \mathbb N.
$$
Hence the result, since $g_p(t)=0$ when $t<1/2$.
$\hfill \square$

\bigskip

\textbf{Proof of Corollary \ref{cor:CC1}.} Fix $t \geq 0$.
From the exchangeability of the vertices of the model and Theorem \ref{thm:forest} (using that $~kN_{p,n}^{(k)}(\lfloor nt \rfloor)/n \leq 1~$ and $~V_{p,n}(\lfloor nt \rfloor))/n \leq 1~$ to apply the Dominated Convergence Theorem), one has for any $k\in \mathbb N$,
$$
\mathbb P\left(\mathrm{CC}^*_{p,n}(\lfloor nt \rfloor) \text{ is a tree of size } k\right)~=~\mathbb E\left[\frac{kN_{p,n}^{(k)}(\lfloor nt \rfloor)}{n}\right] ~\underset{n \rightarrow \infty} \longrightarrow ~ kt_{p,k}(t).
$$
and 
$$
\mathbb P\left(\mathrm{CC}^*_{p,n}(\lfloor nt \rfloor) \text{ is a tree}\right)~=~\mathbb E\left[\frac{V_{p,n}(\lfloor nt \rfloor)}{n}\right]~\underset{n \rightarrow \infty} \longrightarrow ~ 1-g_p(t).
$$
Then let $\mathrm t$ be one of the $k^{k-2}$ trees with $k$ vertices labeled $1,\ldots,k$. Recalling that the labels in $\mathrm{CC}^*_{p,n}(\lfloor nt \rfloor)$ are an increasing relabeling from 1 to $\# \mathrm{CC}^*_{p,n}(\lfloor nt \rfloor)$ of the initial labels of the connected component containing the vertex 1 at time $\lfloor nt \rfloor$, one sees from Proposition \ref{prop_forest_rw} that
$$
\mathbb P\left(\mathrm{CC}^*_{p,n}(\lfloor nt \rfloor)=\mathrm t ~ | ~\text{$\mathrm{CC}^*_{p,n}(\lfloor nt \rfloor)$ is a tree of size $k$} \right)~=~ \frac{1}{k^{k-2}}.
$$ 
Consequently,
\begin{eqnarray*}
\mathbb P\left(\mathrm{CC}^*_{p,n}(\lfloor nt \rfloor)=\mathrm t  ~ | ~\text{$\mathrm{CC}_{n,1}(\lfloor nt \rfloor)$ is a tree} \right)& \underset{n \rightarrow \infty} \longrightarrow& \frac{kt_{p,k}(t)}{(1-g_p(t)) k^{k-2}} \\
&& ~=~ \frac{\left(2t\right)^{k-1}\left(1-g_p(t)\right)^{k-1} \mathrm{e}^{-2kt\left(1-g_p(t)\right)}}{(k-1)!}.
\end{eqnarray*}
We conclude by recalling that a Galton-Watson tree with offspring distribution Poisson with mean $\lambda>0$ -- we denote $\mathrm{GW}_{\mathrm {Poi}(\lambda)}$ such a tree -- equipped with uniformly random labels from $1$ to $\#\mathrm{GW}_{\mathrm {Poi}(\lambda)}$ on its vertices and where the original order is forgotten, as well as the root, verifies for any tree $\mathrm t$ with $k$ vertices labeled $1,\ldots,k$,
$$
\mathbb P\big(\mathrm{GW}_{\mathrm {Poi}(\lambda)}= \mathrm t \big)~=~\frac{e^{-k\lambda } \lambda^{k-1}}{\prod_{i=1}^k n_i!} \cdot \frac{1}{k!} \cdot \prod_{i=1}^k n_i! \cdot k~= \frac{e^{-k\lambda } \lambda^{k-1}}{(k-1)!},
$$
where $n_i$ denotes the number of children of the vertex  $i$ (here $e^{-k\lambda } \lambda^{k-1}/\prod_{i=1}^k n_i!$ is the probability that the Galton-Watson tree, yet ordered and unlabeled, is equal to an ordered unlabeled version of $\mathrm t$, $1/k!$ corresponds to adding the labels,  $\prod_{i=1}^k n_i!$ to removing the order, and last $k$ to forgetting the root). Taking $\lambda=2t(1-g_p(t))$, this corresponds to the above limit in distribution of $\mathrm{CC}^*_{p,n}(\lfloor nt \rfloor)$ conditioned to be a tree. $\hfill \square$

\bigskip

\textbf{Remark.} We could have proved this result by using the continuous model of Section \ref{sec:Poisson} together with Lemma \ref{lem:cvPnk}, and then de-Poissonizing with the help of (\ref{bound:Poisson}). With this approach, one has to use the identity 
$$
e^{2(1-p)\int_0^t g_p(u) \mathrm du}~=~\left(1-g_p(t)\right)e^{2tg_p(t)}, \quad \forall t \geq 0
$$
to conclude. This identity can be proved by verifying that the derivatives of the functions in the left and right hand sides are equal, using the equation (\ref{eq:f}) satisfied by $g_p$.

\section{A continuous version of the model}
\label{sec:continuous}

We introduce in this short section a continuous version of the $p$-frozen Erd\H{o}s-R\'enyi model, $p \in [0,1]$, and then underline some of its properties in the fluid limit.

\subsection{Poissonization}
\label{sec:Poisson}

Starting with $n$ isolated vertices labelled $1,\ldots,n$, we consider, on each of the $n(n-1)/2$ potential edges, independently, a Poisson point process (PPP in the following) with intensity $1/n$. When such a PPP rings, if the edge has not already been added to the graph:
\begin{enumerate}[topsep=0pt]
\item[-] either the edge connects two tree-components of the current graph and we add it to the graph
\item[-] or it connects two vertices of unicycle-components and we discard it
\item[-] or it connects a tree-component and a unicycle-component and we add it with probability $p$ and discard it otherwise.
\end{enumerate}
We let $\mathcal {F}_{p,n}(t)$ denote the graph at time $t \geq 0$ and emphasize that if $\mathcal N$ is a standard Poisson process (with intensity 1), independent of the discrete $p$-frozen model $\mathrm{F}_{p,n}$, then,
\begin{equation}
\label{lien:dc}
\big(\mathcal {F}_{p,n}(t), t\geq 0 \big)~\overset{\mathrm{(d)}}{=}~\big(\mathrm{F}_{p,n}\left(\mathcal N((n-1)t/2)\right), t\geq 0 \big).
\end{equation}

\subsection{Fluid limit in the continuous model}
\label{sec:flcontinuous}

We let $(\mathcal G_{p,n}(t),t\geq 0)$ denote the gel size process in this continuous version of the model (that is the total number of vertices that belong to unicycle components) and state the following corollary of Theorem \ref{thm:fluid limit}. Of course, the other results on the fluid limit of the discrete model also transfer easily to the continuous model, but we  only highlight here what we will really need in the rest of the paper. Here $p \in(0,1]$.

\begin{corollary}
\label{cor:continuous_cv}
As $n \rightarrow \infty$,
$$\left(\frac{\mathcal G_{p,n}(t)}{n},~ t \geq 0 \right)~\overset{\mathbb P}\longrightarrow~\left(g_p(t/2),~ t\geq 0\right)$$
for the topology of uniform convergence on compacts. Consequently, for each $A>0$,
$$
 \int_0^{A} \left(1-\frac{\mathcal G_{p,n}(t)}{n}\right) \mathrm dt~\overset{\mathbb P}\longrightarrow~\int_0^{A} \left(1-g_p(t/2)  \right) \mathrm dt.
$$
\end{corollary}

\begin{proof} We write for $t\geq0$, $\mathcal G_{p,n}(t)=G_{p,n}(\mathcal N((n-1)t/2))$, where $G_{p,n}$ is a discrete version of the $p$-frozen model and $\mathcal N$ is the Poisson process involved in (\ref{lien:dc}), independent of $G_{p,n}$, $\forall n$.  
Using that the derivative of $g_p$ is bounded on $\mathbb R_+$ and the convergence in probability of the rescaled process $\mathcal N(n \cdot)/n$ towards the identity function (for the topology of uniform convergence on compacts), we have that for all $A>0$, as $n \rightarrow \infty$,
$$
\sup_{t \in [0,A]} \left| g_p\left(\frac{\mathcal N((n-1)t/2)}{n}\right)-g_p(t/2)\right|~\overset{\mathbb P} \longrightarrow~0.
$$
Together with Theorem \ref{thm:fluid limit} and a standard use of the triangular inequality this leads to the announced convergence.
\end{proof}

\medskip

The convergence of the integrals can be completed as follows:

\begin{corollary}
\label{cor:cvintegrale}
As $n \rightarrow \infty$,
$$\int_0^{\infty} \left(1-\frac{\mathcal G_{p,n}(t)}{n} \right) \mathrm dt~\overset{\mathbb P}\longrightarrow~\int_0^{\infty} \left(1-g_p(t/2)\right) \mathrm dt=\frac{\psi(1/p)+\gamma_E}{1-p}.$$
Additionally, for any sequence of positive real numbers $(t_n)$ such that $t_n \rightarrow \infty$,
$$\int_0^{t_n} \left(1-\frac{\mathcal G_{p,n}(t)}{n} \right) \mathrm dt~\overset{\mathbb P}\longrightarrow~\frac{\psi(1/p)+\gamma_E}{1-p}.$$
\end{corollary}

\medskip

The proof of this corollary is partly based on the following lemma.

\begin{lemma}
\label{lem:ub_k}
 Let $\sigma_n(1/2):=\inf\{t \geq 0: \mathcal G_{p,n}(t) \geq n/2\}$ and set for $t\geq 0$
$$k_n(t)=\mathbb E\left[1-\frac{\mathcal G_{p,n}(t+\sigma_n(1/2))}{n} \right].$$
Then, $k_n(t) \leq  e^{-pt/2}$ for all $t\geq 0$ and all $n\in \mathbb N$.
\end{lemma}

\begin{proof} 
The dynamic of the continuous model implies that conditional on $\mathcal G_{p,n}(t+\sigma_n(1/2))=\ell$ (for $\ell$ integer, $n/2 \leq \ell \leq n$), 
we have that $
\mathcal G_{p,n}(t+\varepsilon+\sigma_n(1/2)) \geq \mathcal G_{p,n}(t+\sigma_n(1/2))+1$  with probability greater than  
$$p  \left(1-\exp(-\varepsilon  \ell (n-\ell)/n)\right)~\geq~p \left(1-\exp(-\varepsilon (n-\ell)/2)\right), \quad \forall \varepsilon>0.
$$
Since the function $~x \mapsto (1-\exp(-x))/x ~$ is decreasing on $(0,\infty)$ and since $n-\ell \leq n/2$ when $\ell \geq n/2$, the above inequality implies that
$$
k_n(t+\varepsilon)-k_n(t)=\frac{\mathbb E\left[\mathcal G_{p,n}(t+\sigma_n(1/2))-\mathcal G_{p,n}(t+\varepsilon+\sigma_n(1/2)) \right]}{n}~\leq~ -2p  k_n(t)\frac{1-\exp(-\varepsilon n/4)}{n}.
$$
Consequently, for all $t\geq 0$ and all $\varepsilon>0$,
\begin{eqnarray*}
k_n(t)~\leq ~k_n\left(\varepsilon \left \lfloor \frac{t}{\varepsilon} \right\rfloor \right) &\leq& k_n(0) \left(1-2p \frac{1-\exp(-\varepsilon n/4)}{n} \right)^{\lfloor \frac{t}{\varepsilon}\rfloor}~ \underset{\varepsilon \rightarrow 0}\longrightarrow ~ k_n(0) \exp(-pt/2),
\end{eqnarray*}
which gives the upper bound $k_n(t) \leq k_n(0) \exp(-pt/2) \leq \exp(-pt/2)$ for all $t\geq 0$ and all $n\in \mathbb N$.
\end{proof}

\bigskip

\textbf{Proof of Corollary \ref{cor:cvintegrale}.} Lemma \ref{lem:identintegral} gives the identity $\int_0^{\infty} \left(1-g_p(t/2)\right) \mathrm dt=\big(\psi(1/p)+\gamma_E\big)/(1-p).$ Then
fix $\varepsilon >0, \delta>0$. We want to show that 
\begin{equation}
\label{eq:cvP}
\mathbb P\left(\left | \int_0^{\infty} \left(1-\frac{\mathcal G_{p,n}(t)}{n}\right) \mathrm dt- \int_0^{\infty} \left(1-g_p(t/2)  \right) \mathrm dt \right| \geq \varepsilon \right) \leq \delta~\text{ for all } n \text{ large enough.}
\end{equation}
In that aim fix $A>0$ large enough so that $g_p(A/2)>1/2$ and
$$
\int_{2A}^{\infty} \left(1-g_p(t/2)\right) \mathrm dt \leq \frac{\varepsilon}{3} \quad \text{and} \quad \frac{3}{\varepsilon} \int_A^{\infty} e^{-pt/2} \mathrm dt \leq \frac{\delta}{4}.
$$
Then note that
\begin{eqnarray*}
&& \mathbb P\left(\left | \int_0^{\infty} \left(1-\frac{\mathcal G_{p,n}(t)}{n}\right) \mathrm dt- \int_0^{\infty} \left(1-g_p(t/2)  \right) \mathrm dt \right| \geq \varepsilon \right) \\
&\leq&  \mathbb P\left( \int_{2A}^{\infty} \left(1-\frac{\mathcal G_{p,n}(t)}{n}\right) \mathrm dt \geq \frac{\varepsilon}{3} \right) 
+ \mathbb P\left(\left | \int_0^{2A} \left(1-\frac{\mathcal G_{p,n}(t)}{n}\right) \mathrm dt- \int_0^{2A} \left(1-g_p(t/2)  \right) \mathrm dt \right| \geq \frac{\varepsilon}{3}\right).
\end{eqnarray*}
By Corollary \ref{cor:continuous_cv}, the last probability converges to 0 and is therefore smaller than $\delta/2$ for $n$ large enough. Besides, with the notation of Lemma  \ref{lem:ub_k}, 
\begin{eqnarray*}
 \mathbb P\left( \int_{2A}^{\infty} \left(1-\frac{\mathcal G_{p,n}(t)}{n}\right) \mathrm dt \geq \frac{\varepsilon}{3} \right) &\leq& \mathbb P\left(A < \sigma_n(1/2)\right) \\
 &+& \mathbb P\left( \int_{2A-T_n(1/2)}^{\infty}\left(1-\frac{\mathcal G_{p,n}(t+\sigma_n(1/2))}{n}\right) \mathrm dt \geq \frac{\varepsilon}{3}, A \geq  \sigma_n(1/2) \right) \\
 &\leq &  \mathbb P\left(A <\sigma_n(1/2)\right) + \frac{3}{\varepsilon} \int_A^{\infty} k_n(t) \mathrm dt,
\end{eqnarray*}
where for the second inequality we used that $2A-\sigma_n(1/2) \geq A$ when $\sigma_n(1/2)\leq A$ and then Markov's inequality. By Lemma \ref{lem:ub_k} and the choice of $A$, $\frac{3}{\varepsilon} \int_A^{\infty} k_n(t) \mathrm dt \leq \delta/4$. Whereas $ \mathbb P\left(A <\sigma_n(1/2)\right)= \mathbb P\left(\mathcal G_{p,n}(A) <n/2\right)$ which converges to 0 by Corollary \ref{cor:continuous_cv} and since $g_p(A/2)>1/2$, so this probability is also smaller than $\delta/4$ for large $n$.  All this leads to the expected claim (\ref{eq:cvP}).

The proof of the convergence in probability of $~\int_0^{t_n} \left(1-\mathcal G_{p,n}(t)/n \right) \mathrm dt~$  to $~\int_0^{\infty} \left(1-g_p(t/2)\right) \mathrm dt$ when $t_n \rightarrow \infty$ is similar.
$\hfill \square$

\section{Total gelation time and vicinity}
\label{sec:gelation}

In order to study the first time at which all the vertices of $\mathrm F_{p,n}$ are frozen, and the number of trees of size $k \in \mathbb N$ in its neighborhood,  we work in this section with the continuous model, which offers more independence and eases the proofs. At the end of the section, in Subsection \ref{sec:dePoisson}, we transfer the results to the discrete model and prove Theorem \ref{thm:absorption} and Proposition \ref{prop:arbreskdiscrets}. In the continuous setting, we recall that $\big(\mathcal F_{p,n}(t),t\geq 0\big)$ denotes the $p$-frozen model on $n$ vertices and $\big(\mathcal G_{p,n}(t),t\geq 0\big)$ the corresponding process of mass of gel. As mentioned in the Introduction, the presence of trees of size $k \in \mathbb N$ in the model is relative to the threshold time
$$
\mathsf t^{(k)}_{p,n}=\frac{\ln(n)}{kp}+\frac{k-1}{kp}\ln \left(\frac{\ln(n)}{kp} \right),
$$
which is decreasing in $k$ provided that $n$ is large enough. 
Our main goal is to compare the last time at which there is a tree of size $k$ in the continuous model, denoted by $\mathcal A_{p,n}^{(k)}$, and the last time at which there is a tree of size greater or equal to $k$, denoted by $\mathcal A_{p,n}^{(k+)}$, to the threshold time $\mathsf t^{(k)}_{p,n}$ (Theorem \ref{prop:tau_k_continu} and Theorem \ref{prop:A_k_continu}). In that aim we start by evaluating the number of trees of size $k$ at any time $t$. In particular, we prove that at time  $\mathsf t^{(k)}_{p,n}+c$, $c \in \mathbb R$, the number of trees of size $k$ converges in distribution towards a Poisson distribution whose parameter depends on $k,p,c$, the digamma function $\psi$ and the Euler constant $\gamma_E$ (Proposition \ref{prop:Poisson_continu}) .

\subsection{On the number of trees of size $k$}

Let 
$$\mathcal N_{p,n}^{(k)}(t)=\frac{1}{k}\sum_{i=1}^n \mathbbm 1_{\{i  \text{ belongs to a tree of size }k \text{ at time }t\}}$$ be the number of trees of size $k$ at time $t \geq 0$. In order to get some information on this quantity, 
we introduce the following notation. If $\mathrm t_1,\ldots,\mathrm t_\ell$ are trees of size $k$ with vertices in $\{1,\ldots,n\}$ and no common vertices (so necessarily $\ell k \leq n$), we set for all times $t\geq 0$
$$
P_n^{(k)}(\ell,t):=\mathbb P \big(\mathrm t_1,\ldots,\mathrm t_{\ell} \text{ are connected components of } \mathcal{F}_{p,n}(t)\big).
$$
A key observation is that this probability can be expressed as follows.

\begin{lemma}
\label{lem:cvPnk}
For all $t\geq 0$,
\begin{eqnarray*}
P_n^{(k)}(\ell,t) &=& \left(1-e^{-t/n}\right)^{\ell (k-1)} \left(e^{-t/n} \right)^{\frac{\ell k(\ell k-1)}{2}-\ell(k-1)}e^{-\frac{\ell k(n-\ell k)}{n}pt} \\
&& \times \; \mathbb E\left[e^{-\frac{\ell k(n-\ell k)}{n} (1-p)\int_0^t \left(1-\frac{\mathcal G_{p,n-\ell k}(u)}{n- \ell k} \right) \mathrm du} \right]
\end{eqnarray*}
where $\mathcal G_{p,n-\ell k}$ is the gel mass process of a $\mathcal F_{p,n-\ell k}$ model.
\end{lemma}

This lemma, as well as the following results below and some technical corollaries to prepare Section \ref{sec:extinctiontimes}, will be proved in Section \ref{sub:probaPnk} and Section \ref{sub:moments}.

Using the fluid limit approximation in the continuous model (Section \ref{sec:flcontinuous}), this result will in particular give us useful estimates to evaluate the asymptotic behavior of the moments of $\mathcal N_{p,n}^{(k)}(t_n)$ at some times $t_n$ that may depend on $n$. Indeed, one easily sees, recalling that there are $k^{k-2}$ different trees involving a fixed set of  $k$ labeled vertices (this is Cayley's formula) and using the exchangeability of the vertices $1,\ldots,n$, that
\begin{equation}
\label{esp:N}
\mathbb E\left[\mathcal N_{p,n}^{(k)}(t)\right]=\frac{n}{k} \binom{n-1}{k-1} k^{k-2} P^{(k)}_n(1,t).
\end{equation}
More generally we can express the factorial moments of $k \mathcal N_{p,n}^{(k)}(t)$ in terms of the probabilities $P^{(k)}_n\left(\ell,t\right)$. This is the aim of the following lemma, where $\mathcal P_j$ denotes the set of partitions of an integer $j$, i.e. the set of finite non-increasing sequences $(n_1,\ldots,n_\ell) \in \mathbb N$, with $\ell$ the length of the sequence, such that $\sum_{i=1}^\ell n_i=j$. For such a sequence and all $1\leq i \leq j$, we let $m_i$ denote the number of occurrences of the integer $i$ in the sequence.

\begin{lemma} 
\label{lem:factorial_moments}
For all times $t\geq 0$ and all positive integers $j \leq n$ 
\begin{eqnarray*}
&& \mathbb E\left[\prod_{i=0}^{j-1}(k\mathcal N_{p,n}^{(k)}(t)-i) \right] \\
&=& 
\sum_{\substack{(n_1,\ldots,n_\ell) \in \mathcal P_j  \\ \text{ such that } n_i \leq k \: \forall i, \text{and }\ell k\leq n}} \frac{n!}{(n-k\ell)!} \binom{j}{n_1,\ldots,n_\ell} \cdot \prod_{i=1}^j \frac{1}{m_i!} \cdot \prod_{i=1}^l\frac{1}{(k-n_i)!} \cdot (k^{k-2})^\ell P^{(k)}_n(\ell,t),
\end{eqnarray*}
whereas $\prod_{i=0}^{j-1}(k\mathcal N_{p,n}^{(k)}(t)-i)=0$ when $j>n$.
\end{lemma}

As said, together with the previous expression of the probabilities $P^{(k)}_n\left(\ell,t\right)$, this will allow us to obtain asymptotics of these moments. Notably this will lead us to the following estimates on expectation of the number of trees of size $k' \in \mathbb N$ at the threshold time $\mathsf t^{(k)}_{p,n}$:

\begin{proposition}
\label{cor:boundsexp}
For all $k,k' \in \mathbb N$ and $c \in \mathbb R$, as $n\rightarrow \infty$
$$\mathbb E\left[\mathcal N_{p,n}^{(k')}\big(\mathsf t^{(k)}_{p,n}+c \big)\right] = O\left( \left(\frac{n}{\ln(n)} \right)^{1-k'/k} \right).$$
\end{proposition}

And more precisely when $k'=k$:

\begin{proposition}
\label{prop:Poisson_continu}
For all $k \in \mathbb N$ and all $c \in \mathbb R$
$$
\mathcal N_{p,n}^{(k)}(\mathsf t^{(k)}_{p,n}+c) \; \underset{n \rightarrow \infty}{\overset{(\mathrm d)}\longrightarrow} \; \mathcal P\left(\frac{k^{k-2}e^{-kpc} e^{-k\left(\psi(1/p)+\gamma_{\mathrm E}\right)}}{k!} \right).
$$
Additionally, we have the convergence of each positive moment of $~\mathcal N_{p,n}^{(k)}(\mathsf t^{(k)}_{p,n}+c)$ to the corresponding moment of the limit Poisson distribution.
\end{proposition}

\subsubsection{Estimates on the probabilities $P_n^{(k)}(\ell,t)$}
\label{sub:probaPnk}

We set up in this part results related to the probabilities $P_n^{(k)}(\ell,t)$, starting with the proof of Lemma \ref{lem:cvPnk} and then several corollaries that will be useful in the sequel.

\bigskip

\textbf{Proof of Lemma \ref{lem:cvPnk}.} By exchangeability, we  may assume that $\mathrm t_1,\ldots,\mathrm t_\ell$ are trees of size $k$ with vertices in $\{1,\ldots,\ell k\}$ and no common vertices. Note that the evolution process of the graph reduced to the vertices $\ell k+1,\ldots,n$ before there are interactions with vertices $1,\ldots, \ell k$ follows the $\mathcal F_{p,n-\ell k}$ model. The trees $\mathrm t_1,\ldots,\mathrm t_\ell$ are then connected components of  $\mathcal {F}_{p,n}(t)$ if and only if:
\vspace{-0.3cm}
\begin{enumerate}[leftmargin=0.8cm]
\item[$\bullet$] the $\ell (k-1)$ edges of the $\ell$ trees $\mathrm t_1,\ldots,\mathrm t_\ell$ have been added at time $t$, which happens with probability $\big(1-e^{-t/n}\big)^{\ell(k-1)}$
\item[$\bullet$] and, the $\ell k(\ell k-1)/2-\ell(k-1)$ other possible edges between the vertices  $1,\ldots,\ell k$ have not been added at time $t$, which happens with probability $\big(e^{-t/n}\big)^{\frac{\ell k(\ell k-1)}{2}-\ell(k-1)}$, independently
\item[$\bullet$] and, the $\ell k(n-\ell k)$ edges between one of the vertices $1,\ldots,\ell k$ and one of the vertices $\ell k+1,\ldots,n$ have not been added before time $t$: conditioning on the dynamic of the vertices $\ell k+1,\ldots,n$, this happens with probability
$$\mathbb E\left[e^{-\frac{\ell k}{n}\left(\int_0^t \left(n-\ell k-\mathcal G_{p,n-\ell k}(u)\right)\mathrm du+p\int_0^t \mathcal G_{p,n-\ell k}(u) \mathrm du\right)} \right] =\mathbb E\Bigg[e^{-\frac{\ell k(n-\ell k)}{n}\left(pt + (1-p) \int_0^t  \left(1-\frac{\mathcal G_{p,n-\ell k}(u)}{n-\ell k} \right)\mathrm du\right)} \Bigg],$$ independently.
\end{enumerate}
This gives the announced expression of $P^{(k)}_n(\ell,t)$.
$\hfill \square$

\bigskip

This easily leads us to:

\begin{corollary}
\label{cor:cvP}
For any sequence of times $(t_n)$ such that $t_n \rightarrow \infty$ and $t_n=o(n)$, for all $\ell \in \mathbb N$,
\begin{eqnarray*}
P^{(k)}_n(\ell,t_n)
 &=&\left( \frac{t_n}{n}\right)^{\ell(k-1)}e^{-\ell kpt_n-\ell k\left(\psi(1/p)+\gamma_{\mathrm E}\right) +o(1)} \\
 &=& \big(P^{(k)}_n(1,t_n)\big)^{\ell}\left(1+o(1)\right).
\end{eqnarray*}
In particular,
for all $c \in \mathbb R$
$$
n^k \cdot P^{(k)}_n\big(1,\mathsf t^{(k)}_{p,n}+c\big) \underset{n \rightarrow \infty} \longrightarrow e^{-kpc} \cdot e^{-k\left(\psi(1/p)+\gamma_{\mathrm E}\right)}.
$$
\end{corollary}

\begin{proof} $\bullet$ From the expression of Lemma \ref{lem:cvPnk}, we immediately see that when $t_n \rightarrow \infty$ and $t_n=o(n)$,
$$
P^{(k)}_n(\ell,t_n)=\left( \frac{t_n}{n}\right)^{\ell(k-1)}e^{-\ell kpt_n+O\left(\frac{t_n}{n}\right)} \cdot \mathbb E \left[e^{-\ell k (1-p) \int_0^{t_n} \left(1-\frac{\mathcal G_{p,n-\ell k}(u)}{n- \ell k} \right) \mathrm du+O\left(\frac{t_n}{n}\right) }\right],
$$
where the $O\left(\frac{t_n}{n}\right)$ in the expectation is deterministic. 
So by Corollary \ref{cor:cvintegrale} and then Lemma \ref{lem:identintegral},
$$
\mathbb E \left[e^{-\ell k (1-p) \int_0^{t_n} \left(1-\frac{\mathcal G_{p,n-\ell k}(u)}{n- \ell k} \right) \mathrm du+O\left(\frac{t_n}{n}\right) }\right] \underset{n \rightarrow \infty} \longrightarrow e^{-\ell k\left(\psi(1/p)+\gamma_{\mathrm E}\right)}.
$$

$\bullet$ Applying this to $t_n=\mathsf t^{(k)}_{p,n}+c$ immediately gives $n^k \cdot P^{(k)}_n\big(1,\mathsf t^{(k)}_{p,n}+c\big)  \rightarrow e^{-kpc} \cdot e^{-k\left(\psi(1/p)+\gamma_{\mathrm E}\right)}$.
\end{proof}

\bigskip

We will also need the following control in order to apply, later in Section \ref{sub:moments}, the Dominated Convergence Theorem.

\begin{corollary}
\label{cor:TCD1}
Fix $c \in \mathbb R$.
For every $\varepsilon>0$, there exists $\ell_{\varepsilon} \in \mathbb N$ and $n_{\varepsilon} \in \mathbb N$ such that
$$
\frac{n!}{\ell ! (n-k\ell)!} \cdot P_n^{(k)}(\ell,\mathsf t^{(k)}_{p,n}+c)  \leq \varepsilon^{\ell}, \quad \text{ for all }  n \geq n_{\varepsilon} \text{ and all } \ell_{\varepsilon} \leq \ell \leq n/k.
$$
\end{corollary}

\begin{proof}
The "constants" $c_1,c_2,c_3$ appearing in this proof may depend on $k$ and $c$, but not on $n$ or $\ell \leq n/k$.
On the one hand, by Stirling's formula, there exists some constant $c_1$ such that  
\begin{equation}
\label{majo:facto}
\frac{n!}{\ell ! (n-k\ell)!} \; \leq \; \frac{n^{k\ell}}{\ell !} \; \leq \; c_1n^{k\ell} e^{\ell-\ell \ln(\ell)}, \quad \forall n,\ell \geq 1, \ell \leq n/k.
\end{equation}
On the other hand, since the expectation involved in the expression of $P_n^{(k)}(\ell,\mathsf t^{(k)}_{p,n}+c)$ is bounded from above by 1, and since $1-e^{-x}\leq x$ for all $x\geq 0$, 
$$
P_n^{(k)}(\ell,\mathsf t^{(k)}_{p,n}+c) \leq \left( \frac{\mathsf t^{(k)}_{p,n}+c}{n}\right)^{\ell (k-1)} \cdot e^{-\frac{\mathsf t^{(k)}_{p,n}+c}{n} \left( \frac{\ell k(\ell k-1)}{2}-\ell(k-1)+p\ell k(n-\ell k)\right)}.
$$
It is easy to see, using the definition of $\mathsf t^{(k)}_{p,n}$, that for $n$ large enough, simultaneously for all $\ell \geq 1$,
$$
e^{\ell(k-1) \ln(\mathsf t^{(k)}_{p,n}+c)-p\ell k \mathsf t^{(k)}_{p,n}} \leq \frac{e^{\ell}}{n^\ell}.
$$
And also, for $\ell \leq n/k$, that 
$$
e^{\frac{-c}{n} \left( \frac{\ell k(\ell k-1)}{2}-\ell(k-1)+p\ell k(n-\ell k)\right)} \leq e^{c_2 \ell}
$$
for some constant $c_2 \in (0,\infty)$. 
This implies that
\begin{equation}
\label{majo:P_n}
P_n^{(k)}(\ell,\mathsf t^{(k)}_{p,n}+c) \leq \frac{e^{(1+c_2) \ell}}{n^{k\ell}} \cdot e^{-\frac{\mathsf t^{(k)}_{p,n}}{n} \left( \frac{\ell k(\ell k-1)}{2}-\ell(k-1)-p(\ell k)^2\right)}.
\end{equation}
$\bullet$ If $p\leq 1/2$, $ \frac{\ell k(\ell k-1)}{2}-\ell(k-1)-p(\ell k)^2 \geq -\frac{3\ell k}{2}$ for all $\ell \geq 1$, so we have that, since moreover $\mathsf t^{(k)}_{p,n} \geq 0$ and $\frac{\mathsf t^{(k)}_{p,n}}{n} \rightarrow 0$ as $n \rightarrow \infty$,
$$
P_n^{(k)}(\ell,\mathsf t^{(k)}_{p,n}+c) \leq \frac{e^{c_3 \ell}}{n^{k\ell}} 
$$
for all $n$ large enough and all $\ell \geq 1$. Together with (\ref{majo:facto}) this clearly leads to the statement of the corollary.

$\bullet$ If $p\in (1/2,1]$, consider $\eta>0$ such that $a:=(1+\eta)^2(1-\frac{1}{2p}) \in (0,1)$. We then use that $\frac{\mathsf t^{(k)}_{p,n}}{n} \leq (1+\eta) \frac{\ln(n)}{kpn}$ for $n$ large enough, and that $-\frac{\ell k(\ell k-1)}{2}+\ell(k-1)+p(\ell k)^2\leq (1+\eta) (p-\frac{1}{2}) (\ell k)^2$ for $\ell$ large enough, to get for those $n,\ell$, using (\ref{majo:P_n}),
$$
P_n^{(k)}(\ell,\mathsf t^{(k)}_{p,n}+c) \leq \frac{e^{(1+c_2) \ell}}{n^{k\ell}} \cdot e^{ (1+\eta)^2 \frac{\ln(n)}{n} \cdot (1-\frac{1}{2p}) \ell^2 k}=\frac{e^{\ell \left((1+c_2)+a\frac{\ln(n)}{n} \ell k\right)}}{n^{k\ell}}.
$$
Together with (\ref{majo:facto}), we obtain for those $n,\ell$, assuming moreover that $c_1 \leq e^{\ell}$,
\begin{equation}
\label{ineq:intermediaire1}
\frac{n!}{\ell ! (n-k\ell !)} \cdot P_n^{(k)}(\ell,\mathsf t^{(k)}_{p,n}+c) \leq e^{\ell \left((3+c_2)+a\frac{\ln(n)}{n}\ell k-\ln(\ell) \right)}=e^{\ell h(\ell)}
\end{equation}
where $h(x):=(3+c_2)+a\frac{\ln(n)}{n}x k-\ln(x)$. One easily sees that this function is convex on $(0,\infty)$, with
$$
h(x) \leq 3+c_2+\max \left(a \frac{\ln(n)}{n}\ x_0k-\ln(x_0); a \ln(n)-\ln(n/k) \right) \text{ when } x \in [x_0 ; n/k]
$$
(whatever $x_0>0$ is).
For every $\varepsilon>0$, there exists $\tilde \ell_{\varepsilon} \in \mathbb N$ such that $3+c_2+1-\ln(\tilde \ell_{\varepsilon}) \leq \ln(\varepsilon)$. Then, take $x_0 =\tilde \ell_{\varepsilon}$.  Next,
there exists $\tilde n_{\varepsilon} \in \mathbb N$ such that for all $n \geq \tilde n_{\varepsilon}$, we have both $3+c_2+a \ln(n)-\ln(n/k) \leq \ln(\varepsilon)$ (since $a<1$) and $a \frac{\ln(n)}{n} \tilde \ell_{\varepsilon} k\leq 1$. All this implies that for $n \geq \tilde n_{\varepsilon}$ and then $\ell \in [\tilde \ell_{\varepsilon},n/k]$,
$$
h(\ell) \leq \ln(\varepsilon).
$$
Together with (\ref{ineq:intermediaire1}), this gives the expected upper bound.
\end{proof}

\bigskip

Last, we set up the following bound, in order to prove later Corollary \ref{cor:treesafterk}.

\bigskip

\begin{corollary}
\label{cor:unifbound}
Let $k\in \mathbb N$, $c \in \mathbb R$. 
Then for all $n$ large enough,
$$
\sup_{k+1 \leq i\leq n} e^{i} \cdot n^i \cdot P_n^{(i)}(1,\mathsf t^{(k)}_{p,n}+c) \leq \frac{1}{n^{\frac{1}{2k}}}.
$$
\end{corollary}

\begin{proof}
From Lemma \ref{lem:cvPnk},
\begin{eqnarray*}
 e^{i} \cdot n^i \cdot P_n^{(i)}(1,\mathsf t^{(k)}_{p,n}+c) &\leq& e^{i+i\ln(n)+(i-1)\ln\Big(\frac{\mathsf t^{(k)}_{p,n}+c}{n}\Big)-\left(\frac{i(i-1)}{2}-(i-1)\right)\frac{\mathsf t^{(k)}_{p,n}+c}{n}-\frac{i(n-i)}{n}p(\mathsf t^{(k)}_{p,n}+c)} \\
 &=& e^{\ln(n)\cdot h_n(i)}
\end{eqnarray*}
where $h_n$ is the polynomial of degree 2 defined for $x \in \mathbb R$ by
$$
h_n(x)=\frac{x}{\ln(n)}+x+(x-1)\frac{\ln\Big(\frac{\mathsf t^{(k)}_{p,n}+c}{n}\Big)}{\ln(n)}-\left(\frac{x(x-1)}{2}-(x-1)\right)\frac{\mathsf t^{(k)}_{p,n}+c}{n \ln(n)}-\frac{x(n-x)}{n \ln(n)}p(\mathsf t^{(k)}_{p,n}+c).
$$
We let the reader check that as $n \rightarrow \infty$, 
$$h_n(k+1)\rightarrow -1/k, \quad  h'_n(k+1)\rightarrow -1/k, \quad h'_n(n) \rightarrow -(1-p)/kp.$$ Hence when $p \in (0,1)$, for $n$ large enough, $h_n'$ is strictly negative, and therefore $h_n$ strictly decreasing, on $[k+1,n]$, uniformly smaller than $-1/2k$ (for $n$ large enough). When $p=1$, $h''_n(x)=\frac{\mathsf t^{(k)}_{p,n}+c}{n\ln(n)}$ for all $x$, hence for $n$ large enough $g_n$ is convex, and $h_n(n) \sim-n/2k$ which, together with $h_n(k+1)\rightarrow -1/k$, implies that $h_n$ is also uniformly smaller than $-1/2k$ on $[k+1,n]$ for $n$ large enough. 

In conclusion, whatever $p \in (0,1]$, we have that for $n$ large enough and then all $i \in \llbracket k+1 ,n\rrbracket$,
$$
e^{i} \cdot n^i \cdot P_n^{(i)}(1,\mathsf t^{(k)}_{p,n}+c) \leq e^{-\ln(n) \frac{1}{2k}}.
$$
\end{proof}

\subsubsection{Moments and asymptotics of $\mathcal N_{p,n}^{(k)}$}
\label{sub:moments}

We start this section with the proof of the identities of Lemma \ref{lem:factorial_moments}. We will then see how to use them to prove, together with the estimates of Corollary \ref{cor:cvP} and Corollary \ref{cor:TCD1}, the bounds on $\mathcal N_{p,n}^{(k')}(\mathsf t^{(k)}_{p,n}+c)$, $k' \in \mathbb N$ of Proposition \ref{cor:boundsexp} and the asymptotic distribution of $\mathcal N_{p,n}^{(k)}(\mathsf t^{(k)}_{p,n}+c)$ stated in Proposition \ref{prop:Poisson_continu}. Last, in complement and to prepare the next section on the behavior of the times $\mathcal A_{p,n}^{(k)}$, $\mathcal A_{p,n}^{(k+)}$, we set up a corollary saying that there is asymptotically no tree of size strictly larger than $k$ at time $\mathsf t^{(k)}_{p,n}+c$, for any $c$ (Corollary \ref{cor:treesafterk}). 

\bigskip

\textbf{Proof of Lemma \ref{lem:factorial_moments}.}
Since $k\mathcal N_{p,n}^{(k)}(t) \in \llbracket 0;n \rrbracket$, its $j$-th factorial moment is null when $j>n$. In the sequel we fix $j\leq n$.
For  $1 \leq i \leq n$ and $t\geq 0$, consider the random variable $$Y_{n,i:}=\mathbbm 1_{\{\text{the vertice } i \text{ belongs to a tree of size $k$ at time } t\}}$$
so that $k\mathcal N_{p,n}^{(k)}(t)=\sum_{i=1}^n Y_{n,i}$. The five lines that follow are classical in the study of random graphs:
using that $Y_{n,i}^2=Y_{n,i}$, one sees by induction (on $j$) that
$$
\prod_{i=0}^{j-1}(k\mathcal N_{p,n}^{(k)}(t)-i)=\sum_{1 \leq i_1 \neq i_2 \neq \ldots \neq i_j \leq n} \prod_{m=1}^j Y_{n,i_m}.
$$
Since the random variables $Y_{n,i}, 1 \leq i \leq n$ are exchangeable, this gives the following expression for the $j$-th factorial moment of $k\mathcal N_{p,n}^{(k)}(t)$:
\begin{eqnarray*}
\mathbb E\left[\prod_{i=0}^{j-1}(k\mathcal N_{p,n}^{(k)}(t)-i) \right]&=& \frac{n!}{(n-j)!}\cdot \mathbb E\left[\prod_{i=1}^j Y_{n,i}\right].
\end{eqnarray*}
Next, $\mathbb E\left[\prod_{i=1}^j Y_{n,i}\right]$ is the probability that the vertices $1,\ldots,j$ belong to a tree of size $k$ at time $t$. By decomposing according to the number of vertices among $1,\ldots, j$ which belong to a same tree of size $k$ -- which gives a partition of $j$ -- and using that the number of trees on  $k$ labeled vertices is $k^{k-2}$, we obtain:
\begin{equation*}
 \mathbb E\left[\prod_{i=1}^j Y_{n,i}\right] = \sum_{\substack{(n_1,\ldots,n_{\ell}) \in \mathcal P_j, \\ n_i \leq k \: \forall i, \; \ell k\leq n}} \binom{j}{n_1,\ldots,n_{\ell}} \cdot \frac{1}{\prod_{i=1}^j m_i!} \cdot \binom{n-j}{k-n_1,\ldots,k-n_\ell, n-k\ell} \cdot (k^{k-2})^{\ell} P^{(k)}_n(\ell,t).\\
\end{equation*}
Together with the above expression of the $j$-th factorial moment of $k\mathcal N_{p,n}^{(k)}(t)$ this gives the result.
$\hfill \square$

\bigskip

\textbf{Proof of Proposition \ref{cor:boundsexp}.}
We combine (\ref{esp:N}) with Corollary \ref{cor:cvP} and the definition of $\mathsf t^{(k)}_{p,n}$, to see that 
\begin{eqnarray*}
\mathbb E\left[\mathcal N_{p,n}^{(k')}\big(\mathsf t^{(k)}_{p,n}+c \big)\right] &\underset{n\rightarrow \infty}\sim& n^{k'} \cdot \frac{k'^{k'-2}}{k'!} \cdot P_n^{(k')}\big(1,\mathsf t^{(k)}_{p,n}+c \big) \\
&\underset{n\rightarrow \infty} \sim & n^{k'} \cdot \frac{k'^{k'-2}}{k'!} \cdot n^{1-k'-\frac{k'}{k}} \cdot (\ln(n))^{\frac{k'}{k}-1} \cdot  (kp)^{\frac{k'}{k}-k'} \cdot e^{- k' p c-k' \left(\psi(1/p)+\gamma_{\mathrm E}\right)} \\
&=& O \left( \left(\frac{n}{\ln(n)} \right)^{1-k'/k} \right).
\end{eqnarray*}
$\hfill \square$

\bigskip

\textbf{Proof of Proposition \ref{prop:Poisson_continu}.}
1) We start with the convergence in distribution and in that aim use that the probability generating function of a $\llbracket 0 ; n \rrbracket $--valued random variable can be expressed in terms of its factorial moments, which here gives
$$
\mathbb E\left[x^{k \mathcal N^{(k)}_{p,n}(\mathsf t^{(k)}_{p,n}+c)} \right]=1 + \sum_{j=1}^{n}  \frac{(x-1)^{j}}{j!}  \mathbb E\left[\prod_{i=0}^{j-1}\big(k \mathcal N_{p,n}^{(k)}(\mathsf t^{(k)}_{p,n}+c)-i\big) \right], \quad \forall x \in \mathbb R.
$$
From Lemma \ref{lem:factorial_moments}, we rewrite the sum as follows
\begin{eqnarray*}
&& \sum_{j=1}^{n}  \frac{(x-1)^{j}}{j!} \mathbb E\left[\prod_{i=0}^{j-1}\big(k \mathcal N_{p,n}^{(k)}(\mathsf t^{(k)}_{p,n}+c)-i\big) \right] \\
& =& \sum_{j=1}^{n} \frac{ (x-1)^{j}}{j!} \sum_{\substack{(n_1,\ldots,n_\ell) \in \mathcal P_j  \\ n_i \leq k \: \forall i, \; \ell k\leq n}} \frac{n!}{(n-k\ell)!} \binom{j}{n_1,\ldots,n_\ell} \cdot \prod_{i=1}^j \frac{1}{m_i!} \cdot \prod_{i=1}^\ell \frac{1}{(k-n_i)!} \cdot (k^{k-2})^\ell P^{(k)}_n(\ell,\mathsf t^{(k)}_{p,n}+c) \\
&=& \sum_{\ell=1}^{\lfloor n/k\rfloor} \frac{1}{\ell !} \frac{1}{(k!)^{\ell}}\sum_{n_1=1}^k\ldots  \sum_{n_{\ell}=1}^k \prod_{i=1}^{\ell}\binom{k}{n_i} \cdot \frac{n!}{(n-k\ell)!}  (k^{k-2})^{\ell} P^{(k)}_n(\ell,\mathsf t^{(k)}_{p,n}+c) \cdot (x-1)^{\sum_{i=1}^{\ell}n_i} \\
&=& \sum_{\ell=1}^{\lfloor n/k\rfloor} \frac{1}{\ell !} \frac{1}{(k!)^{\ell}}  \cdot \frac{n!}{(n-k\ell)!}  (k^{k-2})^{\ell} P^{(k)}_n(\ell,\mathsf t^{(k)}_{p,n}+c) \left( \sum_{m=1}^k \binom{k}{m} (x-1)^{m}\right)^\ell \\
&=&  \sum_{\ell=1}^{\lfloor n/k\rfloor} \frac{1}{\ell !} \frac{1}{(k!)^{\ell}}  \cdot \frac{n!}{(n-k\ell)!}  (k^{k-2})^{\ell} P^{(k)}_n(\ell,\mathsf t^{(k)}_{p,n}+c) (x^k-1)^\ell.
\end{eqnarray*}
Next, from Corollary \ref{cor:cvP}, for each fixed $\ell \in \mathbb N$,
$$
\frac{n!}{(n-k\ell)!}  P^{(k)}_n(\ell,\mathsf t^{(k)}_{p,n}+c)  \; \underset{n \rightarrow \infty} \longrightarrow \;  \left( e^{-kpc} \cdot e^{-k\left(\psi(1/p)+\gamma_{\mathrm E} \right)}\right)^{\ell},
$$
which leads to 
$$
\mathbb E\left[x^{k \mathcal N^{(k)}_{p,n}(\mathsf t^{(k)}_{p,n}+c)} \right] \; \underset{n \rightarrow \infty} \longrightarrow \; e^{\frac{k^{k-2}}{k!} e^{-kpc} \cdot e^{-k\left(\psi(1/p)+\gamma_{\mathrm E} \right)}(x^k-1)}, 
$$
since we can use the Dominated Convergence Theorem thanks to Corollary \ref{cor:TCD1} (taking there, e.g., $\varepsilon$ such that $\varepsilon (x^k-1) \frac{k^{k-2}}{k!} \leq 1/2$).

Consequently, the probability generating function of $\mathcal N^{(k)}_{p,n}(\mathsf t^{(k)}_{p,n}+c)$ has the following asymptotic behavior
$$
\mathbb E\left[x^{\mathcal N^{(k)}_{p,n}(\mathsf t^{(k)}_{p,n}+c)} \right] \; \underset{n \rightarrow \infty} \longrightarrow \;  e^{\frac{k^{k-2}}{k!} e^{-kpc} \cdot e^{-k\left(\psi(1/p)+\gamma_{\mathrm E} \right)}(x-1)}, \quad \forall x \in \mathbb R,
$$
and we recognize in the right-hand side the probability generating function of the Poisson distribution with parameter $\frac{k^{k-2}}{k!} e^{-kpc} \cdot e^{-k\left(\psi(1/p)+\gamma_{\mathrm E} \right)}$.

\smallskip

2) Using the same arguments, we see that for each fixed $j \in \mathbb N$,
\begin{eqnarray*}
 \mathbb E\left[\prod_{i=0}^{j-1}\big(k \mathcal N_{p,n}^{(k)}(\mathsf t^{(k)}_{p,n}+c)-i\big) \right] &\leq& j! \sum_{\ell =1}^{\lfloor n/k\rfloor} \frac{ (k^{k-2})^{\ell}}{\ell !} \frac{1}{(k!)^{\ell}}  \cdot \frac{n!}{(n-k\ell)!}  P^{(k)}_n(\ell,\mathsf t^{(k)}_{p,n}+c) (2^{k}-1)^\ell,
\end{eqnarray*}
which, thanks to Corollary \ref{cor:cvP} and Corollary \ref{cor:TCD1}, is bounded from above by a finite number independent of $n$. This holds for all $j \in \mathbb N$, consequently any positive moment of $\mathcal N_{p,n}^{(k)}(\mathsf t^{(k)}_{p,n}+c)$ is bounded from above independently of $n$. Together with the convergence in distribution of $\mathcal N_{p,n}^{(k)}(\mathsf t^{(k)}_{p,n}+c)$ to the  Poisson distribution with parameter $k^{k-2}e^{-kpc} \cdot e^{-k\left(\psi(1/p)+\gamma_{\mathrm E} \right)}/k!$, this is sufficient to get the convergence of every positive moment of $\mathcal N_{p,n}^{(k)}(\mathsf t^{(k)}_{p,n}+c)$ to the corresponding moment of the limit Poisson distribution.
$\hfill \square$

\bigskip

Last, the expression (\ref{esp:N}) of the expectation of $\mathcal N_{p,n}^{(i)}(t)$ for $i \geq k+1$, together with Corollary \ref{cor:unifbound}, give immediately that with high probability, there is no tree of size strictly larger than $k$ at times $\mathsf t^{(k)}_{p,n}+c$, for any $c$:

\begin{corollary} 
\label{cor:treesafterk}
For every $k \in \mathbb N$ and every $c \in \mathbb R$, 
$$\mathbb P\left(\text{there exists a tree of size}\geq k+1 \text{  at time } \mathsf t^{(k)}_{p,n} +c \right) \; \underset{n \rightarrow \infty} \longrightarrow \; 0.$$
\end{corollary}

\textbf{Proof.}
From (\ref{esp:N}),
\begin{eqnarray*}
\mathbb P\left(\text{there exists a tree of size}\geq k+1 \text{  at time } \mathsf t^{(k)}_{p,n} +c \right) &\leq& \sum_{i=k+1} ^n \mathbb P\left(\mathcal N_{p,n}^{(i)}(\mathsf t^{(k)}_{p,n}+c) \geq 1\right) \\
&\leq & \sum_{i=k+1}^n \mathbb E\left[\mathcal N_{p,n}^{(i)}(\mathsf t^{(k)}_{p,n}+c) \right] \\
&\leq& \sum_{i=k+1}^n n^i \frac{i^{i-2}}{i!} P_n^{(i)}(1,\mathsf t^{(k)}_{p,n}+c). 
\end{eqnarray*}
By Stirling's formula, $\frac{i^{i-2}}{i!} \leq C \frac{e^i}{i^2\sqrt i}$ for some finite $C$ and all $i\geq 1$. Together with Corollary \ref{cor:unifbound}, this implies that for $n$ large enough
\begin{eqnarray*}
\mathbb P\left(\text{there exists a tree of size}\geq k+1 \text{  at time } \mathsf t^{(k)}_{p,n} +c \right) 
&\leq& C\sum_{i=k+1}^n  n^i \frac{e^i}{i^2\sqrt i} P_n^{(i)}(1,\mathsf t^{(k)}_{p,n}+c) \\
&\leq & \frac{C}{n^{\frac{1}{2k}}} \sum_{i=k+1}^n  \frac{1}{i^2\sqrt i}
\end{eqnarray*}
which converges to 0 as $n \rightarrow \infty$.
$\hfill \square$

\subsection{Asymptotics of $\mathcal A_{p,n}^{(k)}$, $\mathcal A_{p,n}^{(k+)}$}
\label{sec:extinctiontimes}

We are now ready to study the last times at which there is a tree of size $k \in \mathbb N$ or of size greater or equal to $k$ in the model $\mathcal {F}_{p,n}$:
$$\mathcal A_{p,n}^{(k)}=\sup\left\{t\geq 0:\mathcal N_{p,n}^{(k)}(t) \geq 1\right\},\qquad
\mathcal A_{p,n}^{(k+)}=\sup\left\{t\geq 0:\sum_{i \geq k} \mathcal N_{p,n}^{(i)}(t) \geq 1 \right\}.$$
We recall that $\mathrm{Gu}$ denote a standard G\" umbel distribution. 

\begin{theorem}
\label{prop:tau_k_continu}
For all $k \in \mathbb N$
$$
\mathcal A_{p,n}^{(k)} - \mathsf t^{(k)}_{p,n}  \; \underset{n \rightarrow \infty}{\overset{(\mathrm d)}\longrightarrow} \; \frac{\mathrm{Gu}}{kp}-\frac{\Psi(1/p)+\gamma_{\mathrm E}}{p}+\frac{\ln (k^{k-2}/k!)}{kp}.
$$
\end{theorem}

Since $\mathsf t^{(k)}_{p,n} \ll \mathsf t_{p,n}^{(k-1)}$ as $n \rightarrow \infty$, a consequence of this result is that $\mathbb P\big(\mathcal A_{p,n}^{(k)}<\mathcal A_{p,n}^{(k-1)}<\ldots<\mathcal A_{p,n}^{(1)}\big)\rightarrow 1$  for all $k \geq 2$. This alone is however not sufficient to claim that there is asymptotically no trees of size larger than $k$ after $\mathcal A_{p,n}^{(k)}$, but we can improve it as follows.

\begin{theorem}
\label{prop:A_k_continu}
For all $k \in \mathbb N$, 
$$
\mathbb P\left(\mathcal A^{(k+)}_{p,n} =\mathcal A_{p,n}^{(k)} \right)  \; \underset{n \rightarrow \infty}{\longrightarrow} \; 1.
$$
Consequently,
$$
\mathcal A_{p,n}^{(k+)} - \mathsf t^{(k)}_{p,n}  \; \underset{n \rightarrow \infty}{\overset{(\mathrm d)}\longrightarrow} \; \frac{G}{kp}-\frac{\Psi(1/p)+\gamma_{\mathrm E}}{p}+\frac{\ln (k^{k-2}/k!)}{kp}.
$$
In particular, this gives the asymptotic behavior of the total gelation time, $\mathcal A_{p,n}^{(1+)}=\inf\{t\geq 0:\mathcal G_{p,n}(t)=n\}$.
\end{theorem}

To prove these results, we start by setting some preliminary lemmas.

\subsubsection{Preliminaries}

In the following lemmas, c.c. is used as an abbreviation of \emph{connected component}. 

\begin{lemma}
Fix $\ell_1,\ell_2 \in \mathbb N$. For $n\geq \ell_1+\ell_2$, 
consider $\mathrm t_1,\mathrm t_2$ two trees with vertices in $\{1,\ldots,n\}$ and no common vertices, with respective sizes $\ell_1,\ell_2$. Then for every stopping time $T>0$,
\begin{eqnarray*}
&& \mathbb P\left(\mathrm t_1,\mathrm t_2 \text{ connect during the process to give a tree of size }\ell_1+\ell_2 \; | \; \mathrm t_1,\mathrm t_2 \text{ are c.c. of } \mathcal {F}_{p,n}(T)\right) \\
&\leq& \frac{\ell_1\ell_2}{p(n-1)}.
\end{eqnarray*}
\end{lemma}

\begin{proof}
We let $V(\mathrm t_i)$ denote the set of vertices of $\mathrm t_i$, for $i=1,2$.
We also let $T<s_1<s_2<\ldots$ be the times larger than $T$ at which the PPP governing $\mathcal {F}_{p,n}$ rings, set $s_0=T$, and introduce for $i\geq 1$ the events:
\begin{enumerate}
\item[$\bullet$] $E_{n}(s_i)$=\{at time $s_i$ an edge is added between a vertex of $V(\mathrm t_1)$ and a vertex of $V(\mathrm t_2)$\} 
\item[$\bullet$] $\tilde E_{n}(s_i)$=\{at time $s_i$ an edge stemming from either a vertex of $V(\mathrm t_1)$ or a vertex of $V(\mathrm t_2)$ is added in the process\}.
\end{enumerate}
Note that
\begin{eqnarray*}
 &&\mathbb P\left(\mathrm t_1,\mathrm t_2 \text{ connect during the process to give a tree of size }\ell_1+\ell_2 \; | \; \mathrm t_1,\mathrm t_2 \text{ are c.c. of } \mathcal {F}_{p,n}(T)\right) \\
 &=& \mathbb P\left(\cup_{i=1}^{\infty} E_{n}(s_i) \cap_{j=1}^{i-1} \big(\tilde E_{n}(s_j)\big)^c \; | \; \mathrm t_1,\mathrm t_2 \text{ are c.c. of } \mathcal {F}_{p,n}(T)\right) \\
 &=& \sum_{i=1}^{\infty}  \mathbb P\left(E_{n}(s_i) \cap_{j=1}^{i-1} \big(\tilde E_{n}(s_j)\big)^c \; | \; \mathrm t_1,\mathrm t_2 \text{ are c.c. of } \mathcal {F}_{p,n}(T)\right).
\end{eqnarray*}
The dynamic of the process $\mathcal {F}_{p,n}$ implies that when  $\mathrm t_1,\mathrm t_2 \text{ are c.c. of } \mathcal {F}_{p,n}(s_{i-1})$, for all $i\geq 1$:
$$
 \mathbb P\left(E_{n}(s_i) \; | \; \mathrm t_1,\mathrm t_2 \text{ are c.c. of } \mathcal {F}_{p,n}(s_{i-1}) \right)= \frac{2\ell_1\ell_2}{n(n-1)}
$$ 
and
\begin{eqnarray*}
&& \mathbb P\left(\tilde E_{n}(s_i) \; | \; \mathrm t_1,\mathrm t_2 \text{ are c.c. of } \mathcal {F}_{p,n}(s_{i-1}), \mathcal G_{p,n}(s_{i-1}) \right) \\
 &=& \frac{2}{n(n-1)} \cdot \big(\ell_1\left(n-\mathcal G_{p,n}(s_{i-1})-1\right)+\ell_2\left(n-\mathcal G_{p,n}(s_{i-1})-1-\ell_1\right)+(\ell_1+\ell_2)p \mathcal G_{p,n}(s_{i-1}) \big)  \\
 &=& \frac{2}{n(n-1)} \cdot \big((\ell_1+\ell_2)(n-(1-p)\mathcal G_{p,n}(s_{i-1})-1 )-\ell_1\ell_2 \big) \\
 &\geq& \frac{2}{n(n-1)} \cdot \big( (\ell_1+\ell_2)(pn+(1-p)(\ell_1+\ell_2)-1)-\ell_1\ell_2\big),
\end{eqnarray*}
where we have used in the last line that $\mathcal G_{p,n}(s_{i-1})\leq n-(\ell_1+\ell_2)$ when $\mathrm t_1,\mathrm t_2 \text{ are c.c. of } \mathcal {F}_{p,n}(s_{i-1})$. This leads to
$$
 \mathbb P\left(\tilde E_{n}(s_i) \; | \;  \mathrm t_1,\mathrm t_2 \text{ are c.c. of } \mathcal {F}_{p,n}(s_{i-1})\right) \geq  \frac{2p}{n}
$$
(to see this note that for $n\geq \ell_1+\ell_2+1$, the function $p \mapsto (\ell_1+\ell_2)(pn+(1-p)(\ell_1+\ell_2)-1)-\ell_1\ell_2 -(n-1)p$ is increasing in $p$ and positive for $p=0$ ; whereas for $n=\ell_1+\ell_2$ it is decreasing in $p$ and positive for $p=1$). We then use these bounds to get 
\begin{eqnarray*}
 &&\sum_{i=1}^{\infty}  \mathbb P\left(E_{n}(s_i) \cap_{j=1}^{i-1} \big(\tilde E_{n}(s_j)\big)^c \; | \; \mathrm t_1,\mathrm t_2 \text{ are c.c. of } \mathcal {F}_{p,n}(T)\right) \\
 &\leq & \sum_{i=1}^{\infty}  \frac{2\ell_1\ell_2}{n(n-1)} \cdot \left(1- \frac{2p}{n}\right)^{i-1} \\
 &=& \frac{\ell_1\ell_2}{p(n-1)}.
\end{eqnarray*}
\end{proof}

In fact, we will only need that the order of magnitude of this probability is $O(1/n)$. More generally, for every fixed $i\geq 2 $ and $m \geq 2$, we have 

\begin{lemma}
\label{lem:treei}
Let $(\ell_1,\ldots,\ell_m)\in \mathbb N^m$, with $\sum_{j=1}^m \ell_j=i$. For $n\geq i$,
consider $\mathrm t_1,\ldots,\mathrm t_m$ some trees with vertices in $\{1,\ldots,n\}$ and no common vertices, with respective sizes $\ell_1,\ldots,\ell_m \in \mathbb N$. Then for every time $t>0$,
\begin{eqnarray*}
&& \mathbb P\left(\mathrm t_1,\ldots,\mathrm t_m \text{ connect during the process to give a tree of size }i \;  | \; \mathrm t_1,\ldots \mathrm t_m \text{ are c.c. of } \mathcal {F}_{p,n}(t)\right) \\
&=& O \left( \frac{1}{n^{m-1}}\right). \
\end{eqnarray*}
\end{lemma}

\begin{proof}
Since only two trees can connect at a time, this is easily proved by induction on $m$, using the previous lemma.
\end{proof}

Consequently,

\begin{lemma} 
\label{lem:formation}
When $k\geq 2$, for any $c \in \mathbb R$,
$$
\mathbb P \left(\text{a tree of size $\geq k$ is formed after time }\mathsf t^{(k)}_{p,n} +c \right)  \; \underset{n \rightarrow \infty}{\longrightarrow} \; 0.
$$
\end{lemma}

\begin{proof}
$\bullet$ We start by proving that for any $i \geq k$,
\begin{equation}
\label{cv:uni}
\mathbb P \left(\text{a tree of size $i$ is formed after }\mathsf t^{(k)}_{p,n} +c \right)  \; \underset{n \rightarrow \infty}{\longrightarrow} \; 0.
\end{equation}
Indeed, for $i \geq k$, 
\begin{eqnarray*}
&& \mathbb P \left(\text{a tree of size $i$ is formed after }\mathsf t^{(k)}_{p,n} +c \right) \\
&\leq & \sum_{m=2}^i \; \sum_{\substack{\ell_1 \geq \ldots \geq \ell_m \\\sum_{j=1}^m \ell_j=i}} \mathbb P \left(\text{a tree of size $i$ is formed after time }\mathsf t^{(k)}_{p,n} +c \text{ from $m$ trees present at time } \right.\\
&&  \left.  \hspace{3cm}  \mathsf t^{(k)}_{p,n} +c \text{ with respective sizes }\ell_1,\ldots,\ell_m\right) \\
&=& \sum_{m=2}^i \; \sum_{\substack{\ell_1 \geq \ldots \geq \ell_m \\\sum_{j=1}^m \ell_j=i}} \mathbb P \left(\text{a tree of size $i$ is formed after time }\mathsf t^{(k)}_{p,n} +c \text{ from $m$ trees present at time } \right.\\
&&  \Big.   \mathsf t^{(k)}_{p,n} +c \text{ with respective sizes }\ell_1,\ldots,\ell_m; \mathcal N_{p,n}^{(\ell_j)}(\mathsf t^{(k)}_{p,n} +c) \leq (\ln(n))^{\frac{1}{2}}\left(\frac{n}{\ln(n)}\right)^{1-\frac{\ell_j}{k}}, \forall 1\leq j \leq m\Big) \\
&+& o(1) 
\end{eqnarray*}
where the $o(1)$ (relative to $n\rightarrow \infty$) is a consequence of Proposition \ref{cor:boundsexp}. We then conclude with Lemma \ref{lem:treei} and again Proposition \ref{cor:boundsexp} which imply that when $\sum_{j=1}^m \ell_j=i$,
\begin{eqnarray*}
&& \mathbb P \left(\text{a tree of size $i$ is formed after time }\mathsf t^{(k)}_{p,n} +c \text{ from $m$ trees present at time } \mathsf t^{(k)}_{p,n} +c \right. \\ 
&& \Big. \hspace{0.2cm}\text{ with respective sizes }\ell_1,\ldots,\ell_m;  \mathcal N_{p,n}^{(\ell_j)}(\mathsf t^{(k)}_{p,n} +c) \leq (\ln(n))^{\frac{1}{2}}\left(\frac{n}{\ln(n)}\right)^{1-\frac{\ell_j}{k}}, \forall 1\leq j \leq m\Big) \\
&\leq& O\left(\frac{1}{n^{m-1}} \right) \cdot \left(\prod_{j=1}^m  (\ln(n))^{\frac{1}{2}}\left(\frac{n}{\ln(n)}\right)^{1-\frac{\ell_j}{k}} \right) \\
&=& O \left((\ln(n))^{\frac{i}{k}-\frac{m}{2}} \cdot n^{1-\frac{i}{k}} \right) 
\end{eqnarray*}
which converges to 0 as soon as $i\geq k$.

$\bullet$ Next, to improve this in
$
\mathbb P \big(\text{a tree of size $\geq k$ is formed after }\mathsf t^{(k)}_{p,n} +c \big)  \; {\rightarrow} \; 0
$ as $n \rightarrow \infty$,
we bound from above this probability by
$$
 \mathbb P\left(\text{a tree of size $\geq 2k+1$ is formed after }\mathsf t^{(k)}_{p,n} +c \right)+\sum_{i=k}^{2k}\mathbb P\left(\text{a tree of size $i$ is formed after }\mathsf t^{(k)}_{p,n} +c \right)
$$
and note that (since only two trees can connect at a time)
\begin{eqnarray*}
&&  \hspace{1cm}  \mathbb P\left(\text{a tree of size $\geq 2k+1$ is formed after }\mathsf t^{(k)}_{p,n} +c \right) ~\leq \\ 
&& \hspace{-0.5cm} \mathbb P\left(\text{there is a tree of size}\geq k+1 \text{  at time } \mathsf t^{(k)}_{p,n} +c \right)
+ \sum_{i=k}^{2k} \mathbb P\left(\text{a tree of size $i$ is formed after }\mathsf t^{(k)}_{p,n} +c \right).
\end{eqnarray*}
So, we have
\begin{eqnarray*}
&&  \hspace{1cm} \mathbb P \left(\text{a tree of size $\geq k$ is formed after }\mathsf t^{(k)}_{p,n} +c \right) ~\leq \\ 
&& \hspace{-0.5cm} 2\sum_{i=k}^{2k} \mathbb P \left(\text{a tree of size $i$ is formed after } \mathsf t^{(k)}_{p,n} +c \right) + \mathbb P\left(\text{there is a tree of size}\geq k+1 \text{  at time } \mathsf t^{(k)}_{p,n} +c \right),
\end{eqnarray*}
and this last sum converges to 0 according to (\ref{cv:uni}) and Corollary \ref{cor:treesafterk}.
\end{proof}

\subsubsection{Proof of Theorem \ref{prop:tau_k_continu}}

Since $\mathcal N_{p,n}^{(k)}(t)$ is the number of trees of size $k$ present at time $t$, we have, for any $c \in \mathbb R$, 
\begin{eqnarray*}
\mathbb P\left(\mathcal N_{p,n}^{(k)}(\mathsf t^{(k)}_{p,n}+c) \geq 1 \right) &\leq& \mathbb P\left(\mathcal A_{p,n}^{(k)} > \mathsf t^{(k)}_{p,n} +c \right) \\
&\leq& \mathbb P\left(\mathcal N_{p,n}^{(k)}(\mathsf t^{(k)}_{p,n}+c) \geq 1\right) + \mathbb P\left(\text{a tree of size $k$ is formed after time }\mathsf t^{(k)}_{p,n}+c\right).
\end{eqnarray*}
(When $k=1$, $\mathbb P\big(\mathcal A_{p,n}^{(1)} > \mathsf t^{(1)}_{p,n} +c \big)= \mathbb P\big(\mathcal N_{p,n}^{(1)}(\mathsf t^{(1)}_{p,n}+c)\geq 1\big)$.)
According to Lemma \ref{lem:formation}, when $k\geq 2$, the last probability converges to 0 as $n \rightarrow \infty$. Consequently, by Proposition \ref{prop:Poisson_continu}
\begin{eqnarray*}
 \mathbb P\left(\mathcal A_{p,n}^{(k)} > \mathsf t^{(k)}_{p,n} +c \right) &\underset{n \rightarrow \infty}\sim& \mathbb P\left(\mathcal N_{p,n}^{(k)}(\mathsf t^{(k)}_{p,n}+c) \geq 1 \right) \\
&\underset{n \rightarrow \infty}\rightarrow &  1-e^{-\frac{k^{k-2}e^{-kpc} e^{-k\left(\psi(1/p)+\gamma_{\mathrm E}\right)}}{k!}}
\end{eqnarray*}
and therefore $ \mathbb P\left(kp(\mathcal A_{p,n}^{(k)} - \mathsf t^{(k)}_{p,n} )+k\left(\psi(1/p)+\gamma_{\mathrm E}\right)-\ln (k^{k-2}/k!) >x\right) \underset{n \rightarrow \infty}\rightarrow 1-e^{-e^{-x}}$, for all $x \in \mathbb R$.

\subsubsection{Proof of Theorem \ref{prop:A_k_continu}}

Since $\mathcal A^{(k+)}_{p,n} \geq \mathcal A_{p,n}^{(k)}$, we just need to show that $\mathbb P\big(\mathcal A^{(k+)}_{p,n} > \mathcal A_{p,n}^{(k)}\big) \rightarrow 0$ as $n \rightarrow \infty$. 
Fix $\varepsilon>0$ and let $c_{\varepsilon} \in \mathbb R$ be sufficiently small so that $\mathbb P\big(\mathcal A_{p,n}^{(k)}<\mathsf t^{(k)}_{p,n}+c_{\varepsilon}  \big) \leq \varepsilon$ for all $n$ large enough (such a  $c_{\varepsilon}$ exists by Theorem \ref{prop:tau_k_continu}).
Splitting the probability $\mathbb P\big(\mathcal A^{(k+)}_{p,n} > \mathcal A_{p,n}^{(k)}\big)$ according to whether $\mathcal A_{p,n}^{(k)}<\mathsf t^{(k)}_{p,n}+c_{\varepsilon} $ or not, we get the upper bound 
\begin{eqnarray*}
\mathbb P\left(\mathcal A^{(k+)}_{p,n} > \mathcal A_{p,n}^{(k)}\right) &\leq&   \varepsilon+ \mathbb P\left(\mathcal A^{(k+)}_{p,n} > \mathcal A_{p,n}^{(k)}  \geq \mathsf t^{(k)}_{p,n}+c_{\varepsilon} \right)\\
&\leq&  \varepsilon+ \mathbb P\left(\text{there exists a tree of size $\geq k+1$ at time }\mathsf t^{(k)}_{p,n} +c_{\varepsilon}\right) \\
&& \hspace{0.25cm} + \;  \mathbb P\left( \text{a tree of size $\geq k+1$ is formed after time }\mathsf t^{(k)}_{p,n} +c_{\varepsilon} \right).
\end{eqnarray*}
By Corollary \ref{cor:treesafterk} and Lemma \ref{lem:formation}, the two latest probabilities converge to 0 as $n \rightarrow \infty$. Since this holds for every $\varepsilon>0$, we are done.

\subsection{De-Poissonization}
\label{sec:dePoisson}

Starting from the discrete model $\mathrm{F}_{p,n}$ and a standard, independent, Poisson process $\mathcal N$, we work here with the version $\mathrm{F}_{p,n}(\mathcal N((n-1) \cdot /2))$ of the continuous model. All straight notations will refer to the discrete model, while curved notations will refer to the continuous model. 

To show that Theorem \ref{prop:A_k_continu} induces Theorem \ref{thm:absorption} and that Proposition \ref{prop:Poisson_continu} induces Proposition \ref{prop:arbreskdiscrets}, we use the bound
\begin{equation}
\label{bound:Poisson}
\mathbb P\left(|Y_{\mathcal P(\lambda)}-\lambda| \geq a \right)~\leq~\frac{\mathbb E[|Y_{\mathcal P(\lambda)}-\lambda|^3]}{a^3}=\frac{\lambda}{a^3}
\end{equation}
for $a>0$, where $Y_{\mathcal P(\lambda)}$ denotes a Poisson random variable with mean $\lambda>0$. 

\bigskip

\textbf{Proof of Theorem \ref{thm:absorption}.} Noticing that
$$
\mathcal A_{p,n}^{(k+)}=\inf \left\{t\geq 0:\mathcal N\left(\left(\frac{n-1}{2}\right)t\right) \geq A_{p,n}^{(k+)}+1 \right\},
$$
we have for all fixed $x \in \mathbb R$ and $\varepsilon>0$,
\begin{eqnarray*}
&& \mathbb P\left( \mathcal A_{p,n}^{(k+)} - \mathsf t^{(k)}_{p,n} \leq 2x-\varepsilon\right) -\mathbb P\left(\mathcal N\left(\left(\frac{n-1}{2}\right)\left(\mathsf t^{(k)}_{p,n}+2x-\varepsilon \right) \right) > n \left(\frac{\mathsf t^{(k)}_{p,n}}{2}+x \right)+1\right)\\
&\leq& \mathbb P\left(\frac{A_{p,n}^{(k+)}}{n} -\frac{\mathsf t^{(k)}_{p,n}}{2} \leq x\right) \\
&\leq& \mathbb P\left(\mathcal A_{p,n}^{(k+)} - \mathsf t^{(k)}_{p,n}\leq 2x+\varepsilon \right)+ \mathbb P\left(\mathcal N\left(\left(\frac{n-1}{2}\right)\left(\mathsf t^{(k)}_{p,n}+2x+\varepsilon \right) \right) < n \left(\frac{\mathsf t^{(k)}_{p,n}}{2}+x \right)+1\right).
\end{eqnarray*}
Together with (\ref{bound:Poisson}) and the limit in distribution of  $\mathcal A_{p,n}^{(k+)} - \mathsf t^{(k)}_{p,n}$ from Theorem \ref{prop:A_k_continu} towards an absolutely continuous law, this yields the expected limit in distribution of ${A_{p,n}^{(k+)}}/{n} -{\mathsf t^{(k)}_{p,n}}/{2}$.

Regarding the relation between $A_{p,n}^{(k+)}$ and $A_{p,n}^{(k)}$, we use that $\big\{A_{p,n}^{(k+)}=A_{p,n}^{(k)}\big\}=\big\{\mathcal A_{p,n}^{(k+)}=\mathcal A_{p,n}^{(k)}\big\}$ to get, again with the help of Theorem \ref{prop:A_k_continu}, the full statement of Theorem \ref{thm:absorption}. 
$\hfill \square$

\bigskip

\textbf{Proof of Proposition \ref{prop:arbreskdiscrets}.}
Fix $k \in \mathbb N$, $c\in \mathbb R$ and $\varepsilon>0$. 
Let $E_{n}^{(1)}$ be the event "no tree of size $k$ if formed after time $\mathsf t^{(k)}_{p,n}+c-\varepsilon$" in the continuous model $\mathcal {F}_{p,n}$ and $$E_{n}^{(2)}:=\left\{\frac{n}{2}(\mathsf t^{(k)}_{p,n}+c) \in \left[\mathcal N\left(\left(\frac{n-1}{2}\right) \cdot (\mathsf t^{(k)}_{p,n}+c-\varepsilon)\right), \mathcal N\left(\left(\frac{n-1}{2}\right)\cdot(\mathsf t^{(k)}_{p,n}+c+\varepsilon)\right)\right] \right\}.$$
By Lemma \ref{lem:formation} and the bound (\ref{bound:Poisson}), $\mathbb P\big(E_{n}^{(1)} \cap E_{n}^{(2)}\big) \rightarrow 1$ as $n \rightarrow \infty$. Then we use that 
$$
\mathcal N^{(k)}_{p,n}(t)=N_{p,n}^{(k)}\left(\mathcal N \left(\frac{n-1}{2} \cdot t \right) \right), \quad \forall t \geq 0,
$$
to get for $i \in \mathbb Z_+$, 
\begin{eqnarray*}
\mathbb P\left(N_{p,n}^{(k)}\left(\left\lfloor \frac{n}{2} \cdot (\mathsf t^{(k)}_{p,n}+c)\right\rfloor\right) \leq i \right)&=&\mathbb P\left(N_{p,n}^{(k)}\left(\left\lfloor \frac{n}{2} \cdot (\mathsf t^{(k)}_{p,n}+c)\right\rfloor\right) \leq i, E_{n}^{(1)} \cap E_{n}^{(2)} \right) +o(1) \\
&\leq & \mathbb P\left(\mathcal N_{p,n}^{(k)}\left(\mathsf t^{(k)}_{p,n}+c+\varepsilon \right) \leq i \right) +o(1).
\end{eqnarray*}
This leads, with Proposition \ref{prop:Poisson_continu}, to 
$$
\limsup_{n \rightarrow \infty} \mathbb P\left(N_{p,n}^{(k)}\left(\left\lfloor \frac{n}{2} \cdot (\mathsf t^{(k)}_{p,n}+c)\right\rfloor\right) \leq i \right) \leq \mathbb P\left(Y_{\mathcal P\left(k^{k-2}e^{-kp(c+\varepsilon)}e^{-k(\psi(1/p)+\gamma_E)}/k!\right)} \leq i \right)
$$
where we still use the notation $Y_{\mathcal P(\lambda)}$ for a Poisson random variable with mean $\lambda$.
Similarly,
$$
\liminf_{n \rightarrow \infty} \mathbb P\left(N_{p,n}^{(k)}\left(\left\lfloor \frac{n}{2} \cdot (\mathsf t^{(k)}_{p,n}+c)\right\rfloor\right) \leq i \right) \geq \mathbb P\left(Y_{\mathcal P\left(k^{k-2}e^{-kp(c-\varepsilon)}e^{-k(\psi(1/p)+\gamma_E)}/k!\right)} \leq i \right).
$$
We get the expected result by letting $\varepsilon \rightarrow 0$.
$\hfill \square$

\section{Concluding remarks and open questions}
\label{sec:open}

We end this paper with a few remarks and related questions on the $p$-frozen model.

$\bullet$ \textbf{The case $p=0$,} where the evolution of unicycle components is stopped as soon as they are created, is different in nature from the cases $p\in (0,1]$. While Theorem \ref{thm:fluid limit} should also hold when $p=0$ with a function $g_0$ defined as in Definition \ref{def:gel_function} and the related function $d_0$ (\ref{def:d_p}), that is for $t\geq 1/2$
\begin{equation*}
	g_0(t)=1-\frac{1}{2t} \qquad \text{and} \qquad d_0(t)=t-1+\frac{1}{4t},
\end{equation*}
its proof requires a partly different approach. Mainly because the forest part of the graph $\mathrm F_{0,n}(m)$ when $m \gg n/2$ is no more subcritical (as it is when $p \in (0,1]$) but critical. The approach of Section \ref{sec:mise_en_place} needs therefore to be adapted, but we note that the results of Section \ref{section:forests} are still valid. Regarding the total gelation time and the last times at which there are trees of size $k$, $k\geq1$, one expect an asymptotic behavior in $n^2$ -- instead of $n\ln(n)$ when $p \in (0,1]$. Although several intermediate results such as Lemma \ref{lem:cvPnk} and Lemma \ref{lem:factorial_moments} remain valid when $p=0$ and points to this $n^2$ order, the difficulty to implement precisely the behavior of the total gelation time when $p=0$ lies in the presence of trees of all sizes in its vicinity (unlike the case  $p \in (0,1]$ where there are only isolated vertices). These questions will be considered in a future work.

$\bullet$ \textbf{Fluid limit of unicycle components.} When $p \in (0,1]$ the unicycle components continue to grow after their formation, according to a dynamic which is asymptotically similar to the evolution of the gel: they attract new trees with a weight proportional to their size. This is a reinforcement process, see e.g. Pemantle's survey \cite{Pemantle07} on that topic. One could then expect that the fluid limit of a unicycle component after its formation is the same as that of the gel shifted in time, up to a multiplicative random constant to determine.    

$\bullet$ \textbf{Asymptotic distribution of unicycle components at the gelation time.} To complete the result on the asymptotic behavior of the total gelation time $A_{p,n}$ obtained in Theorem \ref{thm:absorption}, it would be very interesting to determine the asymptotic behavior of the number of unicycle components present at that time $A_{p,n}$, as well as of the vector of their relative sizes. In this direction, Krapivsky \cite{krapisvky24} conjectured that the number $U_{p,n}$ of unicycle components at  time $A_{p,n}$ verifies
\begin{equation*}
	\frac{\mathbb E\left[U_{p,n}\right]}{\ln(n)} ~ \underset{n \rightarrow \infty}{\longrightarrow}~  \frac{1}{6}\left(1+\frac{1}{p}\right).
\end{equation*} 
And when $p=1/2$, the question is solved thanks to Proposition 4 of \cite{ContatCurien23}, which implies that the distribution of the final partition of unicycle components at the gelation time is the same as that of a random mapping. This, in the limit, gives a Poisson-Dirichlet distribution with parameter $1/2$ for the relative sizes of unicycle components ranked in decreasing order, see e.g. \cite{AldousexchSF}.

\section*{Acknowledgements}
\addcontentsline{toc}{section}{Acknowledgements}
We thank Nicolas Curien for helpful discussions on the fluid limit part of this work and for pointing out the connection between the final partition of unicycle components when $p=1/2$ and random mappings. We also thank Pavel Krapivsky for presenting his work \cite{krapisvky24} to us during a visit to Paris.

\appendix 
\section{Appendix}
\label{sec:app}

\subsection{The Borel-Tanner distribution}
\label{app:BT}

The Borel distribution and its generalization the Borel-Tanner distribution \cite{Borel42},\cite{Tanner61} were initially introduced for models in queueing theory and relatively branching processes, and are also used since then for applications in real-word phenomena. We gather here some of their basic properties and highlight some consequences we shall need throughout the paper.

\begin{definition} A random variable $B$  follows a \emph{Borel distribution} with parameter $\theta \in (0,1]$ if it is $\mathbb N-$valued and
$$
\mathbb P(B=k)=\frac{k^{k-2}}{(k-1)!}\cdot \theta^{k-1} e^{-\theta k}, \qquad \forall k \in \mathbb N.
$$
For $r \in \mathbb N$, a random variable $T_r$  follows a \emph{Borel-Tanner distribution} with parameter $\theta \in (0,1]$ if it takes its values in $\{r,r+1,r+2, \ldots\}$ and
$$
\mathbb P(T_r=k)=\frac{r}{(k-r)!}\cdot k^{k-r-1}\theta^{k-r} e^{-\theta k}, \qquad \forall k\geq r.
$$
\end{definition}

From our random trees perspective, the Borel distribution with parameter $\theta$ is the distribution of the total progeny of a Galton-Watson tree with Poisson offspring distribution with mean $\theta$, and the Borel-Tanner distribution with parameters $(r,\theta)$  is the distribution of the total progeny of a forest composed by $r$ independent Galton-Watson trees with Poisson offspring distribution with mean $\theta$. In particular, note that if $B,B'$ are independent random variables, both following a Borel distribution with parameter $\theta$, the identity  $\mathbb P(B+B'=k)=\mathbb P(T_2=k)=\sum_{i=1}^{k-1}\mathbb P(B=i)\mathbb P(B'=k-i)$ leads for all $k\geq 2$ to:
\begin{equation}
\label{eq:ijk}
\frac{2 k^{k-3}}{(k-2)!}=\sum_{i=1}^{k-1} \frac{i^{i-2}}{(i-1)!} \cdot \frac{(k-i)^{k-i-2}}{(k-i-1)!}.
\end{equation}

\bigskip

\textbf{Mean and approximation.} When $\theta \in (0,1)$, the expectation of $B$ is finite:
\begin{equation}
\label{lm:Borel:law}
\mathbb E[B]=\sum_{k=1}^{\infty}\frac{k^k}{k!} \cdot \theta^{k-1}\mathrm{e}^{-\theta k}=\frac{1}{1-\theta}.
\end{equation}
We will need in Section \ref{sec:mise_en_place} some estimates on this sum. For $\theta \in [0,1/2)$ let $S(\theta)={1}/{(1-\theta)}$, and for $\theta \in [0,1/2]$ and $N\in \mathbb N$, 
\begin{equation}
		S_N(\theta)=\sum_{k= 1}^N\frac{k^k}{k!}\cdot \theta^{k-1}\mathrm{e}^{-\theta k}. \label{def:S_N}
\end{equation}

\medskip

\begin{lemma}\label{lm:partial_sum}
\begin{enumerate}[topsep=0cm, itemsep=0cm]
\item[\emph{1)}] The sum $S_N$ converges to $S$ uniformly on all compact subsets of $\left[0,1/2\right)$, and $S_N\left(1\right)\to \infty$ as $N \rightarrow \infty$.
\item[\emph{2)}] For any $A>0$, there exists $N_0\geq 1$ and $\delta>0$ such that for every $N\geq N_0$ and $\theta \in \left[1-\delta,1\right]$, $S_N(\theta)\geq A.$
\end{enumerate}
\end{lemma}

\begin{proof} Point 1) is obvious. For 2), note that there exists $N_0\geq 1$ such that $S_{N_0}(1)\geq A.$ Then use that $S_N$ is non-increasing on $\left[1-1/N,1\right]$. \end{proof}	

\bigskip

\textbf{Connexion with the measures $\mu_x$. }
Recall from Section \ref{sec:URF} the definition of $\mu_x$, for $x\in(0,e^{-1}]$, by
$$
\mu_x(k)=\frac{k^{k-2}}{k!} \cdot \frac{x^k}{T(x)}, ~ \forall k \geq 1, \quad \text{with } \quad T(x)=\sum_{k\geq 1} \frac{k^{k-2}}{k!} x^k,
$$
and note that the Borel distribution with parameter $\theta \in (0,1]$ is the size-biasing of $\mu_{\theta e^{-\theta}}$. This remark leads to the following (well-known in the theory of uniform random forests) points.

\medskip

\begin{lemma}\label{lm:lambert_fct}
	If $x=\theta\mathrm{e}^{-\theta}$ with $\theta\in (0,1]$ then
	 $$T(x)=\theta\left(1-\theta/2\right) \quad, \quad \sum_{k=1}^{\infty} k\mu_x(k)=\frac{2}{2-\theta}, \quad \text{and} \quad \mathrm{Var}_{\mu_x}=\frac{2\theta}{(1-\theta)(2-\theta)^2}$$
	\emph{(}the variance is infinite when $\theta=1$\emph{)}. 
\end{lemma} 

\begin{proof}
By definition of the Borel distribution and its expectation,
$$\sum_{k\geq 1}\frac{k^{k-1}}{k!}\cdot \theta^{k}\mathrm{e}^{-k\theta}=\theta,  \qquad \sum_{k\geq 1}\frac{k^{k}}{k!}\cdot \theta^{k}\mathrm{e}^{-k\theta}=\frac{\theta}{1-\theta}, \quad  \forall \theta \in [0,1].$$
Setting $x(\theta)=\theta\mathrm{e}^{-\theta}$, $~T(0)=0$ and differentiating the function
$
\theta \in [0,1] \mapsto T(x(\theta)),
$
we see that 
\begin{eqnarray*}
\partial_{\theta} T(x(\theta))=\sum_{k\geq 1}\frac{k^{k-1}}{k!}\theta^{k-1}\mathrm{e}^{-k\theta}-\sum_{k\geq 1}\frac{k^{k-1}}{k!}\theta^{k}\mathrm{e}^{-k\theta} = 1-\theta.
\end{eqnarray*}
Since $T(x(0))=0$, this indeed gives $T(x(\theta))=\theta-\theta^2/2$ for all $\theta \in (0,1]$ and then 
$$
\sum_{k=1}^{\infty} k\mu_{x(\theta)}(k)=\frac{1}{T(x(\theta))} \cdot \sum_{k\geq 1}\frac{k^{k-1}}{k!}\theta^{k}\mathrm{e}^{-k\theta}=\frac{2}{2-\theta}
$$
$$
\sum_{k=1}^{\infty} k^2\mu_{x(\theta)}(k)=\frac{1}{T(x(\theta))} \cdot \sum_{k\geq 1}\frac{k^{k}}{k!}\theta^{k}\mathrm{e}^{-k\theta}=\frac{2}{(1-\theta)(2-\theta)},
$$
leading to the result.
\end{proof}

\subsection{Wormald's differential equation method}

We give here a version of Wormald's theorem, initially proved in \cite{wormald95} and then deepened in \cite{wormald97,Warnke19}. Fix $k \in \mathbb N$. For $n \in \mathbb N$, let $\big(\mathbf F_n(m),m \in \mathbb Z_+ \big)$ be a filtration and let $Y^{(1)}_n,...,Y^{(k)}_n$ be $\mathbf F_{n}$-adapted discrete-time stochastic processes. Assume that there exists some constant $C_0$ such that $|Y_n^{(l)}(m)| <C_0 n$ almost surely for all $m \in \mathbb Z_+$, $1 \leq l \leq k$, $n \in \mathbb N$. Let then $D$ be a bounded open subset of $\mathbb{R}^{k+1}$ and for $1\leq l\leq k$, $$F_l:D \rightarrow \mathbb R\quad  \text{be a Lipschitz function}.$$ Finally let $H_D(Y^{(1)}_n,...,Y^{(k)}_n)$ be the first time $m \in \mathbb Z_+$ at which $$\left(\frac{m}{n},\frac{Y^{(1)}_n(m)}{n},\ldots,\frac{Y^{(k)}_n(m)}{n}\right)\notin D$$
with the usual convention $\inf{\{\emptyset\}}=\infty$. This  is a stopping time with respect to the filtration $\mathbf F_{n}$.

\begin{theorem}[Theorem 5.1 in \cite{wormald97}, Theorem 2 in \cite{Warnke19}]
\label{thmDEM}
	Assume that $D$ contains the closure of 
	$$\left\{(0,z_1,\ldots,z_k)\in \mathbb{R}^{k+1}:\mathbb{P}\big(Y_n^{(l)}(0)=z_l n, 1\leq l\leq k \big)\neq 0 \text{ for some } n\right\}$$
	and that the two following hypotheses hold for all $n \in \mathbb N$:
	\begin{enumerate}
		\item[$\bullet$] \emph{Boundedness hypothesis.} For some functions $\beta:\mathbb N \rightarrow [1,\infty)$ and $\gamma:\mathbb N \rightarrow [0,1]$ the probability that
		$$\max_{1\leq l\leq k}\left\vert \Delta Y^{(l)}_n(m+1)\right\vert \leq \beta(n),$$ conditional on $\mathbf F_{n}(m)$, is at least $1-\gamma(n)$ when $m<H_D(Y^{(1)}_n,...,Y^{(k)}_n)$.
		\item[$\bullet$] \emph{Trend hypothesis.} For some function $\lambda:\mathbb N \rightarrow \mathbb R_+$ such that $\lambda=o(1)$ as $n \rightarrow \infty$, for all $1 \leq l\leq k$,
		$$\left\vert \mathbb{E}\left[\Delta Y^{(l)}_n(m)\vert \mathbf{F}_{n}(m) \right]-F_l\left(\frac{m}{n},\frac{Y^{(1)}_n(m)}{n},...,\frac{Y^{(k)}_k(m)}{n}\right)\right\vert\leq \lambda(n)$$ when $m<H_D(Y^{(1)}_n,...,Y^{(k)}_n)$.
	\end{enumerate}
	Then:
	\begin{enumerate}[topsep=0cm]
		\item[\emph{(a)}] For $(0,{z}_1,\ldots,{z}_a)\in D$, the system of differential equations
		$$y'_l(t)=F_l(t,y_1,\ldots,y_k),\quad y_l(0)=\hat{z}_l, \quad l=1,\ldots,k $$
		has a unique maximal solution.
		\item[\emph{(b)}] Let $\eta(n) \geq \lambda(n)+C_0 n\gamma(n) $ with $\eta(n)=o(1)$. For a sufficiently large constant $C$, with probability $1-O\left(n\gamma(n)+\frac{\beta(n)}{\eta(n)}\exp\left(-\frac{n\eta^3(n)}{\beta^3(n)}\right)   \right),$
		$$Y^{(l)}_n(m)=n y_l\left(\frac{m}{n}\right)+O\left(\eta(n) n\right) $$
		uniformly in $0\leq m\leq \sigma(n) n $ and $1\leq l \leq k$, where $y_l$ is the solution in \emph{(a)} with $z_l=Y^{(l)}_n(0)/n$, and $\sigma(n)$ is the supremum of the times $t$ to which the solution can be extended before reaching within $\ell^{\infty}$-distance $C\eta(n)$ of the boundary of $D$.
	\end{enumerate}
\end{theorem}

\bibliographystyle{siam}
\bibliography{biblioER_V10}

\end{document}